\documentclass{amsart}
\usepackage{tikz-cd}
\usetikzlibrary{cd}

\usepackage{amsmath,amssymb,amsthm, mathrsfs, mathtools, bm, amsfonts}
\usepackage{mathabx}
\usepackage{amscd,mathtools}
\usepackage{thmtools}
\usepackage[utf8]{inputenc}
\usepackage[english]{babel}
\usepackage{cite}
\usepackage{color}
\usepackage[pagebackref,colorlinks,citecolor=blue,linkcolor=red]{hyperref}
\usepackage[nameinlink]{cleveref}

\renewcommand*{\backrefalt}[4]{\ifcase #1 (Not cited).\or (Cited p.~#2).\else (Cited pp.~#2).\fi} 
\usepackage{float}
\usepackage{tikz}
\usepackage{lmodern}
\usetikzlibrary{decorations.pathmorphing}

\newcommand{\Mod}{{\rm{Mod}}}

\newcommand{\C}{{\mathcal C}}
\newcommand{\B}{{\mathcal B}}
\newcommand{\E}{{\mathcal E}}
\newcommand{\N}{{\mathcal N}}
\newcommand{\T}{{\mathcal T}}
\newcommand{\A}{{\mathcal A}}
\newcommand{\V}{{\mathcal V}}
\newcommand{\cH}{{\mathcal H}}

\newcommand{\Diff}{\mbox{Diff}}

\newcommand{\stab}{{\rm Stab}}

\newcommand{\lk}{{\rm Lk}}
\newcommand{\st}{{\rm St}}
\newcommand{\Sat}{{\rm Sat}}
\newcommand{\AC}{{\mathcal A}{\mathcal C}}
\newcommand{\nbhd}{\mathcal{N}}

\usepackage{multicol}
\tikzset{snake it/.style={decorate, decoration=snake}}
\usetikzlibrary{automata}
\usetikzlibrary{cd}
\usepackage[margin=3cm]{geometry}

\usepackage{comment}
\usepackage{mathtools}

\usepackage{csquotes}

\usepackage{amsthm}

\theoremstyle{plain}
\newtheorem{theorem}{Theorem}[section]
\newtheorem{corollary}[theorem]{Corollary}
\newtheorem{lemma}[theorem]{Lemma}
\newtheorem{proposition}[theorem]{Proposition}

\newtheorem{claim}[theorem]{Claim}

\newtheorem{question}[theorem]{Question}

\theoremstyle{definition}
\newtheorem{definition}[theorem]{Definition}

\newtheorem{remark}[theorem]{Remark}

\newcommand{\abs}[1]{\left\vert#1\right\vert}

\newcommand{\diam}{{\rm{diam}}}
\newcommand{\neck}{{\rm{Neck}}}
\newcommand{\ladder}{{\mathcal{L}}}
\newcommand{\Dladder}{{\mathcal{DL}}}
\newcommand{\half}{{\mathcal{H}}}
\newcommand{\tripod}{{\mathcal{T}}}
\newcommand{\uncov}{{\frak q}}
\newcommand{\Isom}{{\rm Isom}}

\title{Hierarchically Hyperbolic Surface-by-surface Groups}

\author[Dowdall]{Spencer Dowdall}
 \address{Department of Mathematics, Vanderbilt University, Nashville, TN, USA}
 \email{spencer.dowdall@vanderbilt.edu}

 \author[Durham]{Matthew Gentry Durham}
 \address{Department of Mathematics and Statistics, CUNY Hunter College, New York, NY, USA}
 \email{matthew.durham@hunter.edu}

\author[Leininger]{Chris Leininger}
    \address{Department of Mathematics, University of Notre Dame,
Notre Dame, IN, USA}
    \email{cleining@nd.edu}

\author[Russell]{Jacob Russell}
\address{Department of Mathematics and Statistics, Swarthmore College,
Swarthmore, PA, USA}
\email{jrussel2@swarthmore.edu}

\author[Sisto]{Alessandro Sisto}
	\address{Maxwell Institute and Department of Mathematics, Heriot-Watt University,     Edinburgh, UK}
	\email{a.sisto@hw.ac.uk}

\begin{document}
\begin{abstract}
 We give many examples of surface bundles over surfaces whose fundamental groups are hierarchically hyperbolic by showing that extensions of surface groups constructed by the third author and Reid (LR surface groups) are hierarchically hyperbolic. This is facilitated by a new combination theorem for bundles of hyperbolic graphs over a hyperbolic base.
  This result differs from previous combination theorems for graph bundles by being applicable to bundles where the fibers are not properly embedded. We apply our combination theorem to establish the hyperbolicity of certain graph bundles that arise naturally from extensions of parabolically geometrically finite (PGF) subgroups of mapping class groups. These subgroups include LR surface groups, finitely generated Veech groups, free products of multi-twist groups, and many other examples created by a theorem of Udall. When the PGF groups have cyclic peripherals, hyperbolicity of our graph bundle is the key milestone towards showing that the extension groups are hierarchically hyperbolic.
\end{abstract}

\maketitle

\section{Introduction}
The mapping class group $\Mod(S)$ of a closed surface $S$ plays a linchpin function in classifying surface group extensions and surface bundles (see e.g.~\cite{morita,SalterTshishiku}). A surface group extension (or $\pi_1S$--extension) of a group $G$ is a group $\Gamma$ that fits into a short exact sequence $$ 1 \to \pi_1S \to \Gamma \to G \to 1.$$
Any surface bundle gives rise to such an extension via the long exact sequence of the fibration.  When the surface has genus at least 2, the $\pi_1S$--extensions of $G$ are in one-to-one correspondence with homomorphisms of $G$ into $\operatorname{Out}(\pi_1S) \cong \Mod^{\pm}(S)$ (see e.g.~\cite{FM_convex_cocompact}). In the case of injective monodromy, $G$ may be viewed as a subgroup of $\Mod(S)$ and the Birman exact sequence \cite{birman-exact1,birman-exact2} realizes the corresponding $\pi_1S$--extension $\Gamma_G$ as a subgroup of the mapping class group of the surface $\dot S$ obtained by adding a puncture to $S$.

A premier open question on the geometry of surface group extensions is whether or not there exists a surface group $G < \Mod(S)$ so that $\Gamma_G$ is Gromov hyperbolic \cite{FM_convex_cocompact}. This is a group theoretic version of the open question of whether or not there exists a hyperbolic surface bundle over a surface. Work of Farb, Mosher, and Hamenst\"adt established that  $\Gamma_G$ is Gromov hyperbolic  if and only if $G $ is a \emph{convex cocompact} subgroup of $\Mod(S)$ \cite{FM_convex_cocompact,Hamenstadt_extensions_of_surface_groups}; see also \cite{MS:combination}. This is  equivalent to the orbit map of $G$ in the curve graph $\C(S)$ being a quasi-isometric embedding \cite{Kent_Leininger_shadows,Hamenstadt_extensions_of_surface_groups}, and also to $G$ being undistorted and purely pseudo-Anosov in $\Mod(S)$ \cite{BBKL_purely_pA}.

Recently, Kent and the third author made significant progress on this question by producing examples of purely pseudo-Anosov surface subgroups $G < \Mod(S)$ \cite{Kent_Leininger_atoroidal}, and although they conjecture that these surface subgroups are convex cocompact, that conjecture remains open.  Prior to these examples, the closest examples to convex cocompact surface groups were the surface subgroups constructed in \cite{Leininger_Reid_combination} (which we will refer to as LR surface groups), which are purely pseudo-Anosov except for a single conjugacy class of a multitwist and its powers. Since the LR surface groups are not convex cocompact, their extensions cannot be Gromov hyperbolic. The main application of the work in this paper it to prove these extensions have a generalization of Gromov hyperbolicity called \emph{hierarchically hyperbolicity}. These are the first examples of hierarchically hyperbolic surface-by-surface groups with injective monodromy, or, in fact, monodromy with infinite image. 

\begin{theorem}\label{thm:Into_LR_Groups}
    If $G <\Mod(S)$ is a LR surface subgroup, then the extension group $\Gamma_G$ is hierarchically hyperbolic.
\end{theorem}

 The extensive literature of on hierarchically hyperbolic groups produces many corollaries of Theorem \ref{thm:Into_LR_Groups}.  For instance, $\Gamma_G$ is semi-hyperbolic \cite{HHP,DMS1} and coarse median \cite{Bow_MCG, BHS_HHSII}, which imply it has a quadratic Dehn function and a solvable conjugacy problem. Theorem \ref{thm:HHG detailed} gives the precise  hierarchically hyperbolic structure for these $\Gamma_G$. From this, we see that the hierarchically hyperbolic structure is both \emph{short} and \emph{colorable} (see \cite[Section 2.2]{short1} and \cite[Section 9.1]{DMS2} for these definitions). This allows us to conclude that $\Gamma_G$  has many  infinite hyperbolic quotients \cite{short2}, that quasi-isometric embeddings of $\mathbb R^2$ take a particular form \cite{HHS_quasiflats} while $\mathbb R^3$ does not quasi-isometrically embed in $\Gamma_G$ \cite{HHS1}, and that $\Gamma_G$ is asymptotically CAT(0) and has a $\mathcal{Z}$--structure \cite{DMS2}. All these results are new for $\Gamma_G$ as in the theorem.

Our proof of Theorem \ref{thm:Into_LR_Groups} relies on the \emph{parabolically geometrically finite} (PGF) structure of the LR surface groups \cite{Udall}. PGF groups were defined by a subset of the authors as a first step in generalizing convex cocompactness in $\Mod(S)$ to a theory of ``geometric finiteness'' \cite{DDLS_Veech_II}; see Subsection \ref{subsec:geom_finite}. PGF groups are hyperbolic, relative to abelian subgroups generated by multitwists, and are characterized by the orbit map on the curve graph $\C(S)$ being a quasi-isometric embedding with respect to a relatively hyperbolic generating set; see Definition \ref{defn:PGF}.

Inspired by the equivalence between convex cocompactness and the Gromov hyperbolicity of extensions, Mosher speculated that there ought to be a definition of a ``geometrically finite'' subgroup $G <\Mod(S)$ that characterizes some ``generalized hyperbolicity'' of  $\Gamma_G$ \cite{Mosher_Geom_finite}; see also Hamenst\"adt \cite[Problem 5]{Hamenstadt_Geom_finite}.
As the first step to such a theory, a subset of the authors showed that  $\Gamma_G$ is hierarchically hyperbolic when $G < \Mod(S)$ is a lattice Veech group \cite{DDLS_Veech_I,DDLS_Veech_II}. Bongiovanni subsequently extended this to all finitely generated Veech groups \cite{Bongiovanni_Veech}. Veech groups are the prototypical examples of PGF groups; see \cite{Tang_Veech}. However, the proof that $\Gamma_G$ is hierarchically hyperbolic in these cases relied heavily upon on the flat geometry associated with the Veech group rather than on the PGF structure.  The current work generalizes the Veech case to rank 1 PGF groups. This allows us to produce an ``internal'' proof of the Veech case as well as cover the LR surface groups and many new examples created by a combination theorem of Udall  \cite{Udall}; see also \cite{ABDHMW_Combination}.

\begin{restatable}{theorem}{PGF}
    \label{thm:Intro_PGF_w/cyclic}
    If $G < \Mod(S)$ is a parabolically geometrically finite group with cyclic peripherals, then the $\pi_1S$--extension $\Gamma_G$ is a hierarchically hyperbolic group.
\end{restatable}

The key place in \cite{DDLS_Veech_I} where the flat geometry of the Veech group is used is in proving Gromov hyperbolicity of a natural tree bundle that $\Gamma_G$ acts on. We replace the reliance on flat geometry with a new general combination theorem for bundles of hyperbolic graphs of over a hyperbolic base. 

\begin{theorem}\label{thm:intro_combination}
    Let $p \colon \E \to \B$ be a graph bundle where $\B$ is $\delta$--hyperbolic and all fibers are $\delta$--hyperbolic (but not necessarily properly embedded). Then the total space $\E$ is hyperbolic provided there is a collection of sections $\mathcal{S}$ that satisfy the \emph{externally flaring} condition of Definition \ref{defn:externally flaring graph bundle}.
\end{theorem}

While Theorem \ref{thm:intro_combination} is similar to other results in the literature (e.g. \cite{MS:combination, BestvinaFeighn, Hamenstadt_extensions_of_surface_groups}) in using a ``flaring condition'' to encode hyperbolicity, our result differs from previous combination theorems in not requiring the fibers of the bundle to be properly embedded. This relaxation is necessary for our work as the tree bundle that $\Gamma_G$ acts on has fibers that are not properly embedded.  We are able to navigate around this lack of properness by phrasing our flaring condition in terms of distances in the total space $\E$ instead of distances in the individual fibers. 

Our accommodation of improperly embedded fibers requires us to include several additional assumptions about the existence and behavior of sections in the definition of a externally flaring graph bundle. We are able to verify these additional assumption in the case of PGF groups by viewing the associated tree bundle as a subbundle of the curve graph $\C(\dot S)$. The hyperbolicity of $\C(\dot S)$  and Masur--Minksy's bounded geodesic image theorem are then critical to verifying that the $\Gamma_G$ tree bundles satisfy our flaring condition.

Our combination theorem is applicable to any parabolically geometrically finite group, and not just those with cyclic peripherals. 

\begin{theorem}\label{thm:intro_PGF_hyp}
    Let $G < \Mod(S)$ be any parabolically geometrically finite group. Then $\Gamma_G$ acts on a hyperbolic space $\E$ that is a tree bundle whose base is the orbit of $G$ in the curve graph $\C(S)$.
\end{theorem}

The reason that Theorem \ref{thm:intro_PGF_hyp} does not lead to an immediate proof that $\Gamma_G$ is hierarchically hyperbolic for all PGF groups is that it has not yet been established that $\Gamma_A$ is hierarchically hyperbolic when $A <\Mod(S)$ is  virtually an abelian group generated by disjoint multi-twists. When this result in established, our proof of Theorem \ref{thm:Intro_PGF_w/cyclic} will likely be readily adapted to prove that $\Gamma_G$ is hierarchically hyperbolic for any PGF group.

\subsection{Geometric Finiteness in $\Mod(S)$ and future applications}\label{subsec:geom_finite}

Farb and Mosher introduced convex cocompactness in $\Mod(S)$ in analogy with convex cocompact subgroups of $\Isom(\mathbb{H}^n)$ \cite{FM_convex_cocompact}. In $\Isom(\mathbb{H}^n)$, the convex cocompact groups are the geometrically and dynamically best behaved groups of isometries. 
The analogy between convex cocompact groups in $\Isom(\mathbb{H}^n)$ and $\Mod(S)$  is remarkably robust with many of the equivalent formulations of convex cocompactness in $\Isom(\mathbb{H}^n)$ having analogous formulations satisfied in the setting of $\Mod(S)$ \cite{Kent_Leininger_shadows,Hamenstadt_extensions_of_surface_groups}. For example, in both settings, the convex cocompact groups are those where the orbit map is a quasi-isometric embedding in a hyperbolic space, either $\mathbb{H}^n$ or the curve graph $\C(S)$ respectively.

In $\Isom(\mathbb{H}^n)$, geometric finiteness is a natural relaxation of convex cocompactness that can be thought of as convex cocompactness relative to parabolic subgroups. Parabolically geometrically finite groups are a straight forward analogy of this idea to the mapping class group. Thanks to Theorem \ref{thm:Intro_PGF_w/cyclic}, we know several prominent examples of PGF groups have hierarchically hyperbolic extensions, supporting Mosher's suggestion that geometric finiteness should correspond to a generalization of Gromov hyperbolicity. The last remaining step in the PGF case is removing the cyclic peripheral assumption in Theorem \ref{thm:Intro_PGF_w/cyclic}. As mentioned after Theorem \ref{thm:intro_PGF_hyp}, this will likely come down to proving hierarchical hyperbolicity of the extension of abelian twist groups of $\Mod(S)$.

\begin{question}\label{ques:PGF}
    If $G$ is PGF, is $\Gamma_G$ hierarchically hyperbolic?
\end{question}

There are many other subgroups of $\Mod(S)$ that ought to be considered geometrically finite, but do not fall into the PGF framework. The most notable are $\Mod(S)$ itself and the stabilizers of multicurves on the surface.  These examples have also been shown to have hierarchically hyperbolic extensions \cite{BHS_HHSII, Russell_Stabilizers}. If there is to be a unified definition of geometric finiteness for all of these examples in $\Mod(S)$, it is most likely to be the answer to this question:

\begin{question}\label{ques:geom_finite}
    What geometric property of $G < \Mod(S)$ is equivalent to $\Gamma_G$ being a hierarchically hyperbolic group?
\end{question}

Question~\ref{ques:geom_finite} is already interesting for subgroups $G<\Mod(S)$ which are hyperbolic relative to a collection of abelian subgroups generated by multitwists. In this case, the problem becomes deciding whether there are conditions on an HHG structure on $\Gamma_G$ that imply that $G$ is PGF.

There are various constructions of right-angled Artin subgroups of mapping class groups; see e.g.~\cite{CrispWiest,CLM_RAAGs,koberda2012right,runnels2021effective}. These are not always PGF, as they are usually not even hyperbolic relative to abelian subgroups, but they could be considered geometrically finite. In many cases, their hierarchically hyperbolic structure is known to interact well with the hierarchically hyperbolic structure of the ambient mapping class group. For instance, the proofs in \cite{CLM_RAAGs,runnels2021effective} rely on subsurface projections and distance formulas.

\begin{question}\label{ques:RAAGs}
What are sufficient conditions on a right-angled Artin subgroup of the mapping class group of a closed surface to guarantee it has a hierarchically hyperbolic extension?
\end{question}

 Parabolically geometrically finite groups have an even further generalization to \emph{reducibly geometrically finite} groups, where we let the peripheral subgroups be any subgroup of $\Mod(S)$ that is contained in the stabilizer of a multicurve instead of just abelian twist groups (see \cite{ABDHMW_Combination,Udall} for definitions). Our construction of tree bundles in Section \ref{sec:hyperbolic GF bundles} in the PGF case has an analogous construction in the RGF case. This provides an approach to begin tackling the following question.

\begin{question}
    For $G <\Mod(S)$ reducibly geometrically finite, what conditions on the peripheral subgroups imply that $\Gamma_G$ is hierarchically hyperbolic?
\end{question}

A particularly promising case of this question are RGF groups where the peripheral subgroups are the full stabilizers of multicurves. 

 \begin{question}
     Let $G <\Mod(S)$ be reducibly geometrically finite.  If each peripheral subgroup of $G$ is the full stabilizer of a multicurve, is $\Gamma_G$ hierarchically hyperbolic?
 \end{question}

Finally, hierarchical hyperbolicity is often helpful in establishing quasi-isometric rigidity (e.g. \cite{behrstock2012geometry, HHS_quasiflats, DDLS_Veech_II}) inspiring our final question.

\begin{question}\label{ques:QI_rigid}
    Does $\Gamma_G$ satisfy some form of quasi-isometric rigidity, when $G$ is a LR surface group?
\end{question}

A concrete version Question \ref{ques:QI_rigid} would be whether any group quasi-isometric to $\Gamma_G$ has a splitting as a graph of groups where the vertex groups are commensurable to extensions of Veech groups. The techniques from \cite[Section 5]{DDLS_Veech_II} are likely a good starting point for establishing this statement.

\subsection{Outline of the paper and summary of the proof} 

We start by laying out a number of facts about Gromov hyperbolic spaces in Section \ref{sec:background}.  Many readers can safely skip it and refer back to it as needed.

Section \ref{sec:combination theorem} contains our definition of a externally flaring graph bundle (Definition \ref{defn:externally flaring graph bundle}) and the proof of our combination theorem (Theorem \ref{thrm:combination_theorem}). This combination theorem can be taken as a ``black box'' for our applications in Section \ref{sec:hyperbolic GF bundles}. The work in Section \ref{sec:combination theorem} is heavily inspired by a combination theorem of Mj and Sardar for hyperbolic graph bundles with proper fibers \cite{MS:combination} (Mj and Sardar also acknowledge an intellectual debt to \cite{Hamenstadt_extensions_of_surface_groups}). The work in \cite{MS:combination} begins by establishing the existence and properties of quasi-isometric sections $\B \to \E$ for metric grapuh bundles $\E \to \B$ with uniformly hyperbolic base and fibers. Mj and Sardar then use these qi-sections to build ``ladders'' in the total space $\E$. These ladders are essentially the subbundles of $\E$ formed by ``flowing'' a geodesic in a single fiber over the entire base. Mj and Sardar then use their flaring condition to prove that these ladder are hyperbolic and then use the hyperbolicity of the ladders to verify the hyperbolicity of $\E$.

Our main contribution to the proof of Theorem \ref{thm:intro_combination} is isolating the right conditions for the strategy of \cite{MS:combination} to work in the absence of properly embedded fibers. Our first insight was that the flaring condition needed to be stated in terms of the ambient metric instead of the metric on the fibers (when fibers are properly embedded, flaring in the fiber metric implies flaring in the ambient metric, but this is not the case for us). The second realization was that Mj and Sardar's construction of qi-sections also fails without proper fibers because one is not guaranteed to have uniform quasi-isometries between adjacent fibers. Thus, our definition of externally flaring graph bundles includes the existence of a family of sections with several properties inspired from lemmas in \cite{MS:combination}. Armed with the right set of conditions, our proof of Theorem \ref{thm:intro_combination} can follow the strategy of \cite{MS:combination} with minor adaptions. See Section~\ref{sec:contrast_proper_case} for a larger discussion about the similarities between our work and \cite{MS:combination}.

Section \ref{sec:hyperbolic GF bundles} contains the application of our combination theorem to parabolically geometrically finite groups (Theorem \ref{thm:combination_applies}).   The starting tool here is the ``forgetful map'' between curve graphs $\Psi\colon \C(\dot S) \to \C(S)$ induced by filling in the puncture on $\dot S$. A theorem of Kent, Leininger, and Schleimer \cite{KLS_Trees} tells us that the preimage of curves on $S$ under $\Psi$ are Bass-Serre trees dual to the splitting of $\pi_1S$ determined by the curve. This gives $\C(\dot S)$ the structure of a bundle of trees over $\C(S)$.  We exploit this structure to build a tree bundle for the extension of a PGF group.

For simplicity, we will describe the case where $G$ is a PGF group with a single  peripheral subgroup $H$, where $H$ is a cyclic group generated by a twist around a single curve $v$. The general case is described in Section \ref{subsec_E_for_RGF}. Our tree bundle $\E$ for the extension group $\Gamma_G$ is constructed by taking $\Psi^{-1}(G \cdot v)$ and then adding certain edges between trees  that differ by generators of $G$. Because $G$ is PGF, the orbit $G \cdot v$ is quasi-isometric to a hyperbolic graph, making $\E$ a graph bundle with hyperbolic base and hyperbolic fibers. To verify that $\E$ is  externally flaring, we exploit the fact that $\E$ is essentially a sub-bundle of $\C(\dot S)$. This lets us use existing work \cite{KLS_Trees,LeiMjSch} to build the required sections as well as us the hyperbolicity of $\C(\dot S)$ plus Masur--Minsky's bounded geodesic image theorem to verify that $\E$ satisfies the flaring condition.

Section \ref{sec:cyclic PGF HHG} contains our proof that extension of rank 1 PGF groups are hierarchically hyperbolic (Theorem \ref{thm:Intro_PGF_w/cyclic}).  The starting observation here is that the extension of the peripheral subgroups are already known to have a hierarchically hyperbolic structure thanks to work of Hagen, Russell, Spriano, and Sisto \cite{HRSS}. In particular, when $H<\Mod(S)$ is a cyclic group generated by a twist around the curve $v$, then $\Gamma_H$ has a combinatorial hierarchically hyperbolic structure based on the tree $\Psi^{-1}(v)$. 

In the special case where $G$ has a single  peripheral subgroup $H$ generated by the twist around $v$, then every fiber of the tree bundle $\E$ is the tree used to construct the hierarchically hyperbolic structure for a coset of $\Gamma_H$. Thus, the hierarchically hyperbolic structure we build for $\Gamma_G$ is essentially a bundle of the structures on each of these cosets. The bulk of our work is therefore to show that you can move consistently between the structures on different fibers. For this we again turn to the tools in \cite{KLS_Trees,LeiMjSch} to understand the fibration $\Psi \colon \C(\dot S) \to \C(S)$.

\subsection*{Acknowledgments}
The authors would like to thank Mahan Mj for useful discussions. 
Dowdall was partially supported by NSF grants DMS-2005368 and  DMS-2405061.
 Durham was partially supported by NSF grant DMS-1906487 and NSF CAREER DMS-2441982.
Leininger was supported by NSF grants DMS-2305286 and DMS-2639932, and DMS-1928930 while in residence at the Simons Laufer Mathematical Sciences Institute in Berkeley, California, during the Spring 2026 Semester.

\subsection*{AI Disclosure}
No AI was used for any part of this work, other than that imposed by {\em Google} in its search engine.

\section{Background}\label{sec:background}

\subsection{Conventions} \label{sec:conventions}
All metric spaces considered in this paper are assumed to be \emph{geodesic}, meaning any two points may be joined by (the image of) an isometric embedding of an interval in the real line. Given points $x,y$ in the space, we write $[x,y]$ for a choice of  geodesic between. 

For subsets $A,B$ of a metric space $(X,d)$, we write $d(A,B)$ for the infimum of all distances $d(a,b)$ with $a\in A$ and $b\in B$. We also use $\nbhd_D(A)$ to denote the \emph{$D$--neighborhood} of $A$, which is the set of points $x\in X$ with $d(x,A)\le D$.
The set $A$ is \emph{$Q$--quasiconvex} if $[a,a']\subset \nbhd_Q(A)$ for all $a,a'\in A$.
In contrast to $d(A,B)$, the \emph{Hausdorff distance} between $A$ and $B$ is the (possibly infinite) infimum of all $r> 0$ for which $B\subset \nbhd_r(A)$ and $A\subset \nbhd_r(B)$. 

Finally, we will use $\diam(A,B)$ as shorthand for the diameter of the union $A\cup B$.

We sometimes consider set-valued maps and, with an abuse, still use the notation $X\to Y$ to indicate a map defined on $X$ but taking values in subsets of $Y$. 
For example, given a subset $A$ of a metric space $X$, the \emph{closest-point projection to $A$} is the (set-valued) map $\rho_A \colon X\to A$ sending $x\in X$ to the (always nonempty) set of points 
$a\in A$ with $d(x,a)\le d(x,A)+1$.

For the most part we work with \emph{(simplicial) metric graphs}, i.e.~$1$--dimensional simplicial complexes equipped with the path metric in which each edge has length $1$.

Since we generally only care about the vertices, we adapt the following convention for maps: A \emph{coarse graph map} $\phi\colon X\to Y$ between metric graphs is an assignment that sends each vertex of $X$ to a vertex of $Y$. This is then extended to a (discontinuous) set-valued map on the whole of $X$ by declaring that points in the interior of an edge joining vertices $x,x'\in X$ are sent to the set $\{\phi(x),\phi(x')\}$. 

A coarse graph map $\phi\colon X\to Y$ will be called a \emph{$K$--coarsely Lipschitz} if $d(\phi(x),\phi(x'))$ is bounded above by $K d(x,x')+K$ for all vertices $x,x'$ of $X$. If it is additionally bounded below by $\tfrac{1}{K}(d(x,x')-K)$ we moreover call $\phi$ a \emph{$K$--quasi-isometric embedding}.

\subsection{Useful facts about hyperbolic spaces}

In this subsection, we collect some basic facts about hyperbolic spaces that we need in Section \ref{sec:combination theorem}.  The proofs involve standard arguments in Gromov hyperbolic geometry; we provide proof sketches for completeness, leaving the details for the interested reader.
Recall that a metric space $X$ is \emph{$\delta$--hyperbolic} if $[x,y]\subset \nbhd_\delta([y,z]\cup[x,z])$ for all $x,y,z\in X$. We use $\triangle(x,y,z)$ to denote a triangle with geodesic sides $[x,y]$, $[y,z]$, $[z,x]$.

Quasiconvex subsets of hyperbolic spaces admit coarsely well-defined closest point-projections that are moreover coarsely Lipschitz. This is summarized in the following; see \cite[Lemma 11.53]{dructu2018geometric} and \cite[Lemmas 1.102 and 1.105]{kapovich2024trees}.

\begin{lemma}\label{lem:hyp pass close}
    Suppose $X$ is a $\delta$--hyperbolic space and $A \subset X$ is a $Q$--quasiconvex subset.  There exists $\kappa_0 = \kappa_0(\delta, Q)>0$  so that $\rho_A\colon X\to A$ satisfies the following:
    \begin{enumerate}
        \item\label{item:closet-point-coarsely-lip} 
$\rho_A$ is coarsely well-defined and coarsely Lipschitz: $\diam(\rho_A(x),\rho_A(y))\le\kappa_0 d(x,y) + \kappa_0$ for all $x,y\in X$. In particular, $\diam(\rho_A(x))\le \kappa_0$. 
        \item\label{item:pass_near_closest_point} If $x \in X$, $z \in A$, and $x' \in \rho_A(x)$, then every geodesic from $x$ to $z$ passes within $\kappa_0$ of $x'$.
    \end{enumerate}
\end{lemma}

Another well-known fact about projections to quasiconvex subsets is that if two points project far away, then any geodesic connecting them passes close to both projections, as stated below.

\begin{lemma}\label{lem:hyperbolic_BGI} 
Let $X$ be $\delta$--hyperbolic, $A\subset X$ a $Q$--quasiconvex subset, and $x,y\in X$. There exists $\tau=\tau(\delta,Q)\ge 1$ so that if $\diam(\rho_A(x),\rho_A(y))> \tau$, then $\rho_A(y)$ and $\rho_A(y)$ both lie in $\nbhd_\tau([x,y])$.
\end{lemma}
\begin{proof}
Take any $x'\in \rho_A(x)$ and $y'\in \rho_A(y)$.
Consider any point $p\in [x',y']$, and choose $p_0\in A$ with $d(p,p_0)\le Q$. Suppose $p$ happened to lie within $2\delta$ of some point $w\in [x,x']$. Then by the triangle inequality $d(x,p_0) \le d(x,w) + 2\delta+Q$. Hence by definition of $x'\in \rho_A(x)$ we have
\[d(x,w)+d(w,x') = d(x,x') \le d(x,A)+1 \le d(x,p_0)+1 \le d(x,w)+2\delta+Q+1.\]
This bounds $d(w,x')$ and implies $d(x',p)< \tau' = 4\delta+2Q+2$. Symmetrically $d(p,y')< \tau'$ whenever $p$ lies within $2\delta$ of $[y,y']$. 
Now suppose $\diam(\rho_A(x), \rho_A(y))> 2\tau'$. Then the points $p,q\in [x',y']$ defined by the conditions $d(x',p) = d(q,y')=\tau'$ (exist and) cannot lie within $2\delta$ of $[x,x']$ nor of $[y,y']$. Hence by $2\delta$--thinness of the quadrilateral formed by $x,x',y',y$, we may find $p',q'\in [x,y]$ with $d(p,p'),d(q,q') \le 2\delta$. Therefore $x'$ and $y'$ lie in the $(2\delta+\tau')$--neighborhood of $[x,y]$.
\end{proof}

The following lemma says that the union of two quasiconvex subsets is quasiconvex.

\begin{lemma}\label{lem:union_of_QC_subsets}
    Let $X$ be $\delta$--hyperbolic graph and $Y,Z \subset X$ be two $Q$--quasiconvex subsets. There exists $Q' \geq 0$ depending only on $\delta$, $Q$ and $d(Y,Z)$ so that $X \cup Y$ is $Q'$--quasiconvex.
\end{lemma}

\begin{proof}
Fix $y_0\in Y, z_0\in Z$ within $d(Y,Z)+1$ of each other. Clearly, any geodesic connecting points of $Y$ or points of $Z$ stays $Q$--close to $Y\cup Z$. Consider then $y \in Y$ and $z \in Z$. Any point of $[y,z]$ is either $2\delta$--close to $[y,y_0]$ or $[z,z_0]$, hence $(2\delta+Q)$--close to $Y\cup Z$, or $2\delta$--close to $[y_0,z_0]$, hence $(d(Y,Z)+1)/2$--close to $Y\cup Z$. Therefore, we can take $Q'=\max\{2\delta+Q,(d(Y,Z)+1)/2\}$.
\end{proof}

Recall that a \emph{$K$--quasigeodesic} in $X$ is a (possibly set-valued) map $\gamma\colon I\to X$ of an interval $I\subset \mathbb{R}$ so that $\diam(\gamma(t),\gamma(s))$ is bounded above by $K\left\vert t - s\right\vert+K$ and below by $\tfrac{1}{K}\left(\left\vert t-s\right\vert -K\right)$ for all $s,t\in I$. It is a fundamental feature of hyperbolicity that quasigeodesics stay close to geodesics:

\begin{lemma}[Morse lemma]\label{lem:morse}
Suppose $X$ is $\delta$--hyperbolic and $\gamma\colon [a,b]\to X$ a $K$--quasigeodesic. There exists $M = M(\delta,K)$ so that $\gamma([a,b])$ has Hausdorff distance at most $M$ from any $X$--geodesic $[\gamma(a),\gamma(b)]$.
\end{lemma}

One important consequence of this lemma is that hyperbolicity is a quasi-isometry invariant: If metric graphs are quasi-isometric to each other, then one is hyperbolic if and only if the other is. Another  consequence is that quasi-convexity is preserved under quasi-isometric embeddings:

\begin{lemma}\label{lem:quasi-convex_perserved}
Let $\phi\colon X\to Y$ be coarse graph map. Assume $\phi$ is a $K$--quasi-isometric embedding, that $Y$ is $\delta$--hyperbolic, and that $A\subset X$ is $Q$--quasiconvex. Then there exists a constant $Q' = Q'(\delta,K,Q)$ such that $\phi(A)$ is $Q'$--quasiconvex in $Y$.
\end{lemma}
\begin{proof}
Consider $a'_1,a'_2\in \phi(A)$ and choose vertices $a_i\in A$ with $\phi(a_i) = a'_i$. Choose also a geodesic $\gamma' = [a'_1,a'_2]$ in $Y$.  Let $\gamma = [a_1,a_2]$ be a geodesic in $X$. By quasiconvexity, $\gamma$ is contained in $\nbhd_Q(A)$ and hence $\phi(\gamma)$ is contained in the $\nbhd_{KQ+K}(\phi(A))$. Since $\phi(\gamma)$ is a $K$--quasigeodesic, $\gamma'$ and $\phi(\gamma)$ also have Hausdorff distance at most $M = M(\delta,K)$ by Lemma \ref{lem:morse}. Therefore $[a'_1,a'_2]$ is contained in the $KQ+K+M$ neighborhood of $\phi(A)$, as required.
\end{proof}

\begin{lemma}\label{lem:coarse center}
Let $X$ be a $\delta$--hyperbolic geodesic space.  For every $\epsilon>0$ there exists $\delta_1=\delta_1(\delta, \epsilon)>0$ so that the following holds.  Let $a,b,c \in X$ and $[a,b],[a,c],[b,c]$ be geodesics between them.  If $p \in [a,b]$ is so that $d_X(p, [b,c])< \epsilon$ and $d_X(p, [a,c])<\epsilon,$ then the projection $\rho_{[a,b]}(c)$ lies in $\nbhd_{\delta_1}(p)$.
\end{lemma}

\begin{proof}
Consider any projection point $q\in \rho_{[a,b]}(c)$. Up to swapping the roles of $a$ and $b$ we may assume $q\in [a,p]\subset [a,b]$. Choose $p'\in [a,c]$ with $d(p,p')\le \epsilon$. Since the triangle $\triangle(a,p,p')$ has a side $[p,p']$ of length at most $\epsilon$, thinness implies there is a point $q'\in [a,p']\subset[a,c]$ with $d(q,q')\le \delta+\epsilon$. Thus by the triangle inequality we have $d(q',p') \ge d(p,q) - \delta-2\epsilon$. This in turn implies
\[ d(c,q) +\epsilon+\delta\ge d(c,q')  = d(c,p') + d(p',q') \ge d(c,p) + d(p,q) -\delta-3\epsilon.\]
Using the fact that $q$ is a closest point to $c$ along $[a,b]$ also gives $d(c,q) \le d(c,p)+1$. Combining these inequalities and cancelling the $d(c,p)$ terms gives $d(p,q) \le 4\epsilon+2\delta+1$, as desired.
\end{proof}

\begin{lemma}\label{lem:projections_and_thin_triangles}
    Let $X$ be a $\delta$--hyperbolic graph and $x,y,z \in X$. Let $z_{x,y}$ and $y_{x,z}$ lie in the closest point projections of $z$ and $y$ onto $[x,y]$ and $[x,z]$ respectively.  There exists $\delta_2 = \delta_2(\delta)>0$ so that the following hold:
    \begin{enumerate}
        \item $d(z_{x,y},y_{x,z}) \leq \delta_2$.
        \item  $[x,y_{x,z}] \subset N_{\delta_2}([x,z_{x,y}])$.
        \item The closest point projection of $[z,y_{x,z}]$ to $[x,y]$ is contained in the $\delta_2$--neighborhood of $z_{x,y}$.
    \end{enumerate}
\end{lemma}
\begin{proof}
For (1), by Lemma \ref{lem:hyp pass close}(\ref{item:closet-point-coarsely-lip}) there are points $y'\in [z,x]$ and $x'\in [z,y]$ within $\kappa_0$ of $z_{x,y}$. Observe that $y'$ then lies within $2\kappa_0$ of both $[x,y]$ and $[z,y]$; hence Lemma \ref{lem:coarse center} bounds  $d(y',y_{x,z})$ by a constant $\delta_1$ depending on $\delta$ and $2\kappa_0$, and thus only on $\delta$. It follows that $d(z_{x,y},y_{x,z})\le \kappa_0+\delta_1$.

Item (2) follows from the slim triangle $\triangle(x,y_{x,y},z_{x,y})$ having a side $[z_{x,y},y_{x,z}]$ of bounded length.

Towards item (3), let $\tau$ be as in Lemma \ref{lem:hyperbolic_BGI}. Because of item (1) we can regard $[z,y_{x,z}]$ as a concatenation of a subgeodesic (containing $y_{x,y}$) of length bounded in terms of $\delta_2$ and $\kappa$, and a subgeodesic that does not pass $\tau$--close to $z_{x,y}$. The first subgeodesic has uniformly bounded projection to $[x,y]$ by item \eqref{item:closet-point-coarsely-lip} of Lemma \ref{lem:hyp pass close}, and the second by Lemma \ref{lem:hyperbolic_BGI}.
\end{proof}

Lemma \ref{lem:coarse center} shows that each triangle $\triangle(x,y,z)$ admits a point $p$ that lies within $\delta_2$ of each of the three geodesics; we call such a point $p$ a \emph{coarse center} of the triangle $\triangle(x,y,z)$. More generally:
\begin{lemma}\label{lem:thin_rectangles}
 Let $X$ be $\delta$--hyperbolic graph and $C \geq 0$. There exists $\nu = \nu(C, \delta)>0$ so that for all  vertices $x,y,z,q \in X$, if $d(z,q)\leq C$, then there exist vertices $w \in [x,y]$, $x'\in [x,q]$ and $y' \in [y,z]$  so that any two are $\nu$--apart.    
\end{lemma}
\begin{proof}
Let $p$ be a coarse center for $\triangle(x,y,z)$; hence there are vertices $w\in [x,y]$, $y'\in [y,z]$, and $x''\in [x,z]$ each within $\delta_2$ of $p$. By hyperbolicity $x''$ is within $\delta$ of a point $m\in [x,q]\cup [q,z]$ which, since $d(z,q)\le C$, is in turn within $C$ of a point $x'\in [x,q]$. By the triangle inequality, we see that the distance between any pair of $w,x',y'$ is at most $2\delta_2+\delta+C$.
\end{proof}

\subsection{Hyperbolicity criteria}
We state the three hyperbolicity criteria that we will use in this paper. To set some terminology, we say that a subset $A$ of a graph is \emph{$D$--coarsely connected} if it is contained in a connected subset $B$ with $B\subset N_D(A)$.

Our first criterion is the commonly used ``guessing geodesics'' technique in $\delta$--hyperbolic geometry; see \cite[Proposition 3.1]{bowditch2006intersection}, \cite[Theorem 3.15]{masur2013geometry}, \cite[Proposition 3.5]{hamenstadt2007geometry}.  Note that the second conclusion is not explicitly stated in \cite[Proposition 3.5]{hamenstadt2007geometry}, but appears as part of the proof.

\begin{proposition}[Guessing geodesics]
\label{prop:guessing}
For all $D\geq 1$ there exists $\delta_{guess}\geq 1$ with the following property.
Let $X$ be a connected graph, and for all vertices $x,y$ of $X$ let $\varsigma(x,y)$ be a fixed $D$--coarsely connected subset containing $x$ and $y$, with $\varsigma(x,y)=\varsigma(y,x)$. Suppose that
\begin{enumerate}
    \item If $x,y$ are adjacent then the diameter of $\varsigma(x,y)$ is at most $D$.
    \item for all vertices $x,y,z$ of $X$ we have $\varsigma(x,y)\subset \N_D(\varsigma(x,z)\cup\varsigma(z,y))$.
\end{enumerate}
Then $X$ is $\delta_{guess}$--hyperbolic and for each $x,y$, the coarse path $\varsigma(x,y)$ is within uniform Hausdorff distance of any geodesic $[x,y]$.
\end{proposition}

Our second criterion is a restatement of \cite[Corollary 1.52]{MS:combination}. We say a \emph{line of spaces} consists of the following data:
\begin{itemize}
    \item A sequence $X_0,\dots,X_n$ of metric graphs called \emph{vertex spaces}.
    \item  Two sequences of connected subgraphs $$E_0\subset X_0,\dots,E_{n-1}\subset X_{n-1} \text{ and } E'_0\subset X_1,\dots,E'_{n-1}\subset X_{n}$$  where there is an isomorphism $\phi_i \colon E_i \to E'_i$ for each $i \in\{0,\dots, n-1\}$. We call the $E_i$ and $E'_i$ \emph{edge spaces}. 
\end{itemize}
The \emph{total space} $X$ of a line of spaces is obtained  from the disjoint union of the $X_i$  by adding edges connecting $x$ and $\phi_i(x)$ for each $x \in E_i$ and each $i \in \{0,\dots, n-1\}$.

Using the language of lines of spaces,  \cite[Corollary 1.52]{MS:combination} implies the following. 
Note that our formulation below includes the hypothesis that edge spaces are uniformly quasi-isometrically embedded in the whole space. This strong requirement simultaneously implies the two hypotheses of \cite[Corollary 1.52]{MS:combination} regarding edge spaces: namely that they are quasi-isometrically embedded in their associated vertex spaces and properly embedded in the whole space.

\begin{proposition}
\label{prop:malnormal_glue}    
    For every $\delta,Q,C\geq 0$ there exists $\delta_3 = \delta_3(\delta, Q, C) \geq 0$ with the following property. Consider a line of spaces, with total space $X$, where the vertex spaces are $\delta$--hyperbolic, the edge spaces are $C$--quasi-isometrically embedded in $X$, and the closest point projections onto each other of two edges spaces contained in the same vertex space have diameter bounded by $C$. Then the total space of the line of spaces is $\delta_3$--hyperbolic.
\end{proposition}

Our last hyperbolicity criterion is a variant of Proposition \ref{prop:malnormal_glue}  where we relax the requirements on the edge spaces and restrict to only two vertex spaces. To set-up this situation, let $X,Y$ be graphs and  $A$ be a subgraph of $X$ equipped with a coarse graph map $\phi \colon A\to Y$.  We obtain a new graph $X\cup_\phi Y$ by adding an edge from $a$ to $\phi(a)$ for each vertex $a\in A$.  We think of this process as gluing $X$ to $Y$ along $A$.

\begin{proposition}
\label{prop:gluing_lemma}
    For all $\delta, Q,K$ there exist $\delta_3, K'$ such that the following hold. Let $X,Y$ be $\delta$--hyperbolic graphs and let $A$ be a $Q$--quasi-convex subgraph of $X$. Let $\phi\colon A\to Y$ be a coarse graph map that is a $K$--quasi-isometric embedding when $A$ is endowed with the induced metric (not its path metric). Then $X\cup_\phi Y$ is a $\delta_3$--hyperbolic graph in which $X$ and $Y$ are $K'$--quasi-isometric embedded.
\end{proposition}

\begin{proof}
In this proof we call a constant ``uniform'' if it only depends on $\delta,Q,K$.

    Let $X_1$ be the full subgraph of $Z=X\cup_\phi Y$ with the same vertex set as the union of $X$ and the union of all geodesics in $Y$ connecting points of $\phi(A)$. Similarly, let $X_2$ be the union of $Y$ and all geodesics in $X$ connecting points of $A$.

We claim that $X_1$ is quasi-isometric to $X$, with uniform constants, and similarly for $X_2$. This is because geodesics in $Y$ connecting points of $\phi(A)$ stay uniformly close to $\phi(A)$, which lies in the 1-neighborhood of $X$ in $A$, and this allows us to construct a coarse inverse of the inclusion of $X$ into $X_1$.  An analogous argument works for $X_2$.

Set $Y_1=X_1\cap X_2$, that is, the full subgraph with the same vertex set as the union of all geodesics in $X$ connecting points of $A$ and all geodesics in $Y$ connecting points of $\phi(A)$. There is a coarse retraction, with uniform constants, from $Z$ to $Y_1$ obtained combining the closest-point projections in $X$ to $A$ and in $Y$ to $\phi(A)$. In particular, $Y_1$ is uniformly quasi-isometrically embedded in $Z$.

The above shows that the hypotheses of \cite[Corollary 1.52]{MS:combination} apply (the coboundedness assumption holding vacuously in our case), so we get the desired conclusion.
\end{proof}

\section{Combination theorem}\label{sec:combination theorem}

In this section, we prove our main combination theorem.  We first formulate the notion of an \emph{externally flaring graph bundle}  in Definition \ref{defn:externally flaring graph bundle} and then prove such spaces are hyperbolic in Theorem \ref{thrm:combination_theorem}. This result will be applied in Section \ref{sec:hyperbolic GF bundles} to build hyperbolic spaces for extensions of PGF groups.  Our work is in this section inspired by Mj and Sardar's work  on metric graph bundles \cite{MS:combination}. We discuss the similarities and differences between our work our and \cite{MS:combination} in Section~\ref{sec:contrast_proper_case}.

\begin{definition} [Graph Bundle]\label{def:graph_bundle}
A \textbf{\em graph bundle} is a surjective simplicial map $p\colon \E\to \B$ between connected graphs $\E$ and $\B$, so that the subgraph $\E_v = p^{-1}(v)$ is connected for each  vertex $v\in \B$. We call  $\E_v $ the \textbf{\em fiber} over  $v$.
A \textbf{\em section} of the graph bundle is a simplicial embedding $\Sigma\colon \B \to \E$ so that $p \circ \Sigma(v) = v$ for all vertices $v \in \B$. 
\end{definition}

The \emph{metric} graph bundles used by Mj and Sardar in \cite[Definition 1.5]{MS:combination} are graph bundles with two additional assumptions: firstly, that if $v,w\in \B$ are adjacent, then each vertex of $\E_v$ is adjacent to some vertex of $\E_w$, and, secondly, each fiber $\E_v$ is uniformly properly embedded in $\E$, meaning that distance $d_v$ in $\E_v$ is bounded above by some function of the distance in $\E$. These two properties imply that adjacent fibers $\E_v$ and $\E_w$ are uniformly quasi-isometric to each other. 

We emphasize that our graph bundles have no properness assumption. In particular, vertices $x,y\in \E_v$ with arbitrarily large distance in the graph $\E_v$ may both be adjacent to the same vertex $z$ in a neighboring fiber $\E_w$. For this reason, the distance $d_v(x,y)$ in the metric graph $\E_v$ is essentially a useless quantity for us. However, we shall  need geodesics in $\E_v$ and we use $[x,y]_{v}$ to denote a choice of some geodesic in $\E_v$ between $x$ and $y$.

We now introduce our notion of an externally flaring graph bundle.  In what follows, we will refer to the conditions in the definition as ``axioms''.

\begin{definition}[Externally flaring graph bundle] \label{defn:externally flaring graph bundle} 
Let $p \colon \E \to \B$ be a graph bundle and let $\mathcal{S}$ be a family of sections of $p$. We call the sections in $\mathcal{S}$ \textbf{\em tight sections}. We say that the bundle $\E$ is an \textbf{\em externally flaring graph bundle} if $\E$ and $\mathcal{S}$ satisfy the following axioms.

\begin{enumerate}
    \item \label{item:hyperbolic} (Hyperbolicity) $\B$ and $\E_v$ are uniformly hyperbolic.
    \item \label{item:full} (Full Sections) For each $x\in \E$ there is a tight section through $x$.
    \item \label{item:robust} (Robust Sections) For every pair of tight sections $\Sigma_1,\Sigma_2 \in \mathcal{S}$  and any choice of geodesics $[\Sigma_1(u),\Sigma_2(u)]_{u}$ for each vertex   $u\in \B$, if $x$ is a vertex of $[\Sigma_1(v),\Sigma_2(v)]_{v}$ for some $v \in \B$, then     
    there exists a tight section $\Sigma_x$ through  $x$  with the property that $\Sigma_x(u)\in [\Sigma_1(u),\Sigma_2(u)]_{u}$ for all vertices $u\in \B$.
    \item\label{item:retraction} (Fiberwise retraction) There exists $\kappa \geq 1$ with the following property. 
Suppose vertices $x,y,z$ in a fiber $\E_v$ are, respectively, adjacent to vertices $x',y',z'$ in a fiber $\E_{v'}$. Let $q$ be a closest point projection of $z$ to a geodesic $[x,y]_{v}$ in $\E_v$, and define $q'$ similarly. Then $d_\E(q,q')\leq \kappa$.
   
    \item\label{item:flaring} (Flaring in $\E$) There exists $\sigma>0$ such that for all $B\geq \sigma$ there exists a diverging function $\lambda = \lambda (B)\colon \mathbb R^+\to \mathbb R^+$ such that the following holds for any two tight sections $\Sigma_1,\Sigma_2 \in \mathcal{S}$ and any geodesic $\gamma=[u,v]\subset \B$: if $d_\E(\Sigma_1(v),\Sigma_2(v))\leq B$ but $d_\E(\Sigma_1(w),\Sigma_2(w))\ge \sigma$ for all $w\in \gamma-\{v\}$, then 
    \[d_{\E}(\Sigma_1(u),\Sigma_2(u))\geq \lambda( d_\B(v,u)).\]
\end{enumerate}
\end{definition}

Given a graph bundle $p:\E\to\B$ where $\E$ is hyperbolic, it is not hard to show (using that the images of the sections are quasiconvex) that any two sections flare as in Axiom \eqref{item:flaring}. On the other hand, our main result of this section is that external flaring implies hyperbolicity of the bundle.

\begin{theorem}\label{thrm:combination_theorem}
    Externally flaring graph bundles are hyperbolic.
\end{theorem}

\subsection{Contrasting the metrically proper case}
\label{sec:contrast_proper_case}
Let us briefly discuss how our Definition~\ref{defn:externally flaring graph bundle} and Theorem~\ref{thrm:combination_theorem} compares  and contrasts with the setting of \cite[Theorem 4.3]{MS:combination}, which proves that the total space $\E$ of a \emph{metric} graph bundle is hyperbolic under certain conditions. Firstly, \cite{MS:combination} requires that the base space $\B$ and fibers $\E_v$ are uniformly hyperbolic, exactly as in our Axiom (\ref{item:hyperbolic}). They additionally require that the barycenter maps $\partial^3 \E_v\to \E_v$ of the fibers are uniformly coarsely surjective. However, they do not assume our fiberwise retraction Axiom (\ref{item:retraction}), since in their context it is an automatic consequence of the metric properness condition of the bundle and the fact that adjacent fibers are uniformly quasi-isometric \cite[Lemma 1.38]{MS:combination}.
Similarly, \cite{MS:combination} does not utilize any pre-specified collection of tight sections, and does not impose any hypotheses about sections. Instead, they use properness of the metric graph bundle and the barycenter maps to \emph{prove} the following analog of our Axiom (\ref{item:full}): there is a uniform $K$ such that there exists a $K$--quasi-isometric section through any point of $\E$\cite[Proposition 2.10]{MS:combination}; indeed, obtaining this fact is the first key milestone of their proof. With this in hand, fiberwise retraction (\ref{item:retraction}) may be applied \cite[Lemma 3.1]{MS:combination} to an arbitrary qi-section to obtain robust qi-sections in analogy to our Axiom (\ref{item:robust}). 

Finally, and most importantly, \cite{MS:combination} imposes a flaring condition on qi-sections, which roughly says that for each given $K$ there exist $M,N$ and $\lambda>1$ such that if $\gamma=[u,v]\subset \B$ is a geodesic of length $2N$ with midpoint $w$ and if $\Sigma_1,\Sigma_2$ are $K$--qi sections with $d_w(\Sigma_1(w),\Sigma_2(w))\ge M$, then at least one of $d_u(\Sigma_1(u),\Sigma_2(u))$ or $d_v(\Sigma_1(v),\Sigma_2(v))$ is bigger than $\lambda d_w(\Sigma_1(w),\Sigma_2(w))$. This is similar in spirit to our flaring Axiom (\ref{item:flaring}), however we emphasize the key difference that, due to the lack of properness, our condition is formulated in terms of distance in the ambient space $\E$, rather than the (meaningless for us) distances in fibers.

Since in our application of Theorem \ref{thrm:combination_theorem}, the bundles we are concerned with come with natural sections, we did not investigate conditions under which sections exist in settings where fibers that are not properly embedded. 
Due to issues related to our sections being isometric rather than quasi-isometric, it is also not clear if \cite[Theorem 4.3]{MS:combination} can be reduced to  Theorem \ref{thrm:combination_theorem} by taking the family of sections to be the collection of sections constructed in \cite{MS:combination}. 
Exploring these matters more thoroughly suggests several natural problems to tackle, such as:
\begin{question}
Regarding the relationship between Theorem~\ref{thrm:combination_theorem} and the work of Mj--Sardar \cite{MS:combination}:
\begin{enumerate}
\item What conditions on a graph bundle imply that simplicial (or uniform quasi-isometric) sections exist through any point and in every ladder (c.f.~Axioms~\eqref{item:full} and \eqref{item:robust})?
\item Can Definition \ref{defn:externally flaring graph bundle} be weakened to allow the tight sections in $\mathcal{S}$ to be uniform quasi-isometric embeddings, rather than simplicial embeddings?
\item Can the Mj-Sardar combination theorem \cite[Theorem 4.3]{MS:combination} be deduced directly from some version of Theorem~\ref{thrm:combination_theorem} that relies on a generalization of Definition~\ref{defn:externally flaring graph bundle}?
\end{enumerate}
\end{question}

Our proof follows the basic outline set out in \cite{MS:combination}.  First we construct ``ladders'' between the tight sections in our bundles. These ladders are morally the sub-bundle of fiber geodesics between any two tight sections, and we prove they are uniformly hyperbolic. Next we construct a hyperbolic ``tripod'' bundle for any triple of tight sections by combining the ladders between each pair of tight sections. Finally, we prove the hyperbolicity of $\E$ by using the guessing geodesic criteria in Proposition \ref{prop:guessing}. Our preferred paths come from the geodesics in the ladders, which will form thin triangles by virtue of the hyperbolicity of the tripod bundle.

\subsection{Ladders and Necks}

The first construction  we need for our proof of Theorem \ref{thrm:combination_theorem} is the ladder between tight sections.  Ladders first appeared in \cite{mitra1998cannon}, inspired by \cite{CannonThurston}; see also \cite[Definition 2.3]{MS:combination}.

\begin{definition}[Ladder] \label{defn:ladder}
    Let $\kappa \geq 1$ be as in Axiom \eqref{item:retraction} of Definition \ref{defn:externally flaring graph bundle}. Suppose $\Sigma_1,\Sigma_2$ are tight sections, and for each $v\in \mathcal{B}$ choose a geodesic $[\Sigma_1(v),\Sigma_2(v)]_{v}$, which we denote $\ladder_v(\Sigma_1,\Sigma_2)$.  

The \textbf{\em ladder} between $\Sigma_1$ and $\Sigma_2$ (determined by these geodesics) is the graph $\ladder(\Sigma_1,\Sigma_2)$ obtained from
    \[ \bigcup_{v \in \B} \ladder_v(\Sigma_1,\Sigma_2)\]
by adding edges between pairs of vertices $x \in \ladder_v(\Sigma_1,\Sigma_2)$ and $y \in \ladder_w(\Sigma_1,\Sigma_2)$ with $v, w \in \B$ adjacent and so that $d_\E(x,y)\leq \kappa$.  While a ladder formally depends on the choice of geodesics $\ladder_v(\Sigma_1,\Sigma_2)$ in each fiber and not just on $\Sigma_1$ and $\Sigma_2$, we suppress this in the notation $\ladder(\Sigma_1,\Sigma_2)$.
\end{definition}

Next we define and analyze a kind of closest point projection onto the ladder:

\begin{definition}\label{defn:fiberwise projection map}
    Let $ \ladder(\Sigma_1,\Sigma_2)$ be a ladder between two tight sections $\Sigma_1,\Sigma_2$.  The \textbf{fiberwise projection map}  $$\Pi \colon \E \to \ladder(\Sigma_1,\Sigma_2)$$ is the coarse graph map defined as follows: given a vertex $x \in \E$, let $w = p(x) \in \B$, and define $\Pi(x)$ to be a closest point projection of $x\in\E_w$ to $\ladder_w(\Sigma_1,\Sigma_2)$ inside the fiber $\E_w$.
\end{definition}

\smallskip
Our first lemma illuminates why we call Axiom \eqref{item:retraction} the ``fiberwise retraction'' axiom.

\begin{lemma}\label{lem:fiberwise_proj}
    There exists $K_0\geq 1$ such that for any pair of tight sections $\Sigma_1$ and $\Sigma _2$, the fiberwise projection $\Pi \colon \E \to \ladder(\Sigma_1,\Sigma_2)$ is $K_0$--coarsely Lipschitz.
\end{lemma}
\begin{proof}
    It suffices to check that a fiberwise projection of two vertices of $\mathcal E$ connected by an edge are uniformly close. For edges that connect vertices in the same fiber, this follows from Axiom \eqref{item:hyperbolic} and the fact that projection to a geodesic in a hyperbolic space is coarsely Lipschitz (Lemma~\ref{lem:hyp pass close}). For edges connecting points in distinct fibers, this follows directly from Axiom \eqref{item:retraction}.
\end{proof}

For each ladder $\ladder(\Sigma_1,\Sigma_2)$, there is a natural coarse graph map $\iota \colon \ladder(\Sigma_1,\Sigma_2) \to \E$  defined by sending vertices in the ladder to the corresponding vertices in $\E$.

\begin{corollary}
\label{cor:ladder_qie}
 There exists $K_1 = K_1(\kappa)\geq 1$ such that for any ladder $\ladder(\Sigma_1,\Sigma_2)$, the natural map $\iota \colon \ladder(\Sigma_1,\Sigma_2) \to \E$ is a $K_1$--quasi-isometric embedding.
\end{corollary}

\begin{proof}
    By definition of the edges in the ladder, $\iota$ is $\kappa$--coarsely Lipschitz.  Let $\Pi$ be the fiberwise projection onto $\ladder(\Sigma_1,\Sigma_2)$, which is $K_0$--coarsely Lipschitz (Lemma \ref{lem:fiberwise_proj}). Since $\Pi \circ \iota$ is the identity on the vertices of the ladder, $d_{\ladder(\Sigma_1,\Sigma_2)}(x,y) \leq K_0 d_\E(x,y) + K_0$, for all vertices $x,y \in \ladder(\Sigma_1,\Sigma_2)$. Thus $\iota$ is a $K_1$--quasi-isometric embedding, where $K_1 = \max\{\kappa,K_0\}$.
\end{proof}

\begin{remark}\label{rem:laddes_as_subsets}
    Ladders are not naturally subgraphs of $\E$ because of the extra edges added between vertices in adjacent fibers.  However, the natural coarse graph map does restrict to an embedding on vertices, so we will sometimes refer to ladders as subsets of $\E$ or sections as subsets of ladders, clarifying when necessary. To emphasize:  when viewed as a subset of $\E$ the ladder $\ladder(\Sigma_1,\Sigma_2)$ is simply the image of the map $\iota$, which is the set of vertices in the ladder.
\end{remark}

The next lemma says that the geodesics in the fibers used to build ladders have $\E$--diameter comparable to the $\E$--distance between their end points.

\begin{lemma}
\label{lem:bounded_between}
There exists a $K_2 \geq 1$ so that for all vertices $v \in \B$ and for all vertices  $x,y \in \E_v$, every geodesic $[x,y]_{v}$ has diameter at most $K_2d_\E(x,y)+K_2$ in $\E$. 
\end{lemma}

\begin{proof}
  Pick two points $z,q \in [x,y]_{v}$. By Axiom \eqref{item:full} there exist tight sections $\Sigma_z$ and $\Sigma_q$ though $z$ and $q$ respectively. There is also a ladder $\ladder(\Sigma_z,\Sigma_q)$ for which $\ladder_v(\Sigma_z,\Sigma_q)$ is the subgeodesic $[z,q]_v$ of $[x,y]_v$.  Consider the fiberwise projection $\Pi \colon \E \to \ladder(\Sigma_z,\Sigma_q)$. Since $z,q \in [x,y]_v$ we can assume that $z \in \Pi(x)$ and $q \in \Pi(y)$. Since $\Pi$ is $K_0$--coarsely Lipschitz (Lemma \ref{lem:fiberwise_proj})  and the inclusion $\iota\colon \ladder(\Sigma_z,\Sigma_q)\to \E$ a is $K_1$--quasi-isometric embedding (Corollary \ref{cor:ladder_qie}), we have 
\[d_\E(z,q) \leq K_1 d_{\ladder(\Sigma_z,\Sigma_q)}(z,q) + K_1 \leq K_0 K_1d_\E(x,y)+  K_0K_1+K_1.\]
Setting $K_2 = K_0K_1+K_1$ proves the lemma.
\end{proof}

Our second construction is the neck of a ladder.  Roughly, this is the subspace of $\B$ over which the sections defining the ladder are close.

\begin{definition}\label{defn:neck}
    Let $\ladder(\Sigma_1,\Sigma_2)$ be a ladder between two tight sections. For $B \geq 0$, the \textbf{\em $B$--neck} of $\ladder(\Sigma_1,\Sigma_2)$ is the following set of vertices of $\B$:
\[\neck_B(\Sigma_1,\Sigma_2)=\big\{v\in \B: d_{\E}(\Sigma_1(v),\Sigma_2(v))\leq B\big\}.\]
This notation reflects the fact that the neck only depends on the sections, and not the ladder.
\end{definition}

We note that the $B$--neck of a given ladder may be empty for (relatively) smaller values of $B$.  Moreover, we emphasize that we think of ladders as quasi-isometrically embedded subspaces of the total space $\E$ (Corollary \ref{cor:ladder_qie}), while necks are subsets of the hyperbolic base $\B$.

The next lemma says, roughly, that thickened necks grow in a controlled way.

\begin{lemma} \label{lem:neck growth}
If  $B\geq \sigma$ as in Axiom \eqref{item:flaring} and $B' \geq B$, then there is $B'' =B''(B',B) \geq 0$ so that for any ladder $\ladder(\Sigma_1,\Sigma_2)$ we have
$$\neck_B(\Sigma_1,\Sigma_2)\subset \neck_{B'}(\Sigma_1,\Sigma_2)\subset \N_{B''}(\neck_{B}\left(\Sigma_1,\Sigma_2)\right).$$
\end{lemma}

\begin{proof}
     The first containment holds by definition. For the second containment, consider a vertex $u\in  \neck_{B'}(\Sigma_1,\Sigma_2) - \neck_B(\Sigma_1,\Sigma_2)$.  Let $v\in \neck_B(\Sigma_1,\Sigma_2)$ be a vertex minimizing the distance to $u$ (such a vertex exists because $\E$ is a simplicial graph), and let $\gamma$ be a $\B$--geodesic $[u,v]$. This ensures that for any vertex $w \in \gamma- \{v\}$, we have $d_\E(\Sigma_1(w),\Sigma_2(w)) > B \geq \sigma$. Thus we can apply Axiom \eqref{item:flaring} to conclude 
    $$B'\geq d_{\E}(\Sigma_1(u),\Sigma_2(u))\geq \lambda (d_\B(u,v))$$
    for the diverging function $\lambda = \lambda(B)\colon \mathbb{R}^+\to \mathbb{R}^+$. This gives an upper bound $B''$ on $d_\B(u,v)$ in terms of $B$ and $B'$. Therefore $u \in \N_{B''}(\neck_{B}(\Sigma_1,\Sigma_2))$ as desired.
\end{proof}

The next two lemmas describe how the necks of multiple ladders interact.

\begin{lemma}
\label{lem:neck_between}
There is $K \geq 1$ so that for any $B\ge 1$ and 
any ladders which are nested $\ladder(\Sigma_2,\Sigma_3)\subset \ladder(\Sigma_1,\Sigma_4)$ as subsets of $\E$,  we have \[ \neck_B(\Sigma_1,\Sigma_4) \subset \neck_{KB}(\Sigma_2,\Sigma_3).\] 
\end{lemma}
\begin{proof}
Let $ v\in \neck_B(\Sigma_1,\Sigma_4)$. The nested assumption implies that $\Sigma_2(v), \Sigma_3(v) \in \ladder_v(\Sigma_1,\Sigma_4)$.  Since $d_\E(\Sigma_1(v), \Sigma_4(v)) \leq B$, Lemma \ref{lem:bounded_between} then implies that $d_\E(\Sigma_2(v),\Sigma_3(v)) \leq K_2B+K_2$. Since $B\ge 1$, it follows that $v \in  \neck_{KB}(\Sigma_2,\Sigma_3)$ for $K= 2K_2$, as required.
\end{proof}

\begin{lemma} \label{lem:3_necks_intersect_2}
If $B \geq \sigma$ as in Axiom \eqref{item:flaring}, then for all $C\geq 0$ there exists $D = D(B,C)\geq 0$ such that the following holds for any triple $\Sigma_1,\Sigma_2,\Sigma_3$ of tight sections:
 \[ \N_C(\neck_B(\Sigma_1,\Sigma_2)) \cap \N_C(\neck_B(\Sigma_2,\Sigma_3)) \subset \N_D\left(\neck_B(\Sigma_1,\Sigma_3)\right).\]
\end{lemma}

\begin{proof}
    By Definition \ref{defn:neck} of necks, we have $\N_C(\neck_B(\Sigma_1,\Sigma_2))\subset \neck_{B+2C}(\Sigma_1,\Sigma_2)$, and similarly for $\Sigma_2,\Sigma_3$. So, if $v\in\mathcal B$ lies in the intersection on the left-hand side, we have $$d_{\mathcal E}(\Sigma_1(v),\Sigma_2(v)), d_{\mathcal E}(\Sigma_2(v),\Sigma_3(v))\leq B+2C.$$ By the triangle inequality we have $d_{\mathcal E}(\Sigma_1(v),\Sigma_3(v))\leq 2(B+2C)$, which implies $v$ is an element of $\neck_{2B+4C}(\Sigma_1,\Sigma_3)$. Lemma \ref{lem:neck growth} now provides the desired $D$, completing the proof.
\end{proof}

We next show that the neck of a ladder is a quasiconvex subset of the hyperbolic space $\B$.

\begin{lemma}
\label{lem:neck_qc} 
    For all $B\geq \sigma$ as in Axiom \eqref{item:flaring}, there exists $Q = Q(B) \geq 0$ so that $\neck_B(\Sigma_1,\Sigma_2)$ is $Q$--quasiconvex for any tight sections $\Sigma_1,\Sigma_2$.
\end{lemma}

\begin{proof}
Let $\lambda = \lambda(B)$ be the function from Axiom (\ref{item:flaring}).
If $\neck_B(\Sigma_1,\Sigma_2) = \emptyset$, then there is nothing to show.  Otherwise, let  $u_0,v_0$ be two vertices of $\neck_B(\Sigma_1,\Sigma_2)$ and let $\alpha = [u_0,v_0]$ be a $\B$--geodesic between them. We will provide a bound $Q = Q(B)$ on the length of any subgeodesic $\gamma = [u,v]$ of $\alpha$ that intersects $\neck_B(\Sigma_1,\Sigma_2)$ only at its endpoints $u,v$. This will suffice, since it implies $\alpha$ itself is contained the $Q$--neighborhood of $\neck_B(\Sigma_1,\Sigma_2)$.

The fact that $v \in \neck_B(\Sigma_1,\Sigma_2)$ implies $d_\E(\Sigma_1(v),\Sigma_2(v)) \leq B$, but our assumption on $\gamma$ means $d_\E(\Sigma_1(w),\Sigma_2(w)) \ge B\geq \sigma$ for all vertices $w \in \gamma$. Hence Axiom \eqref{item:flaring} implies
\[ B\ge d_{\E}(\Sigma_1(u),\Sigma_2(u))\geq \lambda(d_\B(u,v)).\]

Since $\lambda$ is diverging, this bounds $d_\B(u,v)$ in terms of $B$, and we are done. 
\end{proof}

Since necks are quasiconvex, they admit well-behaved closest-point projections.  The following lemma allows us to control, for a triple of tight sections, the projection of one neck to another when the third neck is empty.  It will be useful in Lemma \ref{lem:double_ladder_proj}, where we prove hyperbolicity of certain double-ladders obtained from gluing two ladders together in a natural way.

\begin{lemma}
\label{lem:proj_neck}
     For all $B \geq \sigma$ there exists $P = P(B)\geq 0$ such that the following holds. Let $\Sigma_1,\Sigma_2,\Sigma_3$ be tight sections such that $\neck_B(\Sigma_1,\Sigma_2)\ne\emptyset$, and let $\rho\colon \B\to \neck_B(\Sigma_1,\Sigma_2)$ be the closest-point projection map to this subset. If $\neck_B(\Sigma_1,\Sigma_3) = \emptyset$, then $\rho(\neck_B(\Sigma_2,\Sigma_3))$ has diameter at most $P$ in $\B$.  
\end{lemma}

\begin{proof}
Let $\tau > 0$ be the constant from Lemma~\ref{lem:hyperbolic_BGI}. To ease notation, set $Y= \neck_B(\Sigma_1,\Sigma_2)$ and $Z = \neck_B(\Sigma_2,\Sigma_3)$. Consider any points $z_1,z_2\in Z$ with projections $y_i\in \rho(z_i)$. It suffices to bound $d_\B(y_1,y_2)$ from above in terms of $B$. We may assume, $d_\B(y_1,y_2)>\tau$, for else we are already done. By Lemma~\ref{lem:hyperbolic_BGI} there then exist points $v_1,v_2\in [z_1,z_2]$ with $d_\B(y_i,v_i)\le \tau$. And since $Z$ is $Q$--quasiconvex for some $Q=Q(B)$ (Lemma~\ref{lem:neck_qc}), there exist $u_1,u_2\in Z= \neck_B(\Sigma_2,\Sigma_3)$ with $d_\B(u_i,v_i)\le Q$. By the triangle inequality, we have $d(u_i,y_i)\le Q+\tau$ and thus
\[d_\B(u_1,u_2) \ge d_\B(y_1,y_2) - 2(\tau+Q).\]
Hence, to bound $\diam(\rho(Z))$ it suffices to give a uniform bound on $d_\B(u_1,u_2)$.

Since $u_1,u_2$ lie both in $Z = \neck_B(\Sigma_2,\Sigma_3)$ and in the $(Q+\tau)$--neighborhood of $Y = \neck_B(\Sigma_1,\Sigma_2)$, the definitions and the triangle inequality give 
\[d_\E(\Sigma_1(u_i),\Sigma_3(u_i)) \le 2B + 2(Q + \tau),\qquad\text{for i = 1,2}.\]
However, if $\neck_B(\Sigma_1,\Sigma_3) = \emptyset$, then $d_\E(\Sigma_1(w),\Sigma_2(w))>B\ge \sigma$ for all vertices $w\in B$ and, in particular, for all vertices on a $\B$--geodesic $\gamma = [u_1,u_2]$. We may therefore apply Axiom \eqref{item:flaring} to get a divergent function $\lambda = \lambda(2B+2Q+2\tau)\colon\mathbb{R}^+\to\mathbb{R}^+$ so that
 $$2(B+Q+\tau) \geq d_{\E}(\Sigma_1(u_1),\Sigma_3(u_1))\geq \lambda(d_\B(u_1,u_2)).$$ 
Since $\lambda$ is diverging, this bounds $d_\B(u_1,u_2)$, completing the proof.  
\end{proof}

\subsection{Hyperbolicity of Ladders}

In this subsection, we  prove that all ladders are uniformly hyperbolic. 
Let us first describe a general construction that will aid in the proof. 

Fix a constant $B\ge \sigma$ as in Axiom \eqref{item:flaring} and a ladder $\ladder = \ladder(\Sigma_1,\Sigma_2)$ between tight sections $\Sigma_1,\Sigma_2$. Suppose that $\mathcal{I}$ is a finite set of vertices in $\ladder$ and for each $\ast\in \mathcal{I}$ we select a tight section $\Sigma_\ast$ that goes through $\ast$ and is contained in $\ladder$. For each $\ast\in \mathcal{I}$, let $\B_\ast$ denote a copy of $\B$ and $\mu_\ast\colon \B\to \B_\ast$ the simplicial identification; for $v\in \B$ we will write $v_\ast = \mu_\ast(v)$ for the corresponding vertex in $\B_\ast$. Observe that for all pairs $x,y\in \mathcal{I}$, the fact that the sections $\Sigma_x,\Sigma_y$ are contained in $\ladder$ implies that $\neck_B(\Sigma_1,\Sigma_2)\subset \neck_{KB}(\Sigma_x,\Sigma_y)$, where $K$ is the constant from Lemma \ref{lem:neck_between}. 

Now build a graph $\Lambda_{\mathcal{I}}$ by taking the disjoint union $\bigsqcup_{\mathcal{I}} \B_\ast$ and, for each pair $x,y\in \mathcal{I}$, inserting an edge attaching $v_x$ to $v_y$ for each vertex $v\in \neck_{KB}(\Sigma_x,\Sigma_y)$. We also define a coarse graph map $\xi_\mathcal{I}\colon \Lambda_{\mathcal{I}}\to \ladder$ which on the subset $\B_\ast$ is simply the simplicial embedding $\Sigma_\ast\circ\mu_\ast^{-1}$ onto the subset $\Sigma_\ast\subset \ladder$. This defines $\xi_{\mathcal{I}}$ on all vertices, hence giving a coarse graph map.  Observe that every subset $\mathcal{J}\subset \mathcal{I}$ determines its own graph $\Lambda_{\mathcal{J}}$ that comes with a simplicial embedding $\iota_{\mathcal{J}}^{\mathcal{I}}\colon \Lambda_{\mathcal{J}}\to \Lambda_{\mathcal{I}}$ that, by construction, satisfies $\xi_{\mathcal{J}} = \xi_{\mathcal{I}}\circ \iota_{\mathcal{J}}^{\mathcal{I}}$.

\begin{lemma}\label{lem:glued_graphs_hyp_and_qi_embedded}
Assume the ladder $\ladder = \ladder(\Sigma_1,\Sigma_2)$ has nonempty $B$--neck $\neck_{B}(\Sigma_1,\Sigma_2)$, where $B\ge \sigma$. For any finite set of vertices $\mathcal{I} \subset \ladder$ and  collection of tight sections $\{\Sigma_\ast\}_{\ast \in \mathcal{I}}$ as above, there exist constants $\delta',K'$, depending only on $B$, the cardinality of $\mathcal{I}$, and the bundle constants of $\E$, such that:
\begin{enumerate}
\item \label{item:coarsely_Lipschitz} $\xi_{\mathcal{I}}\colon \Lambda_{\mathcal{I}}\to \ladder$ is $K'$--coarsely Lipschitz,
\item \label{item:unif_hyp} $\Lambda_{\mathcal{I}}$ is $\delta'$--hyperbolic, and
\item \label{item:qi_inclusion} the inclusion $\Lambda_{\mathcal{J}}\to \Lambda_{\mathcal{I}}$ is a $K'$--quasi-isometric embedding for each $\mathcal{J}\subset \mathcal{I}$.
\end{enumerate}
\end{lemma}
\begin{proof}
For (1), note that edges in $\B_\ast$ map to edges in $\ladder$, and that added edges between $\B_x$ and $\B_y$ map  to points with $\E$--distance at most $KB$ and hence, by Corollary \ref{cor:ladder_qie}, $\ladder$--distance at most $K_1(KB)+K_1$. Thus $\xi_{\mathcal{I}}$ is $K'$--coarsely Lipschitz for $K' =K_1(KB) +K_1$.

The proof of (2) and (3) is by induction on $\left\vert\mathcal{I}\right\vert$. The base case $\left\vert\mathcal{I}\right\vert = 1$ is immediate, since then $\Lambda_\mathcal{I} = \B_\ast$ is isometric to the hyperbolic space $\B$ and there are only trivial subsets $\mathcal{J} = \emptyset$ or $\mathcal{J} = \mathcal{I}$.

Let us assume $\left\vert\mathcal{I}\right\vert > 1$ and that $\mathcal{J}\subset \mathcal{I}$ is a proper nontrivial subset. Choose $\ast\in \mathcal{I}\setminus\mathcal{J}$ and let $\mathcal{H} = \mathcal{I}\setminus\{\ast\}$. By induction we may assume the inclusion $\Lambda_{\mathcal{J}}\to \Lambda_{\mathcal{H}}$ is a $K'$--quasi-isometric embedding between spaces that are both $\delta'$--hyperbolic. Since $\iota_{\mathcal{J}}^{\mathcal{I}}$ equals the composition $\iota_{\mathcal{H}}^{\mathcal{I}}\circ\iota_{\mathcal{J}}^{\mathcal{H}}$, it now suffices to prove $\Lambda_{\mathcal{I}}$ is uniformly hyperbolic and that  $\iota_{\mathcal{H}}^{\mathcal{I}}$ is a uniform quasi-isometric embedding. 

For this we will apply Proposition \ref{prop:gluing_lemma} as follows: 
For each $x\in \mathcal{H}$, let
\[N'_x = \{v_x\in \B_x \mid v\in \neck_{KB}(\Sigma_x,\Sigma_\ast)\} \subset \B_x.\]
Observe that $N'_x$ is $Q$--quasiconvex in $\B_x$ by Lemma \ref{lem:neck_qc}, where $Q\ge0$ is chosen so that all $B$ and $KB$ necks are $Q$--quasiconvex. Since the inclusion $\B_x = \Lambda_{\{x\}}\to \Lambda_{\mathcal{H}}$ is a $K'$--quasi-isometric embedding into a $\delta'$--hyperbolic space (by induction), Lemma \ref{lem:quasi-convex_perserved} implies the image $N_x = \iota_{\{x\}}^{\mathcal{H}}(N'_x)$ is uniformly quasiconvex in $\Lambda_{\mathcal{H}}$. Furthermore, observe that for each pair $x,y\in \mathcal{H}$, the sets $N_x$ and $N_y$ have distance at most $1$ in $\Lambda_{\mathcal{H}}$. Indeed, choose any $v\in \neck_{B}(\Sigma_1,\Sigma_2)$, which is nonempty by assumption; then Lemma \ref{lem:neck_between} implies we have $v_x\in N_x'$ and $v_y\in N_y'$ and, moreover that these are joined by an edge in $\Lambda_{\mathcal{H}}$ by construction since $v\in \neck_{KB}(\Sigma_x,\Sigma_y)$. Therefore, inductively applying Lemma \ref{lem:union_of_QC_subsets}, we conclude that the union $N = \bigcup_{x\in \mathcal{H}}N_x$ is a $Q'$--quasiconvex subset of $\Lambda_{\mathcal{H}}$ for some constant $Q'$ depending only on $B$ and $\left\vert\mathcal{H}\right\vert$ and the constants of the bundle $\E$. Now consider the map $\psi\colon N\to \B_\ast$ defined by $\psi(v_x) = v_\ast$ for each vertex $v_x\in N_x$.

\begin{claim}\label{claim:gluing_psi_a_qi}
The coarse graph map $\psi\colon N\to \B_\ast$ is a uniform quasi-isometric embedding, where $N$ is given the restricted metric from $\Lambda_{\mathcal{H}}$.
\end{claim}
\begin{proof}[Proof of Claim \ref{claim:gluing_psi_a_qi}]
Let us fix vertices $u,v\in N$. We may suppose $u\in N_x$ and $v\in N_y$ where $x,y\in \mathcal{H}$.
First observe that $\psi$ is $1$--Lipschitz. Indeed,  by the triangle inequality we may suppose $u$ and $v$ are joined by an edge in $\Lambda_{\mathcal{H}}$. Then  either $x = y$, so that $u,v$ are adjacent in $\B_x$ and their images are $\psi(u),\psi(v)$ are also adjacent in $\B_\ast\cong \B_x$, or else $x\ne y$, so that  by construction of $\Lambda_{\mathcal{H}}$ we have $u = w_x$ and $v = w_y$ for some $w\in \neck_{KB}(\Sigma_x,\Sigma_y)$ which implies $\psi(u) = w_\ast = \psi(v)$ in $\B_\ast$.

It remains to bound $d_{\Lambda_{\mathcal{H}}}(u,v)$ from above in terms of $d_{\B_\ast}(\psi(u),\psi(v))$. For this, let us set $u_0= \psi(u), v_0 = \psi(v)$ and view these as points in $\B \cong \B_\ast$. Let $A = \neck_{B}(\Sigma_1,\Sigma_2)$, which is a nonempty $Q$--quasiconvex subset of $\B$, and choose closest vertices $u'\in \rho_A(u_0)$ and $v' \in \rho_A(v_0)$. Note that by construction of $N_x,N_y$ we have $u_0\in \neck_{KB}(\Sigma_x,\Sigma_\ast)$ and $v_0\in \neck_{KB}(\Sigma_y,\Sigma_\ast)$.

First suppose $d_\B(u',v') > \tau$, where $\tau\ge 1$ is from Lemma \ref{lem:hyperbolic_BGI}. In this case that lemma implies $u'\in \nbhd_\tau([u_0,v_0])$. That is, the vertex $w'=u'\in A\subset \neck_{KB}(\Sigma_x,\Sigma_y)$ satisfies $d_\B(w',[u_0,v_0])\le \tau$.

Otherwise $d_\B(u',v')\le \tau$ and Lemma \ref{lem:thin_rectangles} implies there are points $u''\in [u',u_0]$, $v''\in [v',v_0]$ and $w\in [u_0,v_0]$ with pairwise distance at most $\nu = \nu(\delta,\tau)$. Recall that $A = \neck_{B}(\Sigma_1,\Sigma_2)$ is contained in both $\neck_{KB}(\Sigma_x,\Sigma_\ast)$ and in $\neck_{KB}(\Sigma_y,\Sigma_\ast)$. Since, $KB$--necks are $Q$-quasiconvex, we have
\[u''\in [u',u_0]\subset \nbhd_{Q}(\neck_{KB}(\Sigma_x,\Sigma_\ast)) \quad\text{and}\quad
v''\in [v',v_0]\subset \nbhd_Q(\neck_{KB}(\Sigma_y,\Sigma_\ast)).\]
Since $d_\B(u'',v'')\le \nu$, it now follows from Lemma \ref{lem:3_necks_intersect_2} that
\[u''\in \nbhd_{Q+\nu}(\neck_{KB}(\Sigma_x,\Sigma_\ast))\cap \nbhd_{Q+\nu}(\neck_{KB}(\Sigma_y,\Sigma_\ast))
\subset \nbhd_D(\neck{KB}(\Sigma_x,\Sigma_y)),\]
where $D$ depends only on $KB$, and $Q+\nu$ and the bundle constants of $\E$. This means there exists a vertex $w'\in \neck_{KB}(\Sigma_x,\Sigma_y)$ with $d_\B(w',u'')\le D$ and thus $d_\B(w',[u_0,v_0]))\le D+\nu$. 

In either case, we now have found a vertex $w'\in \neck_{KB}(\Sigma_x,\Sigma_y)$ with $d_\B(w',[u_0,v_0])\le D+\nu+\tau$. It follows that
\[d_\B(u_0,v_0) \ge d_\B(u_0,w') + d_\B(w',v_0) - 2(D+\nu+\tau).\]
However, the fact that $w'\in \neck_{KB}(\Sigma_x,\Sigma_y)$ means that $w'_x\in \B_x$ and $w'_y\in \B_y$ are joined by an edge in $\Lambda_{\mathcal{H}}$ (or are the same vertex if $x=y$). Therefore, by the triangle inequality, in $\Lambda_{\mathcal{H}}$ we have
\begin{align*}
d_{\Lambda_{\mathcal{H}}}(u,v) &\le d_{\Lambda_{\mathcal{H}}}(u,w'_x) + d_{\Lambda_{\mathcal{H}}}(w'_x,w'_y) + d_{\Lambda_{\mathcal{H}}}(w'_y,v)
\le d_{\B_x}(u,w'_x)+1+d_{\B_y}(w'_y,v)\\
&= d_{\B}(u_0,w') + 1 + d_{\B}(w',v_0) \le d_{\B_\ast}(\psi(u),\psi(v)) + 2(D+\nu+\tau)+1.\qedhere
\end{align*}
\end{proof}
With Claim \ref{claim:gluing_psi_a_qi} in hand, we may now complete the proof of Lemma \ref{lem:glued_graphs_hyp_and_qi_embedded}. Since $\Lambda_{\mathcal{H}}$ is hyperbolic and $\psi\colon N\to \B_\ast$ is a quasi-isometric embedding, Proposition \ref{prop:gluing_lemma} implies $\Lambda_{\mathcal{H}}\cup_\psi \B_\ast$ is a uniformly hyperbolic space in which $\Lambda_{\mathcal{H}}$ is uniformly quasi-isometrically embedded. However, by construction, $\Lambda_{\mathcal{H}}\cup_\psi \B_\ast$ is precisely the space $\Lambda_{\mathcal{I}}$ and the inclusion is just the map $\iota_{\mathcal{H}}^{\mathcal{I}}$. Hence we are done.
\end{proof}
  
With the aid Lemma \ref{lem:glued_graphs_hyp_and_qi_embedded}, we may now prove hyperbolicity for ladders with nonempty $B$--necks.

\begin{proposition}
\label{prop:thin_neck}
For all $B \ge B_0 = \max\{\sigma,\kappa+1\}$ as in Axioms \eqref{item:retraction} and \eqref{item:flaring}, there exists $\epsilon = \epsilon(B)>0$ such that any ladder $\ladder(\Sigma_1,\Sigma_2)$ with $\neck_B(\Sigma_1,\Sigma_2)\neq \emptyset$ is $\epsilon$--hyperbolic and the sections $\Sigma_1,\Sigma_2$ are both $\epsilon$--quasiconvex subsets.
\end{proposition}
\begin{proof}
We establish hyperbolicity by constructing the following system of preferred paths: For each vertex $x\in \ladder(\Sigma_1,\Sigma_2)$, use Axiom \eqref{item:robust} to choose a tight section $\Sigma_x\subset\ladder(\Sigma_1,\Sigma_2)$ through $x$. Now, given vertices $x,y\in \ladder(\Sigma_1,\Sigma_2)$, take the index set $\mathcal{I} = \{x,y\}$ and build the map $\xi_{\mathcal{I}}\colon \Lambda_{\mathcal{I}}\to \ladder(\Sigma_1,\Sigma_2)$ as above. Let $\hat{x}$ and $\hat{y}$ be the vertices of $\Lambda_{\mathcal{I}}$ so that $\xi_{\mathcal{I}}(\hat{x}) = x$ and $\xi_{\mathcal{I}}(\hat{y}) = y$. We define the preferred path $\varsigma(x,y)$ in $\ladder(\Sigma_1,\Sigma_2)$ to be the image under $\xi_{\mathcal{I}}$ of a $\Lambda_{\mathcal{I}}$--geodesic $[\hat{x},\hat{y}]$. This is a coarsely connected subset as it is the image of a connected subset under a coarsely Lipschitz map. 

\begin{claim}\label{claim:good_preferred_paths}
There exists  $D =D(B,\E)\geq 0$ so that:
\begin{enumerate}
\item 
$\diam_{\ladder(\Sigma_1,\Sigma_2)}(\varsigma(x,y)) \leq D$ whenever $x,y\in \ladder(\Sigma_1,\Sigma_2)$ are joined by an edge of $\ladder(\Sigma_1,\Sigma_2)$
\item $\varsigma(x,y)\subset \N_D\left(\varsigma(x,z)\cup\varsigma(z,y)\right)$ for all $x,y,z\in \ladder(\Sigma_1,\Sigma_2)$.
\end{enumerate}
\end{claim}
\begin{proof}
For (1), Since $\xi_{\mathcal{I}}$ is uniformly coarsely Lipschitz by item \eqref{item:coarsely_Lipschitz} of Lemma \ref{lem:glued_graphs_hyp_and_qi_embedded}, it suffices to bound $d_{\Lambda_{\mathcal{I}}}(\hat{x},\hat{y})$. 
If $x$ and $y$ both lie in $\E_v$ for some $v \in \B$, then $v\in \neck_{1}(\Sigma_x,\Sigma_y) \subset \neck_{KB}(\Sigma_x,\Sigma_y)$, thus forcing $\hat{x} = v_x$ and $\hat{y}=v_y$ to be joined by an edge of $\Lambda_{\mathcal{I}}$; hence $d_{\Lambda_{\mathcal{I}}}(\hat{x},\hat{y})\le 1$. Otherwise, $x \in \E_u$ and $y \in \E_v$ are in distinct fibers with $d_\B(u,v) = 1$ and $d_\E(x,y) \leq \kappa$ for $\kappa$ as in Axiom \eqref{item:retraction}. Then
\[d_{\E}(\Sigma_x(u),\Sigma_y(u)) = d_{\E}(x,\Sigma_y(u))\leq d_{\E}(x,y)  + d_{\E}(\Sigma_y(v),\Sigma_y(u))\leq \kappa +1.\]
 Since $B \geq \kappa +1$, this implies $u \in \neck_{KB}(\Sigma_x,\Sigma_y)$ so that $u_x$ and $u_y$ are joined by an edge of $\Lambda_{\mathcal{I}}$ by construction. Thus $\hat{x} = u_x$ and $\hat{y} =v_y$ are 2 apart in $\Lambda_{\mathcal{I}}$, proving $d_{\Lambda_{\mathcal{I}}}(\hat{x},\hat{y})$ is uniformly bounded.

For (2), now set $\mathcal{I} = \{x,y,z\}$. Let $\bar{x}$ denote the subset $\{y,z\}$, and let $\hat{\gamma}_{\bar{x}}$ denote the chosen geodesic in $\Lambda_{\bar{x}}$ with $\xi_{\bar{x}}(\hat{\gamma}_{\bar{x}}) = \varsigma(y,z)$ (that is, $\hat{\gamma}_{\bar{x}}$ is a geodesic between the points that map to $y$ and $z$ under $\xi_{\bar{x}}$). Define $\bar{y}$, $\hat{\gamma}_{\bar{y}}$,  $\bar{z}$, and $\hat{\gamma}_{\bar{z}}$ symmetrically. 
By item \eqref{item:qi_inclusion} of Lemma \ref{lem:glued_graphs_hyp_and_qi_embedded}, for each $\ast\in \mathcal{I}$, we have a uniform quasi-isometric embedding $\iota_{\bar{\ast}}= \iota_{\bar{\ast}}^{\mathcal{I}}\colon \Lambda_{\bar{\ast}}\to \Lambda_{\mathcal{I}}$ that sends the geodesic $\hat{\gamma}_{\bar{\ast}}$ to a uniform quasigeodesic $\gamma_{\bar{\ast}}$ in $\Lambda_{\mathcal{I}}$. Since $\Lambda_{\mathcal{I}}$ is uniformly hyperbolic by item \eqref{item:unif_hyp} of Lemma \ref{lem:glued_graphs_hyp_and_qi_embedded}, Lemma \ref{lem:morse} implies these three quasigeodesics form a thin triangle. That is, there exists $D'$ so that $\gamma_{\bar{z}}\subset \nbhd_{D'}(\gamma_{\bar{y}}\cup \gamma_{\bar{x}})$. But by construction we have $\varsigma(x,y) = \xi_{\mathcal{I}}(\gamma_{\bar{z}})$ (since $\xi_{\bar{z}} = \xi_{\mathcal{I}}\circ \iota_{\bar{z}}^{\mathcal{I}}$), and similarly $\varsigma(y,z)=\xi_{\mathcal{I}}(\gamma_{\bar{x}})$ and $\varsigma(z,x)=\xi_{\mathcal{I}}(\gamma_{\bar{y}})$. The desired conclusion therefore follows from the fact that $\xi_{\mathcal{I}}\colon \Lambda_{I}\to \ladder(\Sigma_1,\Sigma_2)$ is uniformly coarsely Lipschitz by item \eqref{item:coarsely_Lipschitz} of Lemma \ref{lem:glued_graphs_hyp_and_qi_embedded}.
\end{proof}
The claim shows that Proposition \ref{prop:guessing} applies to the family of path $\varsigma(x,y)$ and thus proves that $\ladder(\Sigma_1,\Sigma_2)$ is uniformly hyperbolic. It also implies the sections $\Sigma_1,\Sigma_2$ are uniformly quasiconvex. 
Indeed, if $x,y\in \Sigma_i$, we may take $\Sigma_x = \Sigma_i = \Sigma_y$ so that the preferred path $\varsigma(x,y)$ lies in $\Sigma_i$. Proposition \ref{prop:guessing} then ensures the geodesic $[x,y]$ in $\ladder(\Sigma_1,\Sigma_2)$ lies within uniform Hausdorff distance of $\varsigma(x,y)$ and thus within a uniform neighborhood of $\Sigma_i$.
\end{proof}

We next use hyperbolicity of ladders with nonempty necks to prove that \emph{all} ladders are in fact uniformly hyperbolic. 
To this end, consider tight sections $\Sigma_1,\Sigma_2,\Sigma_3$ so that $\Sigma_2$ is contained in $\ladder(\Sigma_1,\Sigma_3)$. Define the \emph{double ladder}, $\Dladder(\Sigma_1,\Sigma_2,\Sigma_3)$, to be the graph obtained by taking the disjoint union of $\ladder(\Sigma_1,\Sigma_2)$ and $\ladder(\Sigma_2,\Sigma_3)$ and adding edges joining each vertex of $\Sigma_2$ in $\ladder(\Sigma_1,\Sigma_2)$ to the corresponding vertex of $\Sigma_2$ in $\ladder(\Sigma_2,\Sigma_3)$.

\begin{lemma}\label{lem:double_ladder_proj}
For all $B \ge \max\{\sigma,\kappa+1\}$ as in Axioms \eqref{item:retraction} and \eqref{item:flaring}, there exists $\epsilon,D,Q \geq 1$ depending only on $B$ and $\E$ so that the following holds. Let $\Sigma_1,\Sigma_2,\Sigma_3$ be tight sections so that $\Sigma_2\subset \ladder(\Sigma_1,\Sigma_3)$.
     If $\neck_B(\Sigma_1,\Sigma_2) \neq \emptyset$ and $\neck_B(\Sigma_2,\Sigma_3) \neq \emptyset$, then $\Dladder(\Sigma_1,\Sigma_2,\Sigma_3)$ is $\epsilon$--hyperbolic and each copy of $\Sigma_1,\Sigma_2,\Sigma_3$ in $\Dladder(\Sigma_1,\Sigma_2,\Sigma_3)$ is $Q$--quasiconvex. Moreover, if additionally $\neck_B(\Sigma_1,\Sigma_3) = \emptyset$, then the closet point projection of $\Sigma_3$ onto $\Sigma_1$ in $\Dladder(\Sigma_1,\Sigma_2,\Sigma_3)$ has diameter at most $D$.
\end{lemma}

\begin{proof}
Proposition \ref{prop:thin_neck} implies $\ladder(\Sigma_1,\Sigma_3)$ and $\ladder(\Sigma_2,\Sigma_3)$ are both $\epsilon'$--hyperbolic and that $\Sigma_2$ is $\epsilon'$--quasiconvex in each, for some constant $\epsilon'$ depending on $B$. Lemma \ref{prop:gluing_lemma} now ensures that $\Dladder(\Sigma_1,\Sigma_2,\Sigma_3)$ is $\epsilon$--hyperbolic and that $\ladder(\Sigma_1,\Sigma_2)$ and $\ladder(\Sigma_2, \Sigma_3)$ are each $K'$--quasi-isometrically embedded in it, for some constants $\epsilon,K'$ depending only on $B$.  Since $\Sigma_i$ is $\epsilon$'--quasiconvex in $\ladder(\Sigma_i,\Sigma_j)$ when $\abs{i-j} =1$, Lemma \ref{lem:quasi-convex_perserved} therefore implies $\Sigma_1,\Sigma_2,\Sigma_3$ are each $Q$--quasiconvex in $\Dladder(\Sigma_1,\Sigma_2,\Sigma_3)$ for some $Q$ determined again only by $B$.

    To prove the ``moreover'' part of the statement, it will be convenient to work with a related object.  Let $\Dladder'$ be the space obtained from  $\Dladder(\Sigma_1,\Sigma_2,\Sigma_3)$  by collapsing each of the edges between a vertex of $\ladder(\Sigma_1,\Sigma_2)$ and a vertex of $\ladder(\Sigma_2,\Sigma_3)$ to a single vertex. Equivalently, $\Dladder'$ is the space obtain from  $\ladder(\Sigma_1,\Sigma_2)$ and $\ladder(\Sigma_2,\Sigma_3)$ by identifying the vertices in the two copies of $\Sigma_2$. Note, $\Dladder'$ is $2$--quasi-isometric to $\Dladder(\Sigma_1,\Sigma_2,\Sigma_3)$, thus it suffices to prove the ``moreover'' part of the statement for $\Dladder'$ instead. Similarly, by increasing $\epsilon$ and $Q$ a uniform amount, we have that  $\Dladder'$ is $\epsilon$--hyperbolic and each $\Sigma_i$ is $Q$--quasiconvex in $\Dladder'$.

    Assume $\neck_B(\Sigma_1,\Sigma_3) = \emptyset$. For distinct $i,j \in \{1,2,3\}$, let $N^j_{i}$  be the set of vertices $x$ of $\Sigma_i$ where $x \in \E_v$ for some $v \in \neck_B(\Sigma_i,\Sigma_j)$. Let $\rho_i \colon \Dladder' \to \Sigma_i$ be the closest point projection map.

    \begin{claim}
        For each $i,j \in \{1,2,3\}$ with $|i-j| = 1$, every geodesic in $\Dladder'$ from $\Sigma_i$ to $\Sigma_j$ passes $r$--close to $N_i^j$ and $N_j^i$, where $r$ is determined by $B$ and the constants of $\E$.
    \end{claim}

    \begin{proof}
       Let $x \in \Sigma_i$ and $y \in \Sigma_j$. In the proof that $\ladder(\Sigma_i,\Sigma_j)$ is hyperbolic (Proposition \ref{prop:thin_neck}), the preferred path $\varsigma(x,y)$ we build between $x$ and $y$ travels along $\Sigma_i$ until it reaches a point $\Sigma_i(v)$ where $v \in \neck_{KB}(\Sigma_i,\Sigma_j)$, then jumps across to $\Sigma_j(v)$ before traveling to $y$ along $\Sigma_j$.  Lemma \ref{lem:neck growth} tells us that there is some $B''>0$ (depending only on $K$ and $B$, and hence only on $B$ and $\E$) so that $\Sigma_i(v) \in \N_{B''}(N_i^j)$ and $\Sigma_j(v) \in \N_{B''}(N_j^i)$. Since the preferred path $\varsigma(x,y)$ ends up being uniformly close to a geodesic between $x$ and $y$ in $\ladder(\Sigma_i,\Sigma_j)$ (Proposition \ref{prop:guessing}) and $\ladder(\Sigma_i,\Sigma_j)$ is uniformly quasi-isometrically embedded in $\Dladder'$, the claim is true.
    \end{proof}

    To complete the proof, let $x \in \Sigma_3$ and $y \in \Sigma_1$ and let $\gamma$ be any geodesic between them in $\Dladder'$. Since $\gamma$ must pass through $\Sigma_2$, the claim ensures that $\gamma$ passes $r$--close to $N_2^3$ and $N_2^1$. Now Lemma \ref{lem:proj_neck} ensures that the projection $\rho_{N_2^1}(N_2^3)$ of $N_2^3$ onto $N_2^1$ has uniformly bounded diameter (controlled by $B$ and $\E$), and hence any geodesic between $N_2^1$ and $N_2^3$ must pass uniformly close to $\rho_{N_2^1}(N_2^3)$ by Lemma \ref{lem:hyp pass close}.  Hyperbolicity then forces $\gamma$ to passes uniformly close to $\rho_{N_2^1}(N_2^3)$, and it follows that the diameter of the full projection must also be uniformly bounded.    
\end{proof}

\begin{proposition}
    There exists $\epsilon>0$ depending only on $\E$ such that for any two tight sections $\Sigma,\Sigma'$ in $\E$, any ladder $\ladder(\Sigma,\Sigma')$ is $\epsilon$--hyperbolic.
\end{proposition}

\begin{proof}
    Let $\Sigma,\Sigma'$ be tight sections, and consider a ladder $\ladder(\Sigma,\Sigma')$. Fix a constant $B > \max\{\sigma,\kappa+1\}$ that is  large enough to apply Proposition \ref{prop:thin_neck} and Lemma \ref{lem:double_ladder_proj}. Such a choice only depends on the bundle constants for $\E$. Let $\delta$ be the hyperbolicity constant from Proposition \ref{prop:thin_neck} for this  fixed $B$.
    
We now recursively construct a sequence of tight sections $\Sigma_0,\dots, \Sigma_\ell$ in $\ladder(\Sigma,\Sigma')$ beginning with $\Sigma_0 = \Sigma$ and terminating with $\Sigma_\ell = \Sigma'$. These will satisfy the following for any $0 \le i\le j\le k\le \ell$:
\begin{itemize}
\item $\Sigma_j$ is contained in the subladder $\ladder(\Sigma_i,\Sigma_k)$ of $\ladder(\Sigma,\Sigma')$ determined by $\Sigma_i,\Sigma_k$.
\item If $\abs{i-k} =1$, then $\neck_B(\Sigma_i,\Sigma_k) \ne \emptyset$, and hence the subladder $\ladder(\Sigma_i,\Sigma_k)$ is $\delta$--hyperbolic.
\item If $\abs{i-k} > 1$, then $\neck_B(\Sigma_i,\Sigma_k) = \emptyset$.
\end{itemize}
To begin, fix a vertex $b\in \B$ and set $\Sigma_0 = \Sigma$. Now for $k \ge 0$ recursively assume $\Sigma_0,\dots,\Sigma_{k}$ have been chosen. We assume $\Sigma_k \ne \Sigma'$, for otherwise we terminate the process with $k = \ell$. Consider the subladder $\ladder(\Sigma_k,\Sigma')$ of $\ladder(\Sigma,\Sigma')$. If $\neck_B(\Sigma_{k},\Sigma') \ne \emptyset$, then set $\Sigma_{k+1} = \Sigma'$. Otherwise, the subgeodesic $\ladder_b(\Sigma_{k},\Sigma')$ of $\ladder_b(\Sigma,\Sigma')$ must have positive length. Therefore, since $B\ge 1$, Axiom \eqref{item:robust} ensures there exist sections $\Sigma''$ in $\ladder(\Sigma_{k},\Sigma')$ satisfying the dual conditions that $\Sigma''(b) \ne \Sigma_{k}(b)$ and $\neck_{B}(\Sigma_{k},\Sigma'')\ne \emptyset$. Hence, among all such sections, we may choose one maximizing the (positive) length of the geodesic $[\Sigma_{k}(b),\Sigma''(b)]_b$ of $\ladder_b(\Sigma_k,\Sigma')$ and then declare $\Sigma_{k+1} = \Sigma''$. In this way we produce a sequence $\Sigma_0,\Sigma_1,\dots$ that necessarily terminates with $\Sigma_\ell = \Sigma'$ since the lengths of the geodesics $\ladder_b(\Sigma_k,\Sigma')$ are strictly monotonically decreasing. The first two bullets are automatic from the construction, and the third follows from the maximality condition in the choice of $\Sigma_{k+1}$.

    Let $\ladder_k$ denote the subladder $\ladder(\Sigma_{k-1},\Sigma_{k})$ and construct a space $\ladder'$ by taking the disjoint union of $\ladder_1,\dots, \ladder_\ell$ and adding an edge between the two copies of each vertex of $\Sigma_k(\B)$ that are contained in $\ladder_{k-1}$ and $\ladder_{k}$ for each $k \in\{2,\dots \ell\}$.

\begin{claim}\label{claim:L'_hyp}
    $\ladder'$ is $\epsilon$--hyperbolic for some uniform $\epsilon$.
\end{claim}

\begin{proof}
    We will show that $\ladder'$ is a line of hyperbolic spaces and then apply Proposition \ref{prop:malnormal_glue}.

    For each $k \in \{1,\dots,\ell-1\}$, let $\Dladder_k$ be the subset of $\ladder'$ spanned by $\ladder_{k}$ and $\ladder_{k+1}$. Notice that $\Dladder_k$ is precisely the double ladder $\Dladder(\Sigma_{k-1},\Sigma_k,\Sigma_{k+1})$ considered in Lemma \ref{lem:double_ladder_proj}. Hence each $\Dladder_k$ is $\delta_3$--hyperbolic for some $\delta_3$ depending on constants from the bundle. Hence, we can view $\ladder'$ as a line of $\delta_3$--hyperbolic spaces as follows:
    \begin{itemize}
        \item  if $\ell$ is even, then the vertex spaces are $\Dladder_1,\Dladder_3,\dots,\Dladder_{\ell-1}$ and the edges spaces are $\Sigma_2,\Sigma_4, \dots, \Sigma_{\ell-2}$ contained in $\ladder_2,\ladder_4,\dots,\ladder_{\ell-2}$;
        \item if $\ell$ is odd, then  the vertex spaces are $\Dladder_1,\Dladder_3,\dots,\Dladder_{\ell-2}, \ladder_\ell$ and the edges spaces are the copies of $\Sigma_2,\Sigma_4, \dots, \Sigma_{\ell-2}$ contained in $\ladder_2,\ladder_4,\dots,\ladder_{\ell-1}$.
    \end{itemize}

Notice that each edge space is uniformly quasi-isometrically embedded in the total space $\ladder'$, due to the fact that $\ladder'$ admits a coarsely Lipschitz retraction onto the edge space given by mapping to $\B$ and then using the tight section defining the edge space.
Furthermore, in either case, Lemma \ref{lem:double_ladder_proj} ensures that the two edge spaces in each vertex space have uniformly bounded closest-point projections to each other. Therefore this line of hyperbolic spaces satisfies the hypotheses of Proposition \ref{prop:malnormal_glue}, and we conclude that $\ladder'$ is $\epsilon$--hyperbolic where $\epsilon$ ultimately depend only on the bundle constants.
    \end{proof}

Now, Let $\ladder$ be the ladder $\ladder(\Sigma,\Sigma')=\ladder(\Sigma_0,\Sigma_\ell)$. The remainder of the proof is to show that $\ladder$ and $\ladder'$ are uniformly quasi-isometric.

There is a natural $1$--Lipchitz map $\ladder'\to \ladder$ obtained by mapping the edge joining the the two copies of a vertex of $\Sigma_k(\B)$ in $\ladder'$ to the single copy of that vertex in $\ladder$. This map is surjective on the vertices of $\ladder$ and the preimage of each vertex is either a vertex or an edge. To conclude that this is a quasi-isometry with uniform constants, it suffices to argue that if two vertices of $\ladder$ are joined by an edge, then the corresponding vertices of $\ladder'$ lie within bounded distance of each other. 

Let $x,y$ be two vertices of $\ladder$ that are joined by an edge.  
Without loss of generality, assume $x \in \ladder_{i}$ and $y \in \ladder_{j}$ for some $i\leq j$. Let $x'$ and $y'$ be the copies of the vertices $x$ and $y$ in the copies of $\ladder_i$ and $\ladder_j$ contained in $\ladder'$. We need to show that $d_{\ladder'}(x',y')$ is uniformly bounded.

 Being joined by an edge of $\ladder$ means that $d_\E(x,y) \leq \kappa$ and $x$ and $y$ are either in the same fiber of $\E$ or adjacent fibers. If $x$ and $y$ are in the same fiber, let $z = y$, otherwise, Axiom \eqref{item:robust} says  there is a vertex $z \in \ladder_j$ that is in the same fiber as $x$ and joined by an edge of $\ladder_j$ to $y$. Let $z'$ be the copy of the vertex $z$ that is in the copy of $\ladder_j$ in $\ladder'$. Observe that $d_{\ladder'}(z',y')\le 1$.

 Now $d_\E(x,z) \leq d_\E(x,y) + d_\E(y,y') \leq \kappa +1 \leq B$. Thus $\neck_B(\Sigma_i,\Sigma_j) \neq \emptyset$. By construction of the $\Sigma_k$, this implies that $|i -j| \leq 1$. If $i =j$,  then $x'$ and $y'$ are both in the copy of $\ladder_i$ in $\ladder'$. Thus $d_\E(x,y) = d_\E(x',y')\leq \kappa$ implies that $x'$ and $y'$ are either equal or joined by an edge of $\ladder_i$ inside of $\ladder'$.
 
 Now assume $i =j+1$, so that $z \in \ladder_{i+1}$. Let $\E_w$ be the fiber containing $x$ and $z$. Since $d_\E(x,z) \leq B$, Lemma \ref{lem:bounded_between} implies there is $K_2\geq 1$ depending only on the bundle constants so that  $$d_\E(x,\Sigma_{i}(w)) \leq K_2B+K_2 \text{ and } d_\E(\Sigma_{i}(w),z)  \leq K_2B+K_2.$$ Since $\ladder_i$ and $\ladder_{i+1}$  uniformly quasi-isometrically embedded in $\E$, this produces a uniform bound on the distance between $x'$ and $z'$ in $\ladder'$. 
Since $d_{\ladder'}(z',y')\le 1$, as noted above, we obtain the desired uniform bound on $d_{\ladder'}(x',y')$.

Since $\ladder$ and $\ladder'$ are uniformly quasi-isometric, Claim \ref{claim:L'_hyp} makes $\ladder$ uniformly hyperbolic.
\end{proof}

\subsection{Half-ladders, tripod bundles, and the hyperbolicity of $\E$}
We will now describe how to combine three ladders to make a hyperbolic tripod bundle. We will then use the hyperbolicity of these bundles to prove hyperbolicity of the total space $\E$.

Let $\Sigma_1,\Sigma_2,\Sigma_3$ be three tight sections.  For each ladder $\ladder(\Sigma_i,\Sigma_j)$, let $\Pi_{i,j} \colon \E \to \ladder(\Sigma_i,\Sigma_j)$ be the fiberwise projection onto the ladder (Definition \ref{defn:fiberwise projection map}). For each vertex $v\in \B$, let $x_v = \Sigma_1(v)$, $y_v = \Sigma_2(v)$, and $z_v = \Sigma_3(v)$, then let $h_v$ be the vertex in $\Pi_{1,2}(z_v)$ that is furthest along the geodesic $[x_v,y_v]_v$ from $x_v$.  In particular, we note that $h_v$ a coarse closest point projection of $z_v$ to $[x_v,y_v]_v$.  Define $\half_v(\Sigma_1,\Sigma_2; \Sigma_3)$ to be the subgeodesic $[x_v,h_v]_v$ of $[x_v,y_v]_v$, then let $\half(\Sigma_1,\Sigma_2; \Sigma_3)$ denote the union of the $\half_v(\Sigma_1,\Sigma_2; \Sigma_3)$ over all vertices $v \in \B$.  We note that these {\em half-ladders} are collections of vertices and therefore naturally subsets of both $\E$ and their respective ladders. 
\smallskip

We next show that half-ladders are quasiconvex in their ladders. 

 \begin{lemma}\label{lem:half_ladder_QC}
     There exists $Q_1>0$ depending only on the parameters for $\E$ so that for any three tight sections  $\Sigma_1,\Sigma_2,\Sigma_3$, $\half(\Sigma_1,\Sigma_2; \Sigma_3)$ is a $Q_1$--quasiconvex subset of $\ladder(\Sigma_1,\Sigma_2)$.
 \end{lemma}

\begin{proof}
    Let $\half = \half(\Sigma_1,\Sigma_2;\Sigma_3)$ and $\half_v =\half(\Sigma_1,\Sigma_2;\Sigma_3) \cap \E_v$ for each $v\in \B$ .   We will prove the $\half$ is a quasiconvex subset of $\E$. This will imply $\half$ is a quasiconvex subset of $\ladder(\Sigma_1,\Sigma_2)$ because of the uniform quasi-isometric embedding $\ladder(\Sigma_1,\Sigma_2) \to \E$ (Corollary \ref{cor:ladder_qie}), which preserves $\half$ identically.

    To prove $\half$ is a quasiconvex subset in $\E$, it suffices to prove there is a coarsely Lipschitz retraction of the vertex set of $\E$ onto $\half$. For each vertex $v \in \B$, let $\Upsilon_v \colon \E_v \to \half_v$ denote the closest point projection to the geodesic segment $[x_v,h_v]_v \subset \E_v$. Define $\Upsilon\colon \E \to \half$ by applying the map $\Upsilon_v$ in each fiber $\E_v$. While this is similar to the fiberwise closest point projection onto a ladder (Definition \ref{defn:fiberwise projection map}), Lemma \ref{lem:fiberwise_proj} does not immediately imply $\Upsilon$ is coarsely Lipschitz because $h_v$ and $h_{v'}$ are not necessarily joined by an edge of $\E$ (or even an edge of $\ladder(\Sigma_1,\Sigma_2)$) when $v,v'$ are joined by an edge of $\B$.
    
    Let $a,a'$ be vertices of $\E$ that are joined by an edge. If $a,a'$ are in a fiber $\E_v$, then $d_{\E_v}(\Upsilon(a),\Upsilon(a'))$ is uniformly bounded because each $\Upsilon_v$ is a coarsely Lipschitz map. Thus we can assume $a \in \E_v$ and $a'\in \E_{v'}$ where $v$ and $v'$ are joined by an edge of $\B$. Let $b\in \Upsilon(a) = \Upsilon_v(a)$ and $b' \in \Upsilon(a')= \Upsilon_{v'}(a')$. Axiom (\ref{item:robust}) ensures that $h_v$ is joined by an edge to a vertex $h'$ of $\ladder_{v'}(\Sigma_1,\Sigma_2)$. Let $c'$ be contained in the closest point projection of $a'$ to the geodesic $[x_{v'},h']_{v'} \subset \E_{v'}$. Axiom (\ref{item:retraction}) ensures that $d_\E(b,c')\leq \kappa$, so we need to verify that $c'$ is uniformly close to $b'$. Since $\ladder_{v'}(\Sigma_1, \Sigma_2)$ is just a geodesic in $\E_{v'}$ between $x_{v'}, y_{v'}$, we have two cases depending on the order in which $h', h_{v'}$ appear.
    
    \smallskip

    \underline{\textbf{Case 1,  $[x_{v'},h']_{v'}$ contains $h_{v'}$:}} If $c' \in [x_{v'},h_{v'}]$, then we are done as $c', b' \in \Upsilon_{v'}(a)$ (which has uniformly bounded diameter). If $c' \in [h_{v'},h']$ instead, then $|d_{\E_{v'}}(c',b') - d_{\E_{v'}}(c',h_{v'})|$ is uniformly bounded by item (2) of Lemma \ref{lem:hyp pass close}.  Now Axiom (\ref{item:retraction}) says $d_\E(h_v,h_{v'}) \leq \kappa$ since they are the images of $z_v$ and $z_{v'}$ under $\Pi_{1,2}$, which implies that $d_\E(h_{v'},h') \leq \kappa+1$. Hence Lemma \ref{lem:bounded_between} says that $d_{\E}(c',h_{v'})$ is uniformly bounded, which in turn bounds $d_{\E}(c',b')$, as required.

    \smallskip
    
    \underline{\textbf{Case 2, $[x_{v'},h_{v'}]_{v'}$ contains $h'$:}} If $c' \in [h',h_{v'}]_{v'}$, then $c',b'$ are uniformly close again by item (2) of Lemma \ref{lem:hyp pass close}. So we may assume that $c \in [x_v,h']$ and $b' \in [h',h_{v'}]_{v'} $.   As in case 1, item (2) of the same lemma again implies that $d_{\E_{v'}}(c',h')$ is uniformly bounded, and hence $d_\E(h_v,h_{v'}) \leq \kappa$, so $d_\E(h_v,h') \leq \kappa+1$. This in turn bounds the diameter of $[h',h_{v'}]_v$ in $\E$ by Lemma \ref{lem:bounded_between}. Since $c'$ is close to $h'$ and $h'$ is close to $b'$, $d_\E(c',b')$ is uniformly bounded, completing the proof of the lemma.
\end{proof}

 We now build the {\em tripod bundle} $\tripod(\Sigma_1,\Sigma_2,\Sigma_3)$ by gluing together the three ladders for these sections.  To begin, let $\delta>0$ be the hyperbolicity constant for ladders and  fix $\delta_T> 3\max\{\delta_1, \delta_2\}$, where $\delta_2 = \delta_2(\delta)>0$ comes from Lemma \ref{lem:projections_and_thin_triangles} and $\delta_1 = \delta_1(\delta, \delta_2)>0$ comes from Lemma \ref{lem:coarse center}.  Now define a coarse map $\phi_1 \colon \half(\Sigma_1,\Sigma_2;\Sigma_3) \to \ladder(\Sigma_1,\Sigma_3)$ by taking each vertex $a \in \half_v(\Sigma_1,\Sigma_2;\Sigma_3)$ and letting $\phi_1(a)$ be the set of vertices in $\half_v(\Sigma_1,\Sigma_3;\Sigma_2)$ that are at distance at most $\delta_T$  in $\E_v$.  This allows us to form a graph $\ladder(\Sigma_1, \Sigma_2) \cup_{\phi_1} \ladder(\Sigma_1, \Sigma_3)$ by attaching  $\ladder(\Sigma_1, \Sigma_2)$ to $\ladder(\Sigma_1, \Sigma_3)$ by adding an edge between each vertex $a \in \half_v(\Sigma_1, \Sigma_2; \Sigma_3)$ and $\phi_1(a) \in \half_v(\Sigma_1,\Sigma_3; \Sigma_2)$, as in Proposition \ref{prop:gluing_lemma}.
 \smallskip
 
 Next define $\phi_2\colon \ladder(\Sigma_2,\Sigma_3) \to \ladder(\Sigma_1,\Sigma_2) \cup_{\phi_1} \ladder(\Sigma_1,\Sigma_3)$ by taking each vertex $a \in \ladder_v(\Sigma_2,\Sigma_3)$ and mapping it to the vertices of $\half_v(\Sigma_2,\Sigma_1;\Sigma_3) \cup \half_v(\Sigma_3,\Sigma_1;\Sigma_2)$ that are at distance at most $\delta_T$ from $a$ in $\E_v$. Then the tripod bundle is the graph
 $$\tripod(\Sigma_1,\Sigma_2,\Sigma_3) = (\ladder(\Sigma_1,\Sigma_2) \cup_{\phi_1} \ladder(\Sigma_1,\Sigma_3)) \cup_{\phi_2} \ladder(\Sigma_2,\Sigma_3).$$
 
 The natural maps of the ladders $\iota \colon\ladder(\Sigma_i,\Sigma_j) \to \E$ induces a natural coarse graph map of $\tripod(\Sigma_1,\Sigma_2,\Sigma_3)$ into $\E$.  The next proposition is the last key step in setting up our proof of Theorem \ref{thrm:combination_theorem}.

\begin{proposition}\label{prop:tripod_bundle}
Using the above setup and notation, the following hold:
\begin{enumerate}
    
    \item The natural map $\tripod(\Sigma_1,\Sigma_2,\Sigma_3) \to \E$ is uniformly coarsely Lipschitz.
    \item  Tripod bundles $\tripod(\Sigma_1,\Sigma_2,\Sigma_3)$ are uniformly hyperbolic.
    
    \item  Each of the ladders $\ladder(\Sigma_i,\Sigma_j)$ is uniformly undistorted in $\tripod(\Sigma_1,\Sigma_2,\Sigma_3)$.
\end{enumerate}
 
\end{proposition}

\begin{proof}
The map $\tripod(\Sigma_1,\Sigma_2,\Sigma_3) \to \E$  is coarsely Lipschitz because vertices joined by an edge in  $\tripod(\Sigma_1,\Sigma_2,\Sigma_3)$ are either joined by an edge in at least one of the ladders and hence are mapped to vertices with uniformly bounded distance in $\E$ by Corollary \ref{cor:ladder_qie}, or else they are joined by an edge during the gluing process and hence are mapped to vertices with distance at most $\delta_T$ in a fiber of $\E$. This proves item (1).

We prove item (2) by applying Proposition \ref{prop:gluing_lemma} to the two different gluings. This will also imply that each of the ladders are undistorted in $\tripod(\Sigma_1,\Sigma_2,\Sigma_3)$, completing the proof of item (3).

First, recall $\half(\Sigma_1,\Sigma_2;\Sigma_3)$ is quasiconvex in $\ladder(\Sigma_1,\Sigma_2)$ by Lemma \ref{lem:half_ladder_QC}. The map $$\phi_1 \colon \half(\Sigma_1,\Sigma_2;\Sigma_3) \to \ladder(\Sigma_1,\Sigma_3)$$ is coarsely Lipschitz by construction and we can verify that it is a quasi-isometry, by noticing that the analogously defined map  from $\half(\Sigma_1,\Sigma_3;\Sigma_2)$ to $\ladder(\Sigma_1,\Sigma_2)$ is a coarse inverse to $\phi_1$ (using Lemma \ref{lem:projections_and_thin_triangles}.(2)). Thus $\ladder(\Sigma_1,\Sigma_2) \cup_{\phi_1} \ladder(\Sigma_1,\Sigma_3)$ is uniformly hyperbolic by Proposition \ref{prop:gluing_lemma} and both ladders are undistorted in $\ladder(\Sigma_1,\Sigma_2) \cup_{\phi_1} \ladder(\Sigma_1,\Sigma_3)$.

Since $\ladder(\Sigma_2,\Sigma_3)$ is quasiconvex in itself, we only need to verify that  $\phi_2:\ladder(\Sigma_2,\Sigma_3) \to \ladder(\Sigma_1,\Sigma_2) \cup_{\phi_1} \ladder(\Sigma_1,\Sigma_3)$ is a quasi-isometric embedding.

\begin{claim}\label{claim:coarsely well-defined}
    For each vertex $a \in \ladder(\Sigma_2,\Sigma_3)$, $\phi_2(a)$ has uniformly bounded diameter. 
\end{claim}

\begin{proof}
First, observe that if $a \in \ladder_v(\Sigma_2, \Sigma_3)$ and $\phi_2(a)$ is contained entirely in one of $\ladder(\Sigma_1,\Sigma_2)$ or $\ladder(\Sigma_1,\Sigma_3)$, then the diameter of $\phi_2(a)$ in $\E$ is at most $2\delta_T$. Since $\ladder(\Sigma_1, \Sigma_2)$ and $\ladder(\Sigma_1,\Sigma_3)$ uniformly quasi-isometrically embed in $\E$ (Corollary \ref{cor:ladder_qie}) and in $\ladder(\Sigma_1,\Sigma_2) \cup_{\phi_1} \ladder(\Sigma_1,\Sigma_3)$ (Proposition \ref{prop:gluing_lemma}), this give a uniform diameter for $\phi_2(a)$ in $\ladder(\Sigma_1,\Sigma_2) \cup_{\phi_1} \ladder(\Sigma_1,\Sigma_3)$ in this case.

To handle the case where  $\phi_2(a)$ overlaps both $\ladder(\Sigma_1,\Sigma_2)$ and $\ladder(\Sigma_1,\Sigma_3)$, recall that our choice of the vertices used to define the half-ladders $\half_v(\Sigma_2,\Sigma_3;\Sigma_1)$ and $\half_v(\Sigma_3,\Sigma_2;\Sigma_1)$ ensures that
    $$ \half_v(\Sigma_2,\Sigma_3;\Sigma_1) \cup  \half_v(\Sigma_3,\Sigma_2;\Sigma_1) = \ladder_v(\Sigma_2,\Sigma_3)$$
and moreover that the set of vertices in
$$ \half_v(\Sigma_2,\Sigma_3;\Sigma_1) \cap  \half_v(\Sigma_3,\Sigma_2;\Sigma_1)$$
is exactly the set of vertices in the closest point projection of $\Sigma_1(v)$ onto $\ladder_v(\Sigma_2,\Sigma_3)$. So if $\phi_2(a)$ contains vertices on both $\ladder(\Sigma_1,\Sigma_2)$ and $\ladder(\Sigma_1,\Sigma_3)$, then item (2) of Lemma \ref{lem:coarse center} says that $a$ is within $\delta_1 = \delta_1(\delta, \delta_2)>0$ of the closest point projection of $\Sigma_1(v)$ onto $\ladder_v(\Sigma_2,\Sigma_3)$. Let $h_2$ and $h_3$ be vertices in the closest point projections of $\Sigma_2(v)$ onto $\ladder_v(\Sigma_1,\Sigma_3)$ and $\Sigma_3(v)$ onto $\ladder_v(\Sigma_1,\Sigma_2)$, respectively. By Lemma  \ref{lem:projections_and_thin_triangles}, $h_2$, and $h_3$ are at distance at most $\delta_2$ in $\E_v$ for $\delta_2= \delta_2(\delta)>0$.  Since $\delta_T>3 \max\{ \delta_1, \delta_2\}$ and $h_3 \in \half_v(\Sigma_2,\Sigma_1;\Sigma_3) \cap \half_v(\Sigma_1,\Sigma_2;\Sigma_3)$ and $h_2 \in \half_v(\Sigma_3,\Sigma_1;\Sigma_2) \cap \half_v(\Sigma_1,\Sigma_3;\Sigma_2)$, we have that $h_2$ and $h_3$ are joined by an edge in $\tripod(\Sigma_1,\Sigma_2,\Sigma_3)$, as are $h_2,h_3$ with $a$ since $h_2,h_3 \in \phi_2(a)$. Since the portion of $\phi_2(a)$ that are entirely contained in either $\ladder(\Sigma_1,\Sigma_2)$ or $\ladder(\Sigma_1,\Sigma_3)$ are uniformly bounded, this uniform bound on the overlap case provides a uniform bound for $\phi_2(a)$.
\end{proof}

\begin{claim}\label{claim:coarsely lipschitz}
    $\phi_2$ is uniformly coarsely Lipschitz.
\end{claim}

\begin{proof}
It suffices to deal with the case of two vertices $a,a'$ which are joined by an edge of $\ladder(\Sigma_2,\Sigma_3)$. There are two cases, depending on whether the edge is comes from the fiber or the base:
\begin{enumerate}
    \item $a$ and $a'$ are contained in a single fiber $\E_v$, and
    \item There are vertices $v,v' \in \B$ that are joined by an edge in $\B$ and $a \in \E_v$ while $a'\in \E_{v'}$.
\end{enumerate}
In both cases we prove that the distance between $\phi_2(a)$ and $\phi_2(a')$ is uniformly bounded.

\medskip

\textbf{\underline{Case 1}:} Suppose that $a$ and $a'$ are in single fiber $\E_v$.  Since $$ \half_v(\Sigma_2,\Sigma_3;\Sigma_1) \cup  \half_v(\Sigma_3,\Sigma_2;\Sigma_1) = \ladder_v(\Sigma_2,\Sigma_3)$$ and $$\half_v(\Sigma_2,\Sigma_3;\Sigma_1) \cap  \half_v(\Sigma_3,\Sigma_2;\Sigma_1) $$ must contain at least one vertex of $\ladder_v(\Sigma_2,\Sigma_3)$, we must have either $$a,a' \in \half_v(\Sigma_2,\Sigma_3;\Sigma_1) \text{ or } a,a' \in \half_v(\Sigma_3,\Sigma_2;\Sigma_1).$$

Without loss of generality, assume $a,a' \in \half_v(\Sigma_2,\Sigma_3;\Sigma_1)$. Then there is $b \in \phi_2(a)$ and $b'\in \phi_2(a')$ so that $b,b' \in \ladder(\Sigma_1,\Sigma_2)$. Thus $d_{\E}(b,b') \leq 2\delta_T +1$. Since  $\ladder(\Sigma_1,\Sigma_2)$ uniformly quasi-isometrically embeds in $\E$, this implies the distance between $\phi_2(a)$ and $\phi_2(a')$ is uniformly bounded in $\ladder(\Sigma_1,\Sigma_2) \cup_{\phi_1} \ladder(\Sigma_1,\Sigma_3)$.  

\medskip

\textbf{\underline{Case 2}:}  Suppose instead that $a \in \E_v$ and $a' \in \E_{v'}$, where $v$ and $v'$ are joined by an edge of $\B$.  First assume that $a \in \half_v(\Sigma_2,\Sigma_3;\Sigma_1)$ and $a'\in \half_{v'}(\Sigma_2,\Sigma_3;\Sigma_1)$. Then Lemma \ref{lem:projections_and_thin_triangles}.(2) and our choice of $\delta_T> \delta_2$ provides $b \in \phi_2(a) \cap \ladder_v(\Sigma_1,\Sigma_2)$ and $b' \in \phi_2(a') \cap \ladder_{v'}(\Sigma_1,\Sigma_2)$. Let $\Pi_{1,2}$ be the fiberwise closest point projection $\Pi_{1,2} \colon \E \to \ladder(\Sigma_1,\Sigma_2)$ (Definition \ref{defn:fiberwise projection map}). Then $\Pi_{1,2}(a)$ and $\Pi_{1,2}(a')$ are contained in the $\delta_T$--neighborhoods of $a$ and $a'$  in $\E$, respectively. Thus 
    $$d_\E(b, \Pi_{1,2}(a)) \leq 2 \delta_T \text{ and } d_\E(b', \Pi_{1,2}(a')) \leq 2 \delta_T.$$
Since $a$ and $a'$ are joined by an in edge between adjacent fibers in $\ladder(\Sigma_2,\Sigma_3)$, their distance in $\E$ is at most $K_1$ (Corollary~\ref{cor:ladder_qie}). Since $\Pi_{1,2}$ is $K_0$--coarsely Lipschitz (Lemma~\ref{lem:fiberwise_proj}) it follows that $d_\E( \Pi_{1,2}(a),  \Pi_{1,2}(a')) \leq K_0K_1+K_0$. 
This uniformly bounds $d_\E(b,b')$ which in turn bounds the distance between $\phi_2(a)$ and $\phi_2(a')$ in $\ladder(\Sigma_1,\Sigma_2) \cup_{\phi_1} \ladder(\Sigma_1,\Sigma_3)$ by the uniform quasi-isometric embedding of ladders (Corollary \ref{cor:ladder_qie}).

Now assume that $a \in \half_v(\Sigma_2,\Sigma_3;\Sigma_1)$ and $a'\not\in \half_{v'}(\Sigma_2,\Sigma_3;\Sigma_1)$ so that $\phi_2(a) \subset \ladder_v(\Sigma_1,\Sigma_2)$ and $\phi_2(a') \subset \ladder_v(\Sigma_1,\Sigma_3)$. Let $b \in \phi_2(a) \cap \half_v(\Sigma_2,\Sigma_1;\Sigma_3)$ and $b' \in \phi_2(a') \in \half_{v'}(\Sigma_3, \Sigma_1; \Sigma_2)$.  Let $\Pi_{i,j}\colon \E \to \ladder(\Sigma_i,\Sigma_j)$ be the fiberwise closest point projection for $i,j \in \{1,2,3\}$. For $i \in \{1,2,3\}$ and $\ast \in \{v,v'\}$, let $h_{\ast,i}$ be a point in  $\Pi_{i+1,i+2}\left(\Sigma_i(\ast)\right)$  to where the indices are taken mod 3.

Using Lemma \ref{lem:projections_and_thin_triangles}, $\Pi_{1,3}(b)$ and $\Pi_{1,3}(a)$ are both $\delta_2$--close to $h_{v,2}$ in $\E_v$ for $\delta_2=\delta_2(\delta)>0$. By Axiom (\ref{item:retraction}), we have
$$\diam_\E(\Pi_{1,3}(a), \Pi_{1,3}(a')) \leq \kappa.$$

Since $d_\E(b', \Pi_{1,3}(a'))\leq 2\delta_T$,  we have $d_\E(b',h_{v,2}) \leq \kappa + 2\delta_T + \delta_2$.  Repeating this calculation with $b'$, $a'$, $\ladder(\Sigma_1,\Sigma_2)$ and $h_{v',3}$ we get 
\begin{equation}\label{eq:b,h_{v',3}}
    d_\E(b,h_{v',3}) \leq \kappa +2 \delta_T + \delta_2.
\end{equation}
For all $i,j \in \{1,2,3\}$, Axiom (\ref{item:retraction}) says $d_\E(h_{v,i}, h_{v',i}) \leq \kappa$ and Lemma \ref{lem:projections_and_thin_triangles} says
    $$d_{\E_v}(h_{v,i},h_{v,j}) \leq \delta_2\text{ and } d_{\E_{v'}}(h_{v',i}, h_{v',j}) \leq \delta_2.$$
Thus we have
    \begin{equation}\label{eq:b',h_{v',2}}
        d_\E(b', h_{v',2}) \leq d_\E(b',h_{v,2}) + d_\E( h_{v,2},h_{v',2}) \leq 2\kappa + 2\delta_T + \delta_2.
    \end{equation}

Let $d_{\phi_1}(\cdot,\cdot)$ be the distance in $\ladder(\Sigma_1,\Sigma_2) \cup_{\phi_1} \ladder(\Sigma_1,\Sigma_3)$. The triangle inequality provides
    \begin{equation}\label{eq:b',b}
        d_{\phi_1}(b',b) \leq d_{\phi_1}(b', h_{v',2}) + d_{\phi_1}(h_{v',2},h_{v',3}) + d_{\phi_1}(h_{v',3},b).
    \end{equation}
The quasi-isometric embedding of ladders (Corollary \ref{cor:ladder_qie}) and \eqref{eq:b,h_{v',3}} implies that $d_{\phi_1}(b,h_{v',3})$ is uniformly bounded, and similarly for $d_{\phi_1}(b',h_{v',2})$ using \eqref{eq:b',h_{v',2}}.  On the other hand, we have $d_{\phi_1}(h_{v',2},h_{v',3}) =1$ because $h_{v',2}$ and $h_{v',3}$ are joined by an edge of $\ladder(\Sigma_1,\Sigma_2) \cup_{\phi_1} \ladder(\Sigma_1,\Sigma_3)$ by definition of $\phi_1$.  This shows that \eqref{eq:b',b} is uniformly bounded, completing the proof of the claim.
 \end{proof}

To finish the proof that $\phi_2$ is a quasi-isometric embedding, let $\Pi \colon \E \to \ladder(\Sigma_2,\Sigma_3)$ be the fiberwise closest point projection (Definition \ref{defn:fiberwise projection map}). For each vertex $a \in \ladder(\Sigma_2,\Sigma_3)$, we have $d_\E(a, \Pi\circ \phi_2(a)) \leq 2\delta_T$ by definition. Since both $\Pi$  and $\tripod(\Sigma_1,\Sigma_2,\Sigma_3) \to \E$ are coarsely Lipschitz (by Lemma \ref{lem:fiberwise_proj} and Claim \ref{claim:coarsely lipschitz} respectively), the above implies that $\phi_2$ is actually a quasi-isometric embedding.

We can now apply Proposition \ref{prop:gluing_lemma} to conclude that $\tripod(\Sigma_1,\Sigma_2,\Sigma_3)$ is uniformly hyperbolic and $\ladder(\Sigma_2,\Sigma_3)$ and $\ladder(\Sigma_1,\Sigma_2) \cup_{\phi_1} \ladder(\Sigma_1,\Sigma_3)$ are uniformly undistorted subsets. Since we also have that $\ladder(\Sigma_1,\Sigma_2)$ and $\ladder(\Sigma_1,\Sigma_3)$ are uniformly undistorted in  $\ladder(\Sigma_1,\Sigma_2) \cup_{\phi_1} \ladder(\Sigma_1,\Sigma_3)$ (again by Proposition \ref{prop:gluing_lemma}), all three ladders are undistorted in  $\tripod(\Sigma_1,\Sigma_2,\Sigma_3)$.  This completes the proof of the proposition.
\end{proof}

We are now ready to prove our main hyperbolicity statement:

\begin{theorem}
    $\E$ is hyperbolic.
\end{theorem}

\begin{proof}
We will use the guessing geodesics Proposition \ref{prop:guessing}, beginning by describing our guessed geodesics $\varsigma(\cdot,\cdot)$. For each vertex $x\in\mathcal E$, fix a tight section $\Sigma_x$ through $x$. For each pair of vertices $x,y\in \E$, consider the ladder $\ladder_{x,y}=\ladder(\Sigma_x,\Sigma_y)$, which with a slight abuse of notation we think of as containing $x$ and $y$. We then define $\varsigma(x,y)$ to be the image under the natural map $\ladder_{x,y}\to \mathcal E$ (as in Corollary \ref{cor:ladder_qie}) of a geodesic in $\ladder_{x,y}$ connecting $x$ to $y$.

The fact that the natural maps from ladders to $\mathcal E$ are uniform quasi-isometric embeddings (Corollary \ref{cor:ladder_qie}) implies the existence of a uniform constant $D\geq 0$ such that all $\varsigma(x,y)$ are $D$--coarsely connected, and such that whenever $x,y$ are adjacent the diameter of $\varsigma(x,y)$ is at most $D$.

We are left to show the second condition of Proposition \ref{prop:guessing}, namely thinness of $\varsigma(\cdot,\cdot)$ triangles. Consider vertices $x,y,z\in\mathcal E$. Since ladders are uniformly undistorted inside tripod bundles by item (3) of Proposition \ref{prop:tripod_bundle}, a geodesic in $\ladder_{x,y}$ from $x$ to $y$ can be seen as a uniform quasi-geodesic $\gamma_{x,y}$ inside the tripod bundle $\tripod_{x,y,z}=\tripod(\Sigma_x,\Sigma_y,\Sigma_z)$, and similarly for the other pairs. Since $\tripod_{x,y,z}$ is uniformly hyperbolic (item (2) of Proposition \ref{prop:tripod_bundle}), $\gamma_{x,y}$ is contained in a uniform neighborhood of $\gamma_{x,z}\cup\gamma_{z,y}$. Since the natural map $\tripod_{x,y,z}\to \mathcal E$ is uniformly coarsely Lipschitz (item (1) of Proposition \ref{prop:tripod_bundle}), and the images of the $\gamma_{\cdot,\cdot}$ are the $\varsigma_{\cdot,\cdot}$, we conclude that $\varsigma_{x,y}$ is contained in a uniform neighborhood of $\varsigma_{x,z}\cup\varsigma_{z,y}$, as required.
\end{proof}

\section{Hyperbolic bundles over geometrically finite groups} 
\label{sec:hyperbolic GF bundles}

In this section we consider surface bundles whose monodromy is injective with {\em parabolically geometrically finite} image; see Definition \ref{defn:PGF}.  We construct a space on which the fundamental group of this bundle acts, and in Theorem \ref{thm:combination_applies} prove that Theorem~\ref{thrm:combination_theorem} applies to this space, and consequently that the space is hyperbolic.  This will be a key ingredient in Section~\ref{sec:cyclic PGF HHG} where we prove that such fundamental groups of bundles are hierarchically hyperbolic whenever the parabolic subgroups are cyclic.  Before we begin, we describe some of the preliminary definitions and results we will use.

\subsection{Curve complexes and subsurface projections} \label{subsec:curve complexes}
We suppose $S$ is a closed surface of genus $g \geq 2$ throughout.
We let $\star \in S$ denote a fixed basepoint, and write $\dot S = S \smallsetminus \{\star\}$.  Fix a hyperbolic metric on $S$ and let $\uncov \colon \widetilde S \to S$ denote the universal cover, which we may identify with the hyperbolic plane so that $\uncov$ is a local isometry.  We also let $\widetilde \star \in \uncov^{-1}(\star) \subset \widetilde S$ denote a fixed choice of basepoint.

For $R = S$ or $\dot S$, we let $\C(R)$ denote the curve complex of $R$.  This is a simplicial complex with vertex set given by the set of isotopy classes of essential simple closed curves on $R$, and a $k$--simplex defined by a set of $k+1$ distinct isotopy classes of simple closed curves which have pairwise disjoint representatives.  It is classical that $\C(R)$ is connected, and we often consider the geodesic metric on the {\em curve graph}, which is the $1$--skeleton, $\C^1(R)$, of the curve complex in which each edge is endowed with length $1$.

We realize all vertices of $\C(S)$ by their geodesic representatives with respect to the fixed hyperbolic metric.  For each point $v \in \C(S)$, we view $v$ as both a point in the curve complex, as well as a weighted multicurve (weighted according to the barycentric coordinates), also realized by its geodesic representative.  We similarly conflate a point $v \in \C(\dot S)$ with an isotopy class of weighted multicurve, as well as some representative of the isotopy class.

It will be convenient to take the barycentric subdivision of both $\C(S)$ and $\C(\dot S)$, denoted $\hat \C(S)$ and $\hat \C(\dot S)$.  
Every vertex of $\hat \C(S)$ or $\hat \C(\dot S)$ is the barycenter of a $k$--simplex for some $k$; we call such a vertex a $k$--vertex (so that a $0$--vertex is a vertex of the unsubdivided complex).  Given a vertex $v \in \hat \C^0(S)$, we will also write $v \subset \C(S)$ for the simplex it determines, as well as $v \subset S$ for the (unweighted) geodesic multicurve in $S$ that $v$ determines.  We similarly abuse notation for a vertex $x \in \hat \C(\dot S)$ (though we do not assume the multicurve $x$ in $\dot S$ is geodesic, but rather identify it with its isotopy class, or a representative, when convenient).

We also consider the $1$--skeleton $\hat \C^1(R)$ as a geodesic metric space for which all edge lengths are equal to $1$.  The inclusion $\C^1(R) \to \hat \C^1(R)$ is then clearly $2$--Lipschitz and $1$--dense.  Moreover, for any two $0$--vertices in $\hat \C^1(R)$ and any $\hat \C^1(R)$--geodesic between them, we can easily find a path in $\C^1(R)$ between these vertices which is no longer than the $\hat \C^1(R)$--geodesic (by replacing each $k$--vertex $\{v_0,\ldots,v_k\}$ of the $\hat \C^1(R)$--geodesic with any vertex $v_i$ contained in it).   Consequently, the inclusion $\C^1(R) \to \hat \C^1(R)$ is a $2$--quasi-isometry.

We will need the following fundamental result of Masur and Minsky \cite{MM:CC1}.

\begin{theorem} [Masur-Minsky] \label{thm:MM hyperbolic}
The curve graph of a finite type surface is Gromov hyperbolic.
\end{theorem}

In fact, uniform hyperbolicity of the curve graphs was subsequently and independently proved by multiple authors; see \cite{Aougab:Uniform,Bowditch:Uniform,CRS:Uniform,HPW:Uniform}.  This also implies uniform hyperbolicity of the $1$--skeleta of the barycentric subdivisions.  We will thus sometimes refer to {\em the hyperbolicity constant} of curve complexes/graphs, and we mean any fixed bound for the slimness of triangles in curve graphs as well as that of the $1$--skeleta of the barycentric subdivisions.

We will consider essential subsurfaces $Y \subset R = S$ or $\dot S$. For our purpose, an {\em essential subsurface} is a component of the complement of an open neighborhood of an essential multicurve in $R$, or a closed annular neighborhood of an essential simple closed curve in $R$.  We also assume $Y$ is not a three-holed sphere.   If $Y$ is proper, then it has compact boundary (and possibly one puncture if $R = \dot S$). 
Essential subsurfaces are well-defined up to isotopy.

For non-annular surfaces, we write $\AC(Y)$ for the arc-and-curve complex of $Y$, whose vertices are proper isotopy classes of essential arcs and curves in $Y$ with simplices defined by disjointness as in the case of curve complexes. Here an essential arc is any properly embedded arc (i.e.~endpoints in the boundary of $Y$ and open ends limiting to punctures), which cannot be properly homotoped rel endpoints (if any) into an arbitrary neighborhood of the boundary/puncture.  We note that isotopies are proper, but are not {\em fixed} on the boundary, and thus arcs do not have well-defined endpoints on the boundary components they meet.
Being (uniformly) quasi-isometric to the curve graph of $Y$, the arc-and-curve complex $\AC^1(Y)$ is also (uniformly) hyperbolic (c.f.~Theorem~\ref{thm:MM hyperbolic}).  When $Y$ is an annulus, we let $\A(Y)$ denote the arc graph of $Y$, which is defined as in \cite{MM:CC2} by taking the natural visual compactification of the cover of $R$ corresponding to $\pi_1(Y) < \pi_1(R)$, whose vertices are properly embedded arcs with endpoints on distinct boundary components, up to isotopy {\em fixing the boundary pointwise}.  We refer the reader to \cite{MM:CC2} for a full description.

Given an essential subsurface $Y \subset R$, we let $\pi_Y$ denote the {\em arc-and-curve graph subsurface projection}.  When $Y$ is not an annulus, $\pi_Y(\alpha)$ assigns to the isotopy class of a multicurve $\alpha$ on $R$ the simplex in $\AC(Y)$ defined as follows.  After realizing $\alpha$ as a multicurve $\alpha_0 \subset R$ intersecting $\partial Y$ transversely, and meeting $Y$ in the minimal number of components, the vertex set of $\pi_Y(\alpha)$ is the set of proper isotopy classes of components of intersection $\alpha_0 \cap Y$.  This differs slightly from subsurface projection defined in \cite{MM:CC2}, but it is straightforward to translate between the two.  Subsurface projections for annuli $Y \subset R$ are defined as exactly as in \cite{MM:CC2}.  

Given isotopy classes $\alpha,\beta$ of multicurves in $R$ and an essential subsurface $Y \subset R$ for which $\pi_Y(\alpha),\pi_Y(\beta) \neq \emptyset$, we write
\[ d_Y(\alpha,\beta) = \diam(\pi_Y(\alpha) \cup \pi_Y(\beta)).\]
Here we are taking the diameter of the union of the vertices of $\pi_Y(\alpha)$ and $\pi_Y(\beta)$ in $\AC^1(Y)$.  When $Y$ is an annulus around a core curve $v$, we also use the notation $d_v(\alpha,\beta) = d_Y(\alpha,\beta)$.

Two key features we will need are the following due to Masur and Minksy in \cite{MM:CC2}.  The first is immediate from our definition of projections since a pair of adjacent vertices is a (two component) multicurve.

\begin{lemma} [Lipschitz curve complex projection] \label{lem:cc projection lipschitz}
    Given any essential subsurface $Y \subset R$, if $\alpha,\beta \in \C^0(R)$  are adjacent and have $\pi_Y(\alpha),\pi_Y(\beta) \neq \emptyset$, then  $d_Y(\alpha,\beta)\leq 1$.
\end{lemma}

\begin{theorem}
    [Bounded Geodesic Image] \label{thm:MM BGI} There exists a constant $M > 0$ with the following property.  Given any essential subsurface $Y \subset R$ and geodesic in $\C(R)$ with vertex set $\{v_i\}$, if $\pi_Y(v_i) \neq \emptyset$ for all $i$, then
    \[ \diam\left(\bigcup_i \pi_Y(v_i)\right) < M.\]
\end{theorem}

\subsection{Curve complex fibrations}
\label{subsec:cc fibrations}
The puncture filling map $\dot S \to S$ induces a forgetful projection
\[ \Psi \colon \C(\dot S) \to \C(S).\] 
This simplicial map is equivariant with respect to the forgetful homomorphism in the Birman exact sequence
\begin{equation} \label{E:BES} 1 \to \pi_1S \to \Mod(\dot S) \to \Mod(S) \to 1.
\end{equation}
The following theorem of \cite{KLS_Trees} describes the fibers of $\Psi$.
\begin{theorem}
For every point $v \in \C(S)$, $T_v = \Psi^{-1}(v)$ is $\pi_1S$--equivariantly isometric to the weighted Bass-Serre tree dual to the weighted multicurve represented by $v$.
\end{theorem}

The map $\Psi \colon \C(\dot S) \to \C(S)$ is not simplicial with respect to the barycentric subdivisions of domain and range (though it is when restricted to any simplex on which $\Psi$ is injective).  It is straightforward to construct a $\Mod(\dot S)$--invariant subdivision which is simplicially isomorphic to the barycentric subdivision (and which agrees with it on every simplex for which $\Psi$ is injective) so that $\Psi$ is simplicial.  We do so, and continue to denote this by $\hat C(\dot S)$ to avoid adding even more notation.  In particular, $\Psi \colon \hat C(\dot S) \to \hat \C(S)$ simplicial, and for every vertex $v$ of $\hat \C(S)$, $\Psi^{-1}(v)$ is a subcomplex.  In fact, $\Psi^{-1}(v)$  is $\pi_1S$--equivariantly isomorphic to the barycentric subdivision $\hat T_v$ of the Bass-Serre tree, $T_v$, defined by the multicurve $v$.

For every $v \in \C(S)$, there is a $\pi_1S$--equivariant map
\[ \Phi_v \colon \widetilde S \smallsetminus \uncov^{-1}(v) \to \Psi^{-1}(v) \subset \C(\dot S),\]
constant on each complementary component of $\uncov^{-1}(v)$, defined as follows.  Given any $f \in \Diff_0(S)$ such that $f(\star) \not \in v$, $f^{-1}(v)$ is a weighted multicurve representing a point in $\Psi^{-1}(v)$.  There is a unique lift $\widetilde f \colon \widetilde S \to \widetilde S$ such that the isotopy of $f$ to the identity lifts to an isotopy of $\widetilde f$ to the identity.  Because the (isotopy class of the) multicurve $f^{-1}(v)$ depends only on $v$ and the component of $ \widetilde S \smallsetminus \uncov^{-1}(v)$ containing $\widetilde f(\widetilde \star)$, we can define  $\Phi_v$ by
\[ \Phi_v(\tilde{a}) = \Phi_v(\widetilde f(\widetilde \star)) = f^{-1}(v)\]
where $f$ is any map in $\Diff_0(S)$ so that $\tilde{f}(\star)$ is in the same component of $  \widetilde S \smallsetminus \uncov^{-1}(v)$ as $\tilde{a}$.

Geometrically, it can be useful to think of picking a point $a \in S \smallsetminus v$, and dragging the point back by an isotopy to $\star$.  The weighted multicurve $v$ can be dragged along by the isotopy to produce a weighted multicurve $x$ on $\dot S$.  Lifting the isotopy to an isotopy of the identity on $\widetilde S$, there is a point $\widetilde a \in \widetilde S \smallsetminus \uncov^{-1}(v)$ so that the final map of the lifted isotopy sends $\widetilde a$ to $\widetilde \star$, and then $\Phi_v(\widetilde a) = x$.
For a vertex $v$, $\Phi_v$ is surjective onto the vertices in the fiber over $v$.

Let $\Omega$ denote the preimage under $\uncov$ of the complement of the set of all simple geodesics on $S$.  By the Birman-Series Theorem \cite{BirmanSeries}, $\Omega$ is an open, dense, full measure set.  For every $v \in \C(S)$, the map $\Phi_v$ is defined on $\Omega$, and thus $\Phi(v,\tilde{a}) = \Phi_v( \tilde{a})$  gives a well-defined map \[ \Phi \colon \C(S) \times \Omega \to \C(\dot S).\]
Furthermore, for all $\widetilde a \in \Omega$, the restriction $\Phi|_{\C(S) \times \{\widetilde a\}}$ is a simplicial isomorphism to a subcomplex of $\C(\dot S)$, and post-composing this restriction with $\Psi$ is the ``identity" on $\C(S) = \C(S) \times \{\widetilde a\}$.  Thus we view these restrictions, $\Phi|_{\C(S) \times \{\widetilde a\}}$ as sections of $\Psi$.

From the construction of $\Phi$ described above, we see that for each $\widetilde a \in \Omega$, $\Phi(\C(S) \times \{\widetilde a\})$ in fact provides representatives of the isotopy classes that intersect in precisely the same pattern as the geodesic representatives in the domain;
\begin{equation} \label{Eqn:same i}
    i(v,v') = i(\Phi(v,\widetilde a),\Phi(v',\widetilde a)).
\end{equation}
Indeed, if $f \in \Diff_0(S)$ has $\widetilde f(\widetilde \star) = \widetilde a$, then $\Phi(v,\widetilde a) = f^{-1}(v)$, and since $f^{-1}$ is a diffeomorphism, it sends each pair of geodesics to a pair of curves that intersect minimally in the exact same number of points. In particular, $v \mapsto \Phi(v,\widetilde a)$ is simplicial for all $\widetilde a \in \Omega$.

\begin{remark}
    This is a variation on the construction in \cite[Subsection 2.2]{LeiMjSch}, which produces a continuous map defined on all of $\C(S) \times \widetilde S$.   Continuity is less important for us, and the fact that $\Phi_v^{-1}(x) \subset \widetilde S$ is an open convex set bounded by components of $\uncov^{-1}(v)$ is convenient.  
\end{remark}

We note that $\Phi$ respects barycentric subdivisions.  Specifically, for all $\widetilde a \in \Omega$, $\Phi|_{\hat C(S) \times \{\widetilde a\}}$ is a simplicial isomorphism to a subcomplex of $\hat \C(\dot S)$ (in fact, this is true even without modifying the subdivision of $\C(\dot S)$ as above).

\subsection{PGF groups and base spaces}
\label{subsec_E_for_RGF}

Given a finitely generated group $G$ and collection of {\em peripheral subgroups}, $H_1,\ldots,H_k$, recall that the {\em coned-off Cayley graph}, $\hat G$, of $G$ is the Cayley graph with respect to the generating set which is the union a finite generating set together with all elements of the peripheral subgroups.
 If $G$ is  \emph{hyperbolic relative} to $H_1,\ldots,H_k$, then $\hat G$ is hyperbolic and its quasi-isometry type is not impacted by the choice of finite generating set (see \cite{Bowditch-Rel} for the complete definition of relative hyperbolicity).  We are interested in relatively hyperbolic subgroups of $\Mod(S)$ in conjunction with the following notion of geometric finiteness.

\begin{definition}[PGF]\label{defn:PGF}
    A finitely subgroup $G <\Mod(S)$ is \emph{parabolically geometrically finite} if $G$ is hyperbolic relative to a  collection of  subgroups $H_1,\dots,H_k$ where \begin{enumerate}
        \item each $H_i$ contains a finite index, abelian subgroup consisting entirely of multi-twists, and
        \item any orbit map from $G$ to $\C(S)$ is a quasi-isometric embedding of the coned-off Cayley graph $\hat G$ into $\C(S)$.
    \end{enumerate}
\end{definition}

From now on we fix a PGF subgroup $G < \Mod(S)$ and let $H_1,\dots, H_k$ be the peripheral subgroups of $G$. We define two graphs $\B_1$ and $\B_2$ with $\B_1 \subset \B_2$, so that both are $G$--equivariantly quasi-isometric to the coned off Cayley graph, $\hat G$, and there is a $G$--equivariant retraction $\B_2 \to \B_1$.  The first graph $\B_1$ is any connected graph with vertex set given by the set of cosets of $H_1,\ldots,H_k$ and which is $G$--equivariantly quasi-isometric to $\hat G$. We also assume (as we may) that $\B_1$ has finitely many $G$--orbits of edges. For example, we can define edges between $v,w \in \B^0_1$ if the corresponding cosets contain elements that are distance at most $1$ in the Cayley graph of $G$ with respect to a fixed finite generating set.  

Note that each vertex $v \in \B_1^0$ is stabilized by a conjugate, denoted $H_v$, of one of the subgroups $H_1,\ldots,H_k$. 
The union of the multi-curves supporting multi-twists in $H_v$ form a single multi-curve, and thus a vertex of $\hat \C(S)$, which we also denote $v$; note that $H_v$ stabilizes $v$ (setwise).
We thus obtain a $G$--equivariant identification of $\B_1^0$ with a subset of the vertex set of the barycentric subdivision of the curve complex, $\B_1^0 \subset \hat \C^0(S)$.

The second graph $\B_2$ is defined from $\B_1$ as follows.  For each $v \in \B_1^0 \subset \hat \C^0(S)$, we consider all submulticurves $v_0 \subset v$ which are supports for multitwists in the stabilizer of $v$ in $G$.  The set of all such submulticurves of $v$ forms a poset (by inclusion) with $v$ as a maximal element, and the {\em cluster of $v$} is the graph whose vertices are the elements of the poset, with edges whenever there is a containment of multicurves.  We call the edges in this graph {\em cluster edges}.  The cluster of $v$ naturally sits as a subgraph of the barycentric subdivision of the simplex $v \subset \C(S)$.  The graph $\B_2$ is obtained by replacing each $v \in \B_1$ by its vertex cluster, joining vertices from distinct vertex clusters whenever maximal elements from those clusters are joined by an edge in $\B_1$.  By construction, there is an inclusion $\B_1 \subset \B_2$, and a simplicial retraction $\B_2 \to \B_1$ sending each vertex to the maximal element in its cluster.  

The product of two multitwists from distinct parabolic subgroups is loxodromic (on either of $\B_1$ or $\B_2$), and hence must be pseudo-Anosov.  This implies the following (c.f.~\cite{Loa_free_products}).

\begin{lemma} \label{lem:vertices filling}
    For any two vertices $v,v' \in \B_2^0$ which are in different clusters, $v \cup v'$ fills $S$.  Consequently, any two vertices of $\B_1^0$ fill $S$. \qed
\end{lemma}

We extend both inclusions $\B_j^0 \subset \hat \C^0(S)$ to $G$--equivariant maps $\B_j \to \hat \C^1(S)$ sending each edge to a geodesic in the $1$--skeleton $\hat \C^1(S)$.  In order to do this, we may need to replace each edge of $\B_j$ between vertices in different clusters with a finite set of edges so that the action of $G$ on $\B_j$ has trivial edge stabilizers.  That this is possible follows from Lemma \ref{lem:vertices filling}, since the stabilizer in $\Mod(S)$, hence $G$, of a pair of filling multicurves $v \cup v'$ is finite. 

\begin{remark} When the peripheral subgroups $H_j$ are virtually cyclic, we have $\B_1 = \B_2$.  In particular, if each $H_j$ is virtually an infinite cyclic group generated by a Dehn twist, then these graphs are the same and $\B = \B_1 = \B_2$ admits an equivariant map into $\C^1(S)$, without the barycentric subdivision.
\end{remark}

\subsection{Tree bundles for extensions of PGF groups}

We let $\Gamma = \Gamma_G$ be the associated surface group extension
\[ 1 \to \pi_1S \to \Gamma \to G \to 1.\]
Viewing $G < \Mod(S)$, this sequence can be defined by its natural embedding into the Birman exact sequence \eqref{E:BES}, so that $\Gamma_G$ is the preimage of $G$ in $\Mod(\dot S)$.  Alternatively, we can view $G$ as the isomorphic image of the monodromy of a bundle over a space with fundamental group $G$, and $\Gamma_G$ as the fundamental group of the bundle.

For each $j =1,2$, construct a graph $\E_j$ together with a map
\[ p \colon \E_j \to \B_j, \]
equivariant with respect to $\Gamma \to G$,
as follows.  The vertex set of $\E_j$ is the subset $\E_j^0 \subset \hat \C^0(\dot S)$ which is the union of
\[ \bigcup_{v \in \B_j^0} T_v^0 \subset \Psi^{-1}(\B_j^0).\]  Here we are viewing $\B_j^0 \subset \hat \C^0(S)$, and we recall that $ \Psi^{-1}(v) \cong \hat T_v$, the barycentric subdivision of $T_v$ (thus we are omitting all valence $2$ vertices of all $\hat T_v$).  The map $p$ on the vertex set is just the restriction of $\Psi$, and so is equivariant with respect to $\Gamma \to G$.

There are two or three types of edges in $\E_j$.  We include all edges of $T_v$, over all $v \in \B_j^0 \subset \hat \C^0(S)$, and call these {\em vertical edges}.  We view $T_v$ as embedded in $\hat \C(\dot S)$ via the equivariant isomorphism $\Psi^{-1}(v) \cong \hat T_v$. We also add edges between  $x,x' \in \E_j^0 \subset \hat \C^0(\dot S)$ whenever $v= \Psi(x)$ and $v' = \Psi(x')$ are joined by an edge in $\B_j$ and \[ \Phi_{v}^{-1}(x) \cap \Phi_{v'}^{-1}(x') \neq \emptyset,\]   We call such edges {\em primary horizontal edges}. Finally, for $\B_2$, we include all edges in the $\Psi$--preimage of cluster edges of $\B_2$ between vertices of $\E_2^0$, and call these {\em horizontal cluster edges}. These last two types of edges make up the set of {\em horizontal edges}. 

Note that the intersection $\Phi_{v}^{-1}(x) \cap \Phi_{v'}^{-1}(x')$ above is an open, bounded, convex set.  Since edges are only defined between preimages of vertices connected by edges, the map $p$ on vertices of $\E_j$ naturally extends to a simplicial map on $p \colon \E_j \to \B_j$.

\begin{lemma}
\label{lem:E_1-orbits}
    Each of the graphs $\E_1$ and $\E_2$ have finitely many $\Gamma$--orbits of edges.
\end{lemma}

\begin{proof}
There are finitely many $\pi_1S$--orbits of vertical edges in each fiber over $v \in \B_i$, and finitely many $G$--orbits of vertices in $\B_i$, hence finitely many $\Gamma$--orbits of vertical edges, for $i=1,2$.  Similarly, for every cluster edge in $\B_2$, there are finitely many $\pi_1S$--orbits in the preimage, hence finitely many $\Gamma$--orbits of cluster edges.

Now consider any primary horizontal edge between $x,x' \in \E_i^0$, for $i=1,2$.  Such an edge is defined by the intersection of two convex regions, which are preimages in the universal cover of components of the intersections of the complements of $v = p(x)$ and $v'=p(x')$.  There are finitely many such components for each $v,v' \in \E_i^0$, each giving rise to a single $\pi_1S$--orbit of intersections of convex regions. We can now conclude the required finiteness since there are finitely many $G$--orbits of pairs of adjacent vertices $v,v'$ in $\mathcal B_i$.
\end{proof}

\bigskip

Suppose $\widetilde a \in \Omega$, $v,v' \in \B_j^0 \subset \hat \C^0(S)$, $x = \Phi(v,\widetilde a)$, $x' = \Phi(v',\widetilde a)$,  and $e \subset \B_j$ is an edge from $v$ to $v'$, viewed as a geodesic in $\hat \C^1(S)$.
Then there is an edge between $x,x'$, and $\Phi(e \times \{\widetilde a\})$ is a geodesic between $x,x' \subset \hat \C^1(\dot S)$. For each edge $\frak e$ between a pair of vertices $x,x' \in \E^0_j$, we make a choice of $\widetilde a \in \Phi_v^{-1}(x) \cap \Phi_{v'}^{-1}(x')$ and thus define a map $\E_j \to \hat \C^1(\dot S)$, sending the edge to the geodesic $\Phi(e \times \{\widetilde a\})$, where $e = p(\frak e)$. By construction, the following diagram commutes

\begin{center}
\begin{tikzcd}
\E_j \arrow[r] \arrow[d, "p" left] & \hat \C^1(\dot S) \arrow[d, "\Psi"]\\
\B_j \arrow[r] & \hat \C^1(S).
\end{tikzcd}
\end{center}

We will also use this map to view $\E_j$ as a subspace of $\hat \C^1(\dot S)$ when convenient (though again, disjoint edges need not have disjoint images, so it is not really a subspace). The map $p \colon \E_j \to \B_j$ is just the restriction of $\Psi$, via this ``embedding''.

\subsection{Hyperbolicity of tree bundles for PGF groups} 

We now prove that the tree bundles $\E_j$ from the previous section are externally flaring graph bundles. We adopt the convention of using $\B$ (resp. $\E$) to denote either $\B_1$ or $\B_2$ (resp. $\E_1$ or $\E_2$) when the distinction is unimportant.

We define the tight sections $\{\Sigma_{\widetilde a} \colon \B \to \E\}_{\widetilde a \in \Omega}$, by
\[ \Sigma_{\widetilde a}(v) = \Phi(v,\widetilde a) = \Phi_v(\widetilde a),\]
for all $\widetilde a \in \Omega$ and $v \in \B^0 \subset \hat \C^0(S)$.  By definition of the edges in $\B$ and $\E$, this extends over the edges to a simplicial (hence isometric) section.  We denote the fiber over $v \in \B^0$ by $\E_v$.

\begin{remark}
    The definition of tight sections uses a fixed hyperbolic structure on $S$, and it is unclear (and probably not true) that the sections are invariant by the actions of $\Gamma$ in general.  The reason is that the intersection pattern of more than two geodesics on the fixed hyperbolic structure is not likely to be invariant under the action $G$.  In addition, while the map of the vertex set of $\E$ into $\hat \C(\dot S)$ is equivariant with respect to the injection $\Gamma \to \Mod(\dot S)$ it is unclear whether edges map in $\Gamma$--equivariantly (since they are defined, as with the sections, using the intersection configurations of the curves).
\end{remark}

The main result of this section is the following.

\begin{theorem}
\label{thm:combination_applies}
Suppose $(\E,\B) = (\E_j,\B_j)$, $j=1$ or $2$.  The map $p \colon \E \to \B$ together with the tight sections $\{\Sigma_{\widetilde a}\}_{\widetilde a \in \Omega}$ is an externally flaring graph bundle.
\end{theorem}

Most of the conditions of Definition \ref{defn:externally flaring graph bundle} are straightforward, and we verify those first.

\begin{lemma} \label{lem:PGF-Dehn-twist-basic}
    Conditions \eqref{item:hyperbolic}, \eqref{item:full}, \eqref{item:robust}, and \eqref{item:retraction} are satisfied by $p \colon \E \to \B$ and $\{\Sigma_{\widetilde a}\}_{\widetilde a \in \Omega}$.
\end{lemma}
\begin{proof}
    Since $\B$ is quasi-isometric to the coned off Cayley graph of the relatively hyperbolic group $G$, it is hyperbolic.  The fibers $\E_v$, over $v \in \B^0$ are trees, hence are $0$--hyperbolic, and thus \eqref{item:hyperbolic} is satisfied.

    Given a vertex $x \in \E^0$, there exists $\widetilde a \in \Phi_{p(x)}^{-1}(x) \cap \Omega$ and then $\Sigma_{\widetilde a}(p(x)) = x$, and so \eqref{item:full} is satisfied.

    Next, suppose we have tight sections $\Sigma_i = \Sigma_{\widetilde a_i}$, $i=1,2$, defined by $\widetilde a_1,\widetilde a_2 \in \Omega$, together with $v \in \B$ and $x \in [\Sigma_1(v),\Sigma_2(v)]_v$.  Consider the geodesic $[\widetilde a_1,\widetilde a_2] \subset \widetilde S$ and observe that for all $u \in \B$, we have
    \[ [\Sigma_1(u),\Sigma_2(u)]_u \cap \E^0 = \Phi_u([\widetilde a_1,\widetilde a_2] \cap \Omega).\]
    In particular, there exists $\widetilde b \in [\widetilde a_1,\widetilde a_2] \cap \Omega$ such that $\Sigma_{\widetilde b}(v) = \Phi_v(\widetilde b) = x$, and for all $u \in \B^0$, $\Sigma_{\widetilde b}(u) \in [\Sigma_1(u),\Sigma_2(u)]_u$.  Therefore, \eqref{item:robust} is also satisfied. 

    Finally, we claim that \eqref{item:retraction} holds with $\kappa = 1$.  Suppose $x_1,x_2,x_3 \in \E_v$ and $x_1',x_2',x_3' \in \E_{v'}$ with $x_i$ and $x_i'$ connected by an edge of $\E$, for each $i=1,2,3$.  Since $p$ is simplicial, it follows that $v,v'$ are connected by an edge $e$ of $\B$.  For each $i=1,2,3$, let \[ \widetilde a_i \in \Phi_v^{-1}(x_i) \cap \Phi_{v'}^{-1}(x_i') \cap \Omega,\] which exists since $x_i,x_i'$ are connected by an edge.  Since $\Omega$ is an open dense set we assume  (as we may) that $\widetilde a_1,\widetilde a_2,\widetilde a_3$ do not all lie on a single geodesic segment.

    Consider the (nondegenerate) geodesic triangle $\Delta$ in $\widetilde S$ between $\widetilde a_1,\widetilde a_2,\widetilde a_3$ (we view this as a $2$--simplex in $\widetilde S$).  Let $y$ be the closest point projection of $x_3$ to $[x_1,x_2]_v$ in $\E_v$ (and similarly for $y'$).
    Since $\E_v$ is the tree dual to $\uncov^{-1}(v) \subset \widetilde S$, the open convex set $\Phi_v^{-1}(y)$ has non-empty intersection with all three sides of $\Delta$.  See Figure~\ref{Fig:closest point projection}.
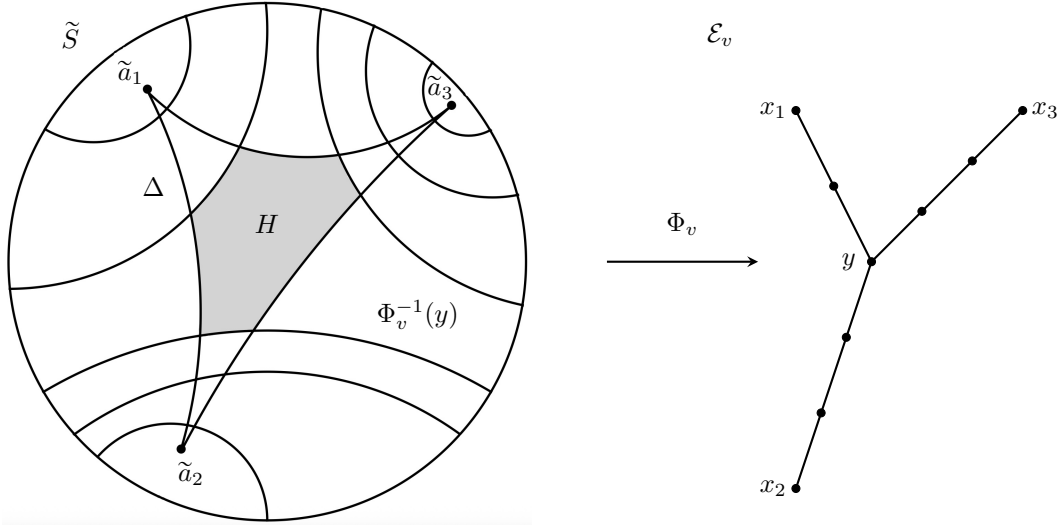
\begin{figure}[h]
\begin{center}
    \begin{tikzpicture}
    \node at (-1,0) {\includegraphics[width=7cm]{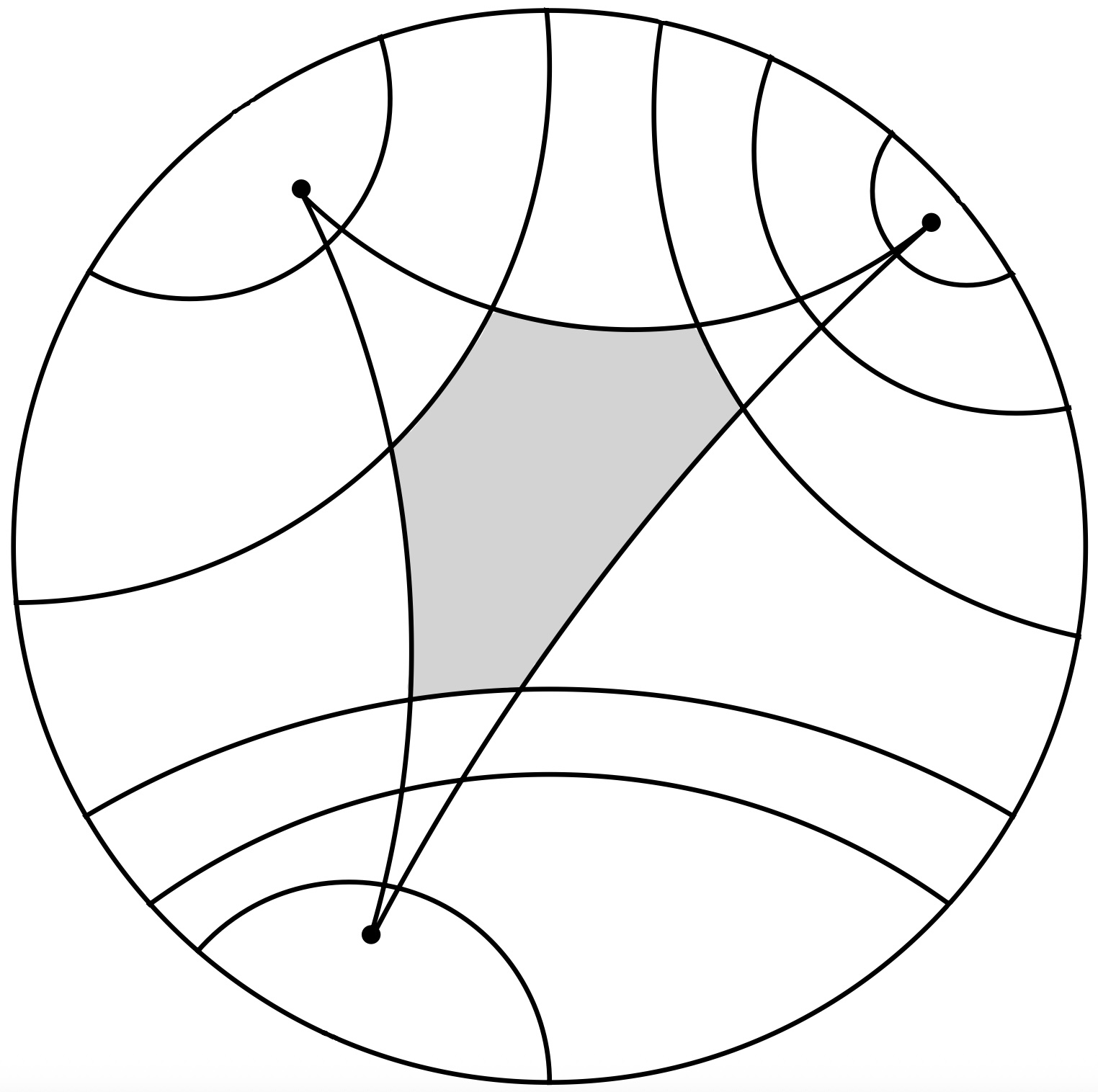}};
    \draw[thick] (6,-3) -- (7,0) -- (6,2);
    \draw[thick] (7,0) -- (9,2);
    \draw[thick, - stealth] (3.5,0) -- (5.5,0);
    \node at (-1,.5) {$H$};
    \node at (4.5,.5) {$\Phi_v$};
    \node at (1,-.7) {$\Phi_v^{-1}(y)$};
    \node at (-2.8,2.5) {$\widetilde a_1$};
    \node at (-2,-2.8) {$\widetilde a_2$};
    \node at (1.3,2.3) {$\widetilde a_3$};
    \node at (5.7,-3) {$x_2$};
    \node at (9.3,2) {$x_3$};
    \node at (5.7,2) {$x_1$};
    \node at (6.7,0) {$y$};
    \node at (-2.5,1) {$\Delta$};
    \node at (-3.6,3) {$\widetilde S$};
    \node at (5,3) {$\E_v$};
    \filldraw (6,-3) circle (1.5pt);
    \filldraw (6.3333,-2) circle (1.5pt);
    \filldraw (6.6666,-1) circle (1.5pt);
    \filldraw (7,0) circle (1.5pt);
    \filldraw (9,2) circle (1.5pt);
    \filldraw (7.66666,.66666) circle (1.5pt);
    \filldraw (8.3333333,1.33333) circle (1.5pt);
    \filldraw (6,2) circle (1.5pt);
    \filldraw (6.5,1) circle (1.5pt);
    \end{tikzpicture}    
\caption{Closest point projection $y$ of $x_3$ to $[x_1,x_2]$ and corresponding regions in $\widetilde S$.} \label{Fig:closest point projection}
\end{center}
\end{figure}
    
    In particular,
    \[ H = \overline{\Phi_v^{-1}(y)} \cap \Delta \]
    is a (possibly degenerate) hexagon, with sides alternating between arcs of sides of $\Delta$ and arcs of the boundary of $\overline{\Phi_v^{-1}(y)}$ (where the latter type of side may degenerate to a vertex of $\Delta$).
    There is similarly a hexagon $H' \subset \Delta$ of intersection of $\overline{\Phi_{v'}^{-1}(y')}$ with $\Delta$.  
    
    By inspection, $H \cap H' \neq \emptyset$; indeed, any two such hexagons in a triangle necessarily have nonempty intersection, and so there is some point $\widetilde b \in H \cap H' \cap \Omega$ with $\Phi_v(\widetilde b) = y$ and $\Phi_{v'}(\widetilde b) = y'$.  It follows that $y$ and $y'$ are connected by an edge in $\E$.
\end{proof}

Before proving condition \eqref{item:flaring} holds, we need to prove a few preliminary facts. {\bf In all of the following, we assume we have fixed $G$, $p \colon \E \to \B$, and $\{\Sigma_{\widetilde a}\}_{\widetilde a \in \Omega}$ as above.}

\begin{lemma} \label{lem:cobounded edges,ex1}
    There exists $C > 0$ so that all horizontal edges of $\E$, viewed as geodesics in $\hat \C(\dot S)$ are $C$--cobounded; that is, $d_Y(x,y) \leq C$, for all subsurfaces $Y \subset \dot S$ and pairs $x,y \in \E^0$ joined by a horizontal edge. 
\end{lemma}
\begin{proof} By assumption, there are only finitely may $G$--orbits of edges of $\B$, which we view as finitely many $G$--orbits of geodesics in $\hat \C(S)$, by $G$--equivariance.  Consequently, there is a $C' >0$ so that $i(v,v') \leq C'$ for any two vertices $v,v' \in \B^0 \subset \hat \C^0(S)$ which are connected by an edge.  By Equation (\ref{Eqn:same i}), we have $i(x,x') \leq C'$ for any two vertices $x,x' \in \E^0 \subset \hat \C(\dot S)$ which are connected by an edge.  Since projection distance is bounded by a linear function of intersection number, it follows that there is some $C$, depending only on $C'$, so that for every essential subsurface $Y \subset \dot S$, we have $d_Y(x,x') \leq C$, as required. 
\end{proof}

\begin{lemma} \label{lem:projections from base}
There is a constant $K >0 $ with the following property.  Given a tight section $\Sigma_{\widetilde a}$ and vertices $v,v' \in \B$, we have
\[ d_Y(\Sigma_{\widetilde a}(v),\Sigma_{\widetilde a}(v')) \leq K,\]
for all subsurfaces $Y \subsetneq \dot S$, except possibly when $\pi_Y(x) = \emptyset$ for some vertex $x$ in $\Sigma_{\widetilde a}([v,v'])$ where $[v,v']$ is a geodesic in $\B$.  
\end{lemma}

\begin{proof}
The statement when $\B = \B_2$ easily follows from the case when $\B = \B_1$, so we assume $\B = \B_1$ in the proof.  Suppose $v = v_0,\ldots,v_n = v'$ are the vertices in a geodesic in $\B$ between $v$ and $v'$, and write $x = \Sigma_{\widetilde a}(v)$, $x' = \Sigma_{\widetilde a}(v')$, and $x_j = \Sigma_{\widetilde a}(v_j)$, for all $j$, so that $x = x_0,\ldots,x_n = x'$
are vertices of the lifted geodesic in the tight section $\Sigma_{\widetilde a}$.

To prove the lemma, it suffices to show that if $Y \subsetneq \dot S$ is any subsurface with $\pi_Y(x_j) \neq \emptyset$ for all $j$, then $d_Y(x,x')$ is uniformly bounded by some constant $K$.
To prove this, we first note that since $[x_j,x_{j+1}]$ is $C$--cobounded by Lemma~\ref{lem:cobounded edges,ex1}, we have
\[ d_Y(x_j,x_{j+1}) \leq C\]
for all $j$.  Therefore by the triangle inequality, we have 
\[ d_Y(x_i,x_j) \leq C|j-i|,\]
for all $0 \leq i < j \leq n$.

On the other hand, $\Sigma_{\widetilde a}([v,v'])$ is a uniform quasi-geodesic in $\hat C^1(\dot S)$.  By Gromov hyperbolicity, this is uniformly close (in $\hat C(\dot S)$) to a geodesic in $\hat C^1(\dot S)$ between the endpoints.  This geodesic in $\hat \C^1(\dot S)$ is itself uniformly close to a geodesic $\gamma$ in $\C^1(\dot S)$, for which the endpoints are curves contained in $x$ and $x'$, respectively, and for which we may assume have non-empty projection to $Y$ realizing $d_Y(x,x')$.

If every vertex of $\gamma$ has non-empty projection to $Y$, then $d_Y(x,x') \leq M$, by Theorem~\ref{thm:MM BGI}.  If some vertex of $\gamma$ has empty projection to $Y$, then we can express $\Sigma_{\widetilde a}([v,v'])$ as a concatenation of three subgeodesics 
\[[x_0,x_i][x_i,x_j][x_j,x_n] = \Sigma_{\widetilde a}([v,v']) \subset  \Sigma_{\widetilde a}(\B),\]
with the following properties.
The difference $|j-i|$ is uniformly bounded by some $N$, and the initial and terminal segments, $[x_0,x_i]$ and $[x_j,x_n]$, are close enough to subsegments $\gamma',\gamma''\subset\gamma$, all of whose vertices have non-empty projections to $Y$, so that we can connect a curve in $x_i$ to a vertex $y'$ of $\gamma'$ and $x_j$ to a vertex $y''$ in $\gamma''$ by geodesics all of whose vertices have non-empty projection to $Y$.  The triangle inequality and several applications of Theorem~\ref{thm:MM BGI} implies
\[ d_Y(x,x') \leq d_Y(x,y') + d_Y(y',x_i) + d_Y(x_i,x_j) + d_Y(x_j,y'') + d_Y(y'',x') \leq  4M + CN.\]
This is the required uniform bound on $d_Y(x,x')$.
\end{proof}

In the next section it will be useful to have an analogous statement for geodesics in $\B$ and subsurfaces of $S$.  The proof is essentially identical and so we do not repeat it, but state the lemma here for later reference.  See also \cite[Proposition~4.3]{Udall} for a stronger version.

\begin{lemma} \label{lem:bounded projections in B}
    There is a constant $K > 0$ with the following property.  Given $v,v' \in \B$, we have
    \[ d_Y(v,v') \leq K,\]
    for all subsurfaces $Y \subsetneq S$, except possibly when $\pi_Y(w) = \emptyset$ for some vertex $w$ in a geodesic $[v,v'] \subset \B$. \qed
\end{lemma}

Let $\delta_\C$ denote the uniform hyperbolicity constant of curve complexes; see the comments after Theorem~\ref{thm:MM hyperbolic}.

\begin{lemma} \label{lem:close sections close}
There exists $R> 0$ depending only on $\delta = \delta_\C$ and the quasi-isometry constants for $\B \to \hat \C(S)$ having the following property.  If $\mu \geq 0$, $u,v \in \B$, and $\Sigma_1,\Sigma_2$ are tight sections with 
\[ d_{\dot S}(\Sigma_1(u),\Sigma_2(u)),d_{\dot S}(\Sigma_1(v),\Sigma_2(v))\leq \mu, \]
and if $w \in [u,v]$ is any point with $d_S(u,w),d_S(v,w) \geq 2(\mu+\delta)$, then
\[ d_{\dot S}(\Sigma_1(w),\Sigma_2(w)) \leq R.\]
\end{lemma}
\begin{proof}
Let $w \in [u,v]$ be any point as in the lemma. Then $\Sigma_1([u,v])$ and $\Sigma_2([u,v])$ are uniform quasi-geodesics in $\hat \C(\dot S)$, each of which connect points coarsely at least distance $4(\mu+\delta)$ apart.  On the other hand, the initial points $\Sigma_1(u),\Sigma_2(u)$ and terminal points $\Sigma_1(v),\Sigma_2(v)$ are coarsely closer (at most $\mu$ apart), and hence there is some point $w' \in [u,v]$ so that  $d_{\dot S}(\Sigma_1(w),\Sigma_2(w'))$ is uniformly bounded by some constant $R'$ (using a slim quadrilateral).  Now observe that since $\Psi$ is $1$--Lipschitz, we have
\[ d_S(w,w') \leq d_{\dot S}(\Sigma_1(w),\Sigma_2(w')) \leq R'.\]
Then by the triangle inequality
\begin{eqnarray*}
d_{\dot S}(\Sigma_1(w),\Sigma_2(w)) & \leq & d_{\dot S}(\Sigma_1(w),\Sigma_2(w')) + d_{\dot S}(\Sigma_2(w'),\Sigma_2(w))\\
& = &d_{\dot S}(\Sigma_1(w),\Sigma_2(w')) + d_S(w',w)\\
& \leq & R'+R'.
\end{eqnarray*}
Therefore, $d_{\dot S}(\Sigma_1(w),\Sigma_2(w))$ is uniformly bounded by $R = 2R'$.
\end{proof}

We will also need the following fact, which essentially states that the required flaring condition holds if we measure distances in $\hat \C(\dot S)$, instead of in $\E$.  

\begin{lemma} \label{lem:weak flare}
    There exists $\sigma_0 > 0$, so that for all $B_0 \geq \sigma_0$, there exists $\lambda_0 > 0$ with the following property.
    If $\gamma = [u,v]$ is a geodesic in $\B$, and $\Sigma_1,\Sigma_2$ are two tight sections with   $d_{\dot S}(\Sigma_1(v),\Sigma_2(v)) \leq B_0$ and $d_{\dot S}(\Sigma_1(w),\Sigma_2(w)) \geq \sigma_0$ for all vertices $w \in \gamma -\{v\}$, then 
    \[ d_{\dot S}(\Sigma_1(u),\Sigma_2(u)) \geq d_S(u,v) - \lambda_0.\]
\end{lemma}
\begin{proof}
By composing a geodesic $\gamma$ in $\B$ with a tight section $\Sigma_{\widetilde a}$, we obtain a uniform quasi-geodesic in $\hat \C(\dot S)$.  Combining this with Lemma~\ref{lem:morse}, it follows that there exists some $\sigma_1 > 0$ so that if $\gamma = [u,v]$ is a geodesic in $\B$ and $\Sigma_1,\Sigma_2$ are two tight sections with $d_{\dot S}(\Sigma_1(w),\Sigma_2(w)) \geq \sigma_1$ for all $w \in \gamma -\{v\}$, then the $\hat \C(\dot S)$--geodesics, $[\Sigma_1(u),\Sigma_1(v)]$ and $[\Sigma_2(u),\Sigma_2(v)]$ are much farther than $\delta_\C$ apart. 
In particular, we may take $\sigma_1$ large enough so that in a finite Gromov approximating tree for the points $\Sigma_1(u)$, $\Sigma_1(v)$, $\Sigma_2(u)$, and $\Sigma_2(v)$, which approximates distances between the points
up to an additive error $\epsilon =\epsilon(\delta_\C)$, the geodesic between comparison points for $\Sigma_1(u),\Sigma_1(v)$ and the geodesic between comparison points for $\Sigma_2(u),\Sigma_2(v)$, are disjoint.
\begin{figure}[h]
\begin{center}
    \begin{tikzpicture}
    \begin{scope}[scale=.5]
    \draw[thick] (-1,0) -- (3,1) -- (6,0);
    \draw[thick] (-1,3) -- (3,2) -- (6,3);
    \draw[thick] (3,1) -- (3,2);
    \node at (-2,0) {$\Sigma_1(u)$};
    \node at (-2,3) {$\Sigma_2(u)$};
    \node at (7,0) {$\Sigma_1(v)$};
    \node at (7,3) {$\Sigma_2(v)$};
    \node at (5,1.5) {$\leq B_0$};
    \end{scope}
    \end{tikzpicture}    
\caption{The $\epsilon$--approximating tree for $\sigma_0 \geq \sigma_1$ sufficiently large and $B_0 \geq \sigma_0$.} \label{Fig:ambient flaring}
\end{center}
\end{figure}
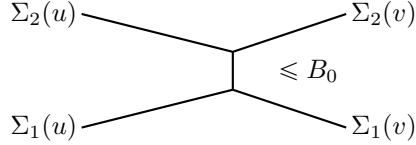

Now, suppose that $\sigma_0 > \sigma_1$ is large enough so that for any $B_0 \geq \sigma_0$, if $\gamma = [u,v]$ is a geodesic and $\Sigma_1,\Sigma_2$ two tight sections with $d_{\dot S}(\Sigma_1(w),\Sigma_2(w)) \geq \sigma_0$ for all $w \in \gamma - \{v\}$, then $d_{\dot S}(\Sigma_1(w),\Sigma_2(w)) \geq \sigma_1$ for all $w \in \gamma - \{v\}$.
By inspection $\lambda_0 = B_0+3 \epsilon$ suffices to prove the lemma; see Figure~\ref{Fig:ambient flaring}.
\end{proof}

The next lemma establishes Condition~\eqref{item:flaring} (with the associated function a linear function).  This is the only remaining condition needed to verify the hypotheses of an externally flaring graph bundle.

\begin{lemma}
\label{lem:flare}
There exists $\sigma>0$ such that for all $B\geq \sigma$ there exists $\lambda \geq 1$ with the following property. Consider tight sections $\Sigma_1,\Sigma_2$ and a geodesic $\gamma=[u,v]\subset \B$ such that for all $w\in\gamma-\{v\}$ we have $d_\E(\Sigma_1(w),\Sigma_2(w))\geq\sigma$, while $d_\E(\Sigma_1(v),\Sigma_2(v))\leq B$. Then
    $$d_\E(\Sigma_1(u),\Sigma_2(u))\geq d_\B(u,v)/\lambda -\lambda.$$
\end{lemma}

\begin{proof}
The ``inclusion" $\E \to \hat \C(\dot S)$ is $\eta$--Lipschitz for some $\eta \geq 2$.  To see this, we can take $\eta$ to be the maximum of lengths in $\hat \C(\dot S)$ over all horizontal edge of $\E$ (vertical edges are mapped to length $2$ paths in $\hat \C(\dot S)$).  To see that such a maximum exists, observe that this is the same as the maximal length in $\hat \C(S)$ of any edge in $\B$, and this quantity is bounded since there are only finitely many $G$--orbits of edges of $\B$ by assumption.

Let $\sigma_1 = \sigma_0/\eta$, where $\sigma_0$ is from Lemma~\ref{lem:weak flare}.  For any $B$ with $\eta B \geq \eta \sigma_1=\sigma_0$, let $\lambda_0$ be as in the same lemma (i.e.~for $B_0 = \eta B$).
That is, if $\Sigma_1$ and $\Sigma_2$ are tight sections and $\gamma = [u,v]$ is a geodesic in $\B$, such that $d_{\dot S}(\Sigma_1(w),\Sigma_2(w)) > \eta \sigma_1$ for all $w \in \gamma - \{v\}$, and $d_{\dot S}(\Sigma_1(v),\Sigma_2(v)) \leq \eta B$, then we have 
\[ d_{\dot S}(\Sigma_1(u),\Sigma_2(u)) \geq d_S(u,v) - \lambda_0.
\]

Next, let $R$ be as in Lemma~\ref{lem:close sections close}, $K$ as in Lemma~\ref{lem:projections from base}, and $M$ the constant from Theorem~\ref{thm:MM BGI}. Then there exists $\sigma > 0$ so that if $x,y \in p^{-1}(v)$ for some $v \in \B^0$ and if $d_\E(x,y) \geq \sigma$ while $d_{\dot S}(x,y) \leq R$, then for some proper subsurface $Y \subset \dot S$, we have $d_Y(x,y) > M+2K$.  To see this, note that if $d_\E(x,y)$ is large, then the distance between $x$ and $y$ in the tree $p^{-1}(v) \cong T_v$ is large, and hence the curves $x,y$ must have large intersection number.  This then implies the existence of a large subsurface projection to some subsurface  (c.f.~\cite[Corollary D]{ChoiRafi}).  Without loss of generality, we may assume that $\sigma \geq \eta\sigma_1$ and that $\eta \sigma > R+2$.

Now suppose $B \geq \sigma \geq \sigma_1$, let $\lambda_0$ be as above, and suppose we have a geodesic $\gamma = [u,v]$ and two tight sections $\Sigma_1,\Sigma_2$ satisfying the hypotheses of the lemma for our choice of $\sigma$.
For $i=1,2$, let $\gamma_i=\Sigma_i(\gamma)$, and for any vertex $w \in \gamma$ we denote the corresponding point in this section with the associated subscript, $w_i = \Sigma_i(w) \in \gamma_i$.  Thus, from the hypotheses of the lemma, we have
\begin{equation} \label{eqn:hypotheses repeated}
    d_\E(v_1,v_2) \leq B \quad \mbox{ and } \quad d_\E(w_1,w_2) \geq \sigma,
\end{equation}
for all vertices $w \in \gamma - \{v\}$.  Since inclusion of $\E$ into $\hat \C(\dot S)$ is $\eta$--Lipschitz, we have $d_{\dot S}(v_1,v_2) \leq \eta B$.
Let $w'$ be the closest point to $u$ in $\gamma$ such that $d_{\dot S}(w'_1,w'_2)\leq \eta  B$.  
Then $\gamma' = [u,w']$ is a geodesic satisfying $d_{\dot S}(w'_1,w'_2)\le \eta B$ and $d_{\dot S}(w_1,w_2)> \eta B$ for all $w\in \gamma'- \{w'\}$.  
See Figure~\ref{Fig:configuration of sections}.

\def\arr{-{Stealth[length=3mm, width=2mm]}}
\begin{figure}[htb]
\begin{center}
\begin{tikzpicture}
\draw[thick] (0,-.2) -- (1,0) -- (2,-.2) -- (3,-.2) -- (4,.2) -- (5,.2) -- (6,0) -- (7,.1) -- (8,.3) -- (9,.2) -- (10,0) -- (11,.3) -- (12,.3);
\draw[thick] (0,2.5) -- (1,2) -- (2,2) -- (3,2.2) -- (4,1.8) -- (5,2) -- (6,1.8) -- (7,1.8) -- (8,1.6) -- (9,1.8) -- (10,2) -- (11,1.8) -- (12,1.7);
\draw[thick] (0,-1.2) -- (12,-1.2);
\filldraw (0,-.2) circle (.05);
\filldraw (0,2.5) circle (.05);
\filldraw (12,.3) circle (.05);
\filldraw (12,1.7) circle (.05);
\filldraw (4,.2) circle (.05);
\filldraw (4,1.8) circle (.05);
\filldraw (3,-.2) circle (.05);
\filldraw (3,2.2) circle (.05);
\filldraw (0,-1.2) circle (.05);
\filldraw (3,-1.2) circle (.05);
\filldraw (4,-1.2) circle (.05);
\filldraw (12,-1.2) circle (.05);
\node at (0,-.5) {$u_1$};
\node at (0,2.8) {$u_2$};
\node at (12,0) {$v_1$};
\node at (12,2) {$v_2$};
\node at (4,-.1) {$w_1'$};
\node at (4,2.1) {$w_2'$};
\node at (3,-.5) {$w_1$};
\node at (3,2.5) {$w_2$};
\node at (8,1.9) {$\gamma_2$};
\node at (8,0) {$\gamma_1$};
\node at (0,-1.5) {$u$};
\node at (12,-1.44) {$v$};
\node at (4,-1.44) {$w'$};
\node at (3,-1.5) {$w$};
\node at (8,-1.5) {$\gamma$};
\draw[dotted] (12,.3) -- (12,1.7);
\draw[dotted] (4,.2) -- (4,1.8);
\draw[dotted] (3,-.2) -- (3,2.2);
\draw[dotted] (0,-.2) -- (0,2.5);
\node at (12.5,1) {$\leq \eta B$};
\node at (4.5,1) {$\leq \eta B$};
\node at (2.5,1) {$> \eta B$};
\node at (-.5,1.2) {$> \eta B$};
\end{tikzpicture}
\caption{}
\label{Fig:configuration of sections}
\end{center}
\end{figure}

Since $B \geq \sigma_1$, $\eta B \geq \eta \sigma_1 = \sigma_0$ it follows that 
\[ d_{\dot S}(u_1,u_2) \geq d_S(u,w') - \lambda_0.\]
Since $d_S(u,w')$ and $d_\B(u,w')$ are comparable, and since $d_{\dot S}(u_1,u_2)$ is a coarse lower bound for $d_\E(u_1,u_2)$, it suffices to show that $d_S(v,w')$ is bounded above in terms of $B$, since the triangle inequality implies $d_S(u,w') \geq d_S(u,v) - d_S(v,w')$.
    
Suppose $d_S(v,w')$ is not bounded in terms of $B$.  In particular, we may assume that there is some $w'' \in [w',v] \subset \gamma$ with
\[ d_S(w'',w'),d_S(w'',v) \geq 2(\eta B+\delta),\]
where $\delta = \delta_\C$.
Letting $\mu = \eta B$, and applying Lemma~\ref{lem:close sections close}, we have $d_{\dot S}(w''_1,w''_2)\leq R$.  See Figure~\ref{Fig:where is Y}. 

\def\arr{-{Stealth[length=3mm, width=2mm]}}
\begin{figure}[htb]
\begin{center}
\begin{tikzpicture}
\draw[thick] (4,.2) -- (5,.2) -- (6,0) -- (7,.1) -- (8,.3) -- (9,.2) -- (10,0) -- (11,.3) -- (12,.3);
\draw[thick] (4,1.8) -- (5,2) -- (6,1.8) -- (7,1.8) -- (8,1.6) -- (9,1.8) -- (10,2) -- (11,1.8) -- (12,1.7);
\draw[thick] (4,-.95) -- (12,-.95);
\filldraw (12,.3) circle (.05);
\filldraw (12,1.7) circle (.05);
\filldraw (4,.2) circle (.05);
\filldraw (4,1.8) circle (.05);
\filldraw (4,-.95) circle (.05);
\filldraw (8,.3) circle (.05);
\filldraw (8,1.6) circle (.05);
\filldraw (12,-.95) circle (.05);
\filldraw (8,-.95) circle (.05);
\node at (4,-.1) {$w_1'$};
\node at (4,2.1) {$w_2'$};
\node at (8,0) {$w_1''$};
\node at (8,1.9) {$w_2''$};
\node at (12,0) {$v_1$};
\node at (12,2) {$v_2$};
\node at (4,-1.2) {$w'$};
\node at (12,-1.2) {$v$};
\node at (8,-1.2) {$w''$};
\node at (6,-1.3) {$\geq 2(\eta B+\delta)$};
\node at (10,-1.3) {$\geq 2(\eta B+\delta)$};
\draw[dotted] (12,.3) -- (12,1.7);
\draw[dotted] (4,.2) -- (4,1.8);
\draw[dotted] (8,1.6) -- (8,.3);
\node at (12.5,1) {$\leq \eta B$};
\node at (3.5,1) {$\leq \eta B$};
\node at (8.5,1) {$\leq R$};
\end{tikzpicture}
\caption{}
\label{Fig:where is Y}
\end{center}
\end{figure}

By \eqref{eqn:hypotheses repeated} and our choice of $\sigma$, there exists some subsurface $Y \subset \dot S$ so that $d_Y(w_1'',w_2'') > M+ 2K$.  We claim that this $Y$ must satisfy 
\begin{equation}\label{eq:bdd_subsurf_at_ends}
d_Y(w'_1, w'_2) , d_Y(v_1,v_2)\le M.
\end{equation}
Indeed, if $d_Y(w'_1,w'_2) > M$, then Theorem~\ref{thm:MM BGI} implies any $\hat \C(\dot S)$ geodesic between $w'_1,w'_2$ passes within distance $1$ of $\partial Y$. The same holds for any $\hat\C(\dot S)$ geodesic between $w''_1,w''_2$; hence these two geodesics pass within distance $2$ of each other. Since $d_{\dot S}(w'_1,w'_2)\le \eta B$ and $d_{\dot S}(w''_1,w''_2)\le R$, the triangle inequality thus implies $d_{\dot S}(w'_1,w''_1) \le \eta B + R + 2$. However, since we have arranged that $\eta B \ge \eta \sigma > R+2$, this contradicts  $d_{\dot S}(w'_1,w''_1) \ge 2(\eta B + \delta)$. Assuming $d_Y(v_1,v_2) > M$ similarly leads to a contradiction.

Next we claim that
\begin{equation}\label{eq:bdd_subsurf_left_or_right}
d_Y(w_1',w_1''),d_Y(w_2',w_2'') \leq K
\qquad\text{or}\qquad
d_Y(w_1'',v_1),d_Y(w_2'',v_2) \leq K.
\end{equation}
To see this, first suppose $d_Y(w_1',w_1'') > K$.  By Lemma~\ref{lem:projections from base}, there is a vertex $x$ in $[w'_1,w''_1] \subset [w_1',v_1]$ so that $\pi_Y(x) = \emptyset$.  Then $\pi_{p(Y)}(p(x)) = \emptyset$.  It follows from Lemma~\ref{lem:vertices filling} that for any vertex $y$ in either $[w''_1,v_1]$ or $[w''_2,v_2]$, we have 
$\pi_{p(Y)}(p(y)) \neq \emptyset$, and hence $\pi_Y(y) \neq \emptyset$. By Lemma~\ref{lem:projections from base}, the second inequality above must hold.  A similar argument proves the other cases.

By the triangle inequality, combining Equations \eqref{eq:bdd_subsurf_at_ends} and \eqref{eq:bdd_subsurf_left_or_right} now gives 
\[ d_Y(w_1'',w_2'') \leq M+2K.\]
This is a contradiction, hence $d_S(v,w')$ is bounded in terms of $B$, which completes the proof.
\end{proof}

\section{Hierarchical hyperbolicity of extensions of rank 1 PGF groups}
\label{sec:cyclic PGF HHG}

We now turn to verifying the hierarchical hyperbolicity of the extensions of parabolically geometrically finite groups with cyclic peripheral subgroups, Theorem \ref{thm:intro_PGF_hyp}. For this we employ the combinatorial HHS machinery from \cite{BHMS}.

\subsection{Combinatorial hierarchical hyperbolicity}

In this subsection, we recall all relevant definitions regarding combinatorial hierarchical hyperbolicity. See \cite{BHMS}, and in particular the ``User's guide'' for more information and motivation. We start with purely graph theoretical notions.

\begin{definition}[Join, link, and star]
	Let $X$ be a flag simplicial complex. If $Y,Z$ are disjoint flag subcomplexes of $X$ so that every vertex of $Y$ is joined by an edge to $Z$, then the \emph{join} of $Y$ and $Z$, $Y \ast Z$, is the subcomplex of $X$ spanned by $Y$ and $Z$. Given a simplex $\Delta$ of $X$, the \emph{link} of $\Delta$, $\lk(\Delta)$, is the subcomplex of $X$ spanned by the vertices of $X$ that are joined by an edge to all the vertices of $\Delta$. The \emph{star} of $\Delta$, $\st(\Delta)$, is the join $\Delta \ast \lk(\Delta)$. We consider $\emptyset$ as a simplex of $X$ whose link and star are both $X$.
\end{definition}

A combinatorial HHS will be a pair $(X,W)$ where $X$ is a flag simplicial complex, and $W$ is a graph whose vertices are the maximal simplices of $X$. Here, the HHS is $W$, while $X$ encodes the HHS structure of $W$. With this in mind, we give the next definition.

\begin{definition}
	Given a flag simplicial complex $X$, an \emph{$X$--graph} is any graph $W$ whose vertices are maximal simplices of $X$. Here maximal means not contained in a larger simplex.
	
	If $W$ is a $X$--graph, we define the \emph{$W$--augmented graph} $X^{+W}$ as the graph with the same vertex set as $X$ and with two types of edges: 
	\begin{enumerate}
		\item ($X$--edge) If two vertices $x_1,x_2 \in X$ are joined by an edge in $X$, then $x_1$ and $x_2$ are joined by an edge in $X^{+W}$.
		\item ($W$--edge) If $\Delta_1$ and $\Delta_2$ are maximal simplices of $X$ that are joined by an edge in $W$, then each vertex of $\Delta_1$ is joined by an edge to each vertex of $\Delta_2$ in $X^{+W}$.
	\end{enumerate}
	We note that if a group $G$ acts by simplicial automorphisms on $X$ and that action extends to an action by isometries on $X^{+W}$, then there is an induced action by isometries of $G$ on $W$. 
\end{definition}

The next definition is used to encode the index set of the HHS using equivalence classes of simplicies in $X$.

\begin{definition}\label{definition:saturation}
	Let $\Delta$ and $\Delta'$ be simplices of the flag simplicial complex $X$. We write $\Delta \sim \Delta'$ if $\lk(\Delta) = \lk(\Delta')$. We define the \emph{saturation} of $\Delta$, $\Sat(\Delta)$, to be the set of vertices of $X$ contained in a simplex in the $\sim$--equivalence class of $\Delta$. That is $x \in \Sat(\Delta)$  if and only if there exists $\Delta'  \sim \Delta$  so that $x$ is  a vertex of $\Delta'$.
\end{definition}

The spaces $\C(\Delta)$ appearing in the next definition are the hyperbolic spaces of the HHS structure.

\begin{definition}\label{definition:Y_Delta}
	Let $W$ be a $X$--graph. For each simplex $\Delta$ of $X$, define $Y_\Delta$  to be the subgraph of $X^{+W}$ spanned by the vertices of $X^{+W} - \Sat(\Delta)$. 
	
	Define $\C(\Delta)$ to be the subgraph of $Y_\Delta$ spanned by the vertices in $\lk(\Delta)$. Note, we are taking the link in $X$, not in $X^{+W}$, and then considering the subgraph of $Y_\Delta$ induced by those vertices. We give $\C(\Delta)$ its intrinsic path metric (as opposed to the metric induced as a subset of $Y_\Delta$). By construction, we have $\C(\Delta) = \C(\Delta')$ whenever $\Delta \sim \Delta'$.	Note, since $\emptyset$ is a simplex of $Y$ with $\lk(\emptyset) = X$, we have $Y_\emptyset = \C(\emptyset) = X^{+W}$. 
\end{definition}

Finally, we are ready to give the definition of a combinatorial HHS.

\begin{definition}\label{defn:CHHS}
	Let $\delta \geq 0$, $X$ be a flag simplicial complex and $W$ be a $X$--graph. The pair $(X,W)$  is a \emph{$\delta$--combinatorial HHS} if the following are satisfied.
	\begin{enumerate}
		\item \label{CHHS:finite_complexity} Any chain of the form $\lk(\Delta_1) \subsetneq \lk(\Delta_2) \subsetneq \dots$ has length at most $\delta$.
		\item \label{CHHS:hyp_X} For each non-maximal simplex $\Delta \subset X$, the space $\C(\Delta)$ is $\delta$--hyperbolic.
		\item \label{CHHS:geom_link_condition} For each non-maximal simplex $\Delta$, the inclusion $\C(\Delta) \to Y_\Delta$ is a $(\delta,\delta)$--quasi-isometric embedding.
		\item \label{CHHS:comb_nesting_condition} Whenever $\Delta_1$ and $\Delta_2$ are non-maximal simplices of $X$, there exists a (possibly empty) simplex $\Theta$ of $\lk(\Delta_1)$ such that $\lk(\Delta_1 \ast \Theta) \subset \lk(\Delta_2)$ and for all non-maximal simplices $\Lambda$ of $X$ so that $\lk(\Lambda) \subset \lk(\Delta_1) \cap \lk(\Delta_2)$ either
		\begin{enumerate}
			\item $\mathrm{diam}(\C(\Lambda))\leq \delta$ or;
			\item $\lk(\Lambda) \subset \lk(\Delta_1 \ast \Theta)$.
		\end{enumerate}
		\item \label{CHHS:C=C_0_condition} For each non-maximal simplex $\Delta \subset X$ and $x,y \in \lk(\Delta)$, if $x$ and $y$ are not joined by a $X$--edge of $X^{+W}$, but are joined by a $W$--edge of $X^{+W}$, there then exist simplices $\Lambda_x, \Lambda_y \subset \lk(\Delta)$ so that  $x \in \Lambda_x$, $y \in \Lambda_y$, and $\Delta \ast \Lambda_x$ is joined by an edge of $W$ to $\Delta \ast \Lambda_y$.
	\end{enumerate}
\end{definition}

\subsection{The Combinatorial HHS for rank 1 PGF groups.}

Let $G < \Mod(S)$ be a PGF with cyclic peripheral subgroups. Let $\cH$ be the collection of peripheral subgroups of $G$.  For each $H \in \cH$, the $\pi_1S$--extension $\Gamma_H$ is a natural subgroup of $\Gamma_G$. When the peripherals are cyclic, the construction of $\B_1$ and $\B_2$ from Section \ref{subsec_E_for_RGF} are identical, so we let $\B = \B_1=\B_2$.  Recall from Section \ref{subsec_E_for_RGF}, we can identify each vertex  $v \in \B$ with the multicurve on $S$ that is stabilized by $\stab_G(v)$. For each $v \in \B^{0}$ there is a unique $H \in \cH$ so that $\stab_G(v)$ is a conjugate to $H$ in $G$.  As is Section \ref{subsec:cc fibrations}, let $T_v$ be the tree $\Psi^{-1}(v)$ for each $v\in \B^{0}$. 

 For each $H \in \cH$, $\Gamma_H$ is   isomorphic to the fundamental group of a non-geometric graph manifold. Such graph manifold groups have the structure of an \emph{admissible graph of groups}; see Definition 2.13 of\cite{HRSS}. Admissible graph of groups are  given a combinatorial HHS structure in \cite{HRSS}.  We will build the combinatorial HHS structure for $\Gamma_G$ by combining the combinatorial HHS structures for all the cosets of the $\Gamma_H$.

We start by recalling the combinatorial HHS structure  for $\Gamma_H$ that was established in \cite{HRSS}. 
Fix $H \in \cH$ and let $v$ be the multicurve stabilized by $H$. For convenience of notation, let $\Gamma_v$ be $\Gamma_H$. Let $x_1,\dots,x_n$ be representatives of the finitely many $\Gamma_v$--orbits in $T_v^{0}$. Thus each $x \in T_v^{0}$ is contained in a unique $\Gamma_v$--coset of $\stab_{\Gamma_v}(x_i)$ for one of the $x_i$. Let $\A_x$ be a copy of the elements of this coset. We now define the graph $X_v$:
\begin{itemize} 
    \item The vertex set is  $T_v^{0} \cup \bigsqcup_{x \in T_v^{0}} \A_x.$
    \item For each $x \in T_v^{0}$, there is an edge between every element of $\A_x$ and $x$.
    \item If $x,y \in T_v^{0}$ are joined by an edge of $T_v$,  then every vertex in $\A_x \cup \{x\}$ is joined by an edge to every vertex of  $\A_y \cup \{y\}$.
\end{itemize}
The graph $X_v$ is exactly the simplicial complex defined in \cite{HRSS} for the combinatorial HHS structure for $\Gamma_H =\Gamma_v$. Let $W_v$ be the $X_v$--graph described in \cite{HRSS} for $\Gamma_v$. The following proposition contains the salient properties of $W_v$ that we will need:
\begin{proposition}[{\cite[Theorem 3.1]{HRSS}}]\label{prop:HRSS facts}
    The following hold:
    \begin{enumerate}
    \item $(X_v,W_v)$ is a combinatorial HHS;
    \item $\Gamma_v$ acts on $X_v$ with finitely many orbits of links of simplices;
    \item the action of $\Gamma_v$ on $X_v$ induces a metrically proper and cobounded action of $\Gamma_v$ on $W_v$. In particular, there is a $\Gamma_v$--invariant quasi-isometry $P_v \colon W_v \to \Gamma_v$ where $\Gamma_v$ has any finitely generated word metric.
\end{enumerate} 
    
\end{proposition}

We now use the combinatorial HHSs $(X_v,W_v)$ to build the combinatorial HHS $(X,W)$ for all of $\Gamma_G$.   Let $u \in \B^{0}$. Recall, there exist  $v \in \B^{0}$ so that $v$ is stabilized by $H \in \cH$ and $u$ is in the $G$ orbit of $v$.  Now fix a $g\in \Gamma_G$ so that $gT_v = T_u$. Here we view $g$ as an element of $\Mod(\dot S)$ and have  $g$ act on each curve that is a vertex of $T_v$.
For each $x\in T_v^{0}$, let $\A_{gx}$ be the set of elements of $\Gamma_G$ obtained by multiplying each element of $\A_x$ on the left by $g$. That is, $\A_{gx} = g\A_x$. We can now define $X_{u}$ completely analogously to how we defined $X_v$ so that $X_u$ is the $g$--translate of $X_v$. We can similarly define $W_{u}$ as the $g$--translate of $W_v$.

Define $X$ to be the disjoint union of the $X_u$ over $u \in \B^{0}$. The space $W$ will be a space obtained from the disjoint union $\bigsqcup_{u \in \B^{0}} W_u$ by adding additional edges. To add these edges, observe that each vertex $\Delta \in W_u$ uses a single edge in $T_u$. As described in Section \ref{subsec:cc fibrations}, the edges of $T_u$ are in bijection with the lifts of the curves in $u$ to bi-infinite geodesics in $\mathbb{H}$. Let $\gamma_\Delta$ be the $\mathbb{H}$--geodesic corresponding to the $T_u$--edge that is in $\Delta$. We now let $W$ be any graph obtained  from $\bigsqcup_{u \in \B^{0}} W_u$ by adding new edges to achieve the following:
\begin{itemize}
    \item $W$ is connected;
    \item the action of $\Gamma_G$ on $X$ induces an action by graphical automorphism on $W$ with only finitely many $\Gamma_G$--orbits of new edges;
    \item whenever $\Delta \in W_u$ and $\Lambda \in W_v$ are joined by a new edge, $u$ and $v$ are adjacent in $\B$;
     \item whenever $\Delta \in W_u$ and $\Lambda \in W_v$ are joined by a new edge $\gamma_\Delta$ intersects $\gamma_\Lambda$ in $\mathbb{H}$.
\end{itemize}

One can make such a choice of new edges as follows. Let $\V$ be a set of representatives of the finitely many $G$--orbits of edges of $\B$. For each pair of vertices $v,u$ that are joined by an edge of $\V$ choose vertices $\Delta\in W_v$ and $\Lambda\in W_u$ so that $\gamma_\Delta$ intersects $\gamma_\Lambda$. Add a new edge between this choice of $\Delta$ and $\Lambda$.  We then add every edge in the $\Gamma_G$--orbit of this finite collection of new edges. Since $\B$ is connected, we can use the edges we described to connect any two given $W_u$ and $W_v$ by an edge-path. Since each $W_u$ is already connected, this choice therefore makes $W$ connected.

To see why we can always find  $\Delta\in W_v$ and $\Lambda\in W_u$ so that $\gamma_\Delta$ intersects $\gamma_\Lambda$, observe that for any choice of lifts $\gamma_v$, $\gamma_u$ of curves in $v$ and $u$ there is $\Delta \in W_v$ and $\Lambda \in W_u$ so that $\gamma_v = \gamma_\Delta$ and $\gamma_u = \gamma_\Lambda$. Since $v$ and $u$ fill $S$, we can find $\Delta$ and $\Lambda$ so that $\gamma_\Delta$ and $\gamma_\Lambda$ intersect.

We now set out some useful notation in parallel to \cite{HRSS}. Let $$\T = \bigsqcup_{v\in \B^{0}} T_v.$$ 
There is a simplical map $\nu \colon X \to \T$ that is the identity on $\T$ and has $\nu(\A_x) = x$ for each $x \in \T^{0}$. Note, that $X_v$ is equal to $\nu^{-1}(T_v)$ for each $v \in \B^{0}$. Let $\A$  denote the vertices of $X$ that are not in $\T$. That is, $\A$ is the disjoint union of all the $\A_x$ for $x \in \T$.  

Since the $X_v$ are precisely the connected components of $X$, every simplex of $X$ is also a simplex of one of the $X_v$. As described in \cite[Corollary 6.1]{HRSS}, there are only eight combinatorial types of simplices of $X_v$ and hence of $X$. We call simplices of the form $\{x,\nu(x)\}$ for some $x \in \A$ the \emph{QT--type} simplices and simplices of the form $\{x,\nu(x),t\}$ for $x \in\A$ and $t \in \T^{0}$ the \emph{QL--type} simplices. These are the ``Type 7'' and ``Type 8'' simplices of \cite{HRSS}. We choose the new names because $\C(\Delta)$ for these simplices will be quasi-trees and quasi-lines respectively; see Proposition \ref{prop:import_CHHS_facts}. The natural $\Gamma_G$ action on $\T$ extends to a simplicial action on $X$ and $\nu$ is $\Gamma_G$--equivariant.

\begin{lemma}\label{lem:action_on_(x,W)}
\
\begin{itemize}
    \item  The action of $\Gamma_G$ on $X$ has finitely many orbits of links of simplices. 
    \item The action of $\Gamma_G$ on $X$ induces a metrically proper and cobounded action on $W$. 
\end{itemize}

\end{lemma}

\begin{proof} 
The fact that the action of $\Gamma_G$ on $X$ induces an action by graph automorphism on $W$ is immediate from the construction.

Let $\V$ be the finite set of multicurves in $\B$ that are stabilized an element of $\cH$.   By \cite[Lemma 6.7]{HRSS}, we have that $\Gamma_v$ acts with finitely many orbits of links of simplices on $X_v$  for each $v \in \V$. Now if $u \in \B^{0}$ there exists $g\in \Gamma_G$ and $v \in \V$ so that $gT_v = T_u$. Thus  $g\Gamma_vg^{-1}$ will act on $X_u$ with the same number orbits of  links of simplices as $\Gamma_v$ acting on $X_v$. Since $\V$ is finite, there are only finitely many $\Gamma_G$--orbits of links of simplices in $X$. 

We now turn our attention to the induced action on $W$. For each $v \in \V$, let $P_v\colon W_v \to \Gamma_v$ be a $\Gamma_v$--invariant quasi-isometry.
 If $\Lambda$ is a vertex of $W_u$ for some $u \in \B^{0}$,  there is $g \in \Gamma_G$, $v \in \V$, and $\Delta \in W_v^{0}$ so that $g\Delta = \Lambda$. Define $P(\Lambda)$ to be $g P_v(\Delta)$, which is a subset of the coset $g\Gamma_v$.
 Since there are only finitely many curves in $\V$ and  the diameter of $P_v(\Delta)$ does not depend the specific vertex $\Delta$, there is a uniform bound on $\diam(P(\Lambda))$ in $\Gamma_G$  (this uses the fact that each $\Gamma_v$ coarsely Lipschitz embeds into $\Gamma_G$).

 Now, given a bounded diameter set $K_W$ in $W$, let $K_G$ be the set of elements of $\Gamma_G$: $\bigcup_{\Delta \in K_W} P(\Delta)$. The previous paragraph implies $K_G$ is a bounded diameter subset of $\Gamma_G$. Since the coarse stabilizer of $K_W$ is a subset of the coarse stabilizer of $K_G$, the fact that $\Gamma_G$ ask metrically properly on itself implies it acts metrically properly on $W$ as well. 
\end{proof}

Lemma \ref{lem:action_on_(x,W)} means that we can verify that $\Gamma_G$ is a hierarchically hyperbolic group by showing that $(X,W)$ is a combinatorial HHS. Because $X$ and $W$ are created by combining the known combinatorial HHS structures $(X_v,W_v)$, many of the requirements of Definition \ref{defn:CHHS} follow directly from the work on $(X_v,W_v)$ in \cite{HRSS}; we collect these in the following proposition.

\begin{proposition}\label{prop:import_CHHS_facts}
Items \eqref{CHHS:finite_complexity},\eqref{CHHS:comb_nesting_condition}, and \eqref{CHHS:C=C_0_condition} of Definition \ref{defn:CHHS} hold for  $(X,W)$. Moreover, if $\Delta$ is a simplex of $X$, then we have one of the following.
 
    \begin{enumerate}
\item (QT--type simplex) If $\Delta=\{s,\nu(s)\}$, for $s\in \A$, then,  $\C(\Delta)$ is uniformly quasi-isometric to a tree and $Sat(\Delta)$ is $\nu^{-1}(\nu(s))$. 
    
\item (QL--type simplex)  if $\Delta=\{s,\nu(s),\nu(t)\}$, for $s,t\in \A$ adjacent, then   $\C(\Delta)$ is uniformly quasi-isometric to a line and  $Sat(\Delta)$ is $\{\nu(t)\}\cup \bigcup_q \{q,\nu(q)\}$, where the union is taken over all $q\in \A$ adjacent to $t$. 

\item (Bounded simplex) If $\Delta$ is not covered by the above two cases, then $\diam(\C(\Delta)) \leq 2$.
  \end{enumerate}
\end{proposition}

\begin{proof}
Recall that there is a uniform $\delta\geq 0$ so that for all $v \in \B^{0}$, $(X_v,W_v)$ is a $\delta$--combinatorial HHS.

By construction, every simplex $\Delta$ of $X$ is contained in $X_v$ for some $v \in \B^{0}$. Thus the statement on the bounded simplices and the saturations of the QT and QL simplices follow from \cite[Lemma 6.2]{HRSS}. Similarly, Items \eqref{CHHS:finite_complexity} and \eqref{CHHS:comb_nesting_condition} of Definition \ref{defn:CHHS} follow from \cite[Lemmas 6.4 and 6.6]{HRSS}.

Now if $x,y \in X^{0}$ are joined by a $W$--edge of $X^{+W}$, then either \begin{itemize}
    \item $x,y \in X_v$ for some $v \in \B^{0}$ and $x,y$ are joined by a $W_v$ edge of $X_v^{+W_v}$; or
    \item $x \in X_v$ and $y \in X_u$ for $v \neq u$ and hence $x$ and $y$ are not contained in $\lk(\Delta)$ for any simplex $\Delta$ of $X$.
\end{itemize}
This implies that Item \eqref{CHHS:C=C_0_condition} of Definition \ref{defn:CHHS} follows from  \cite[Lemma 6.5]{HRSS} and that each $\C(\Delta)$ for $(X,W)$ is also a $\C(\Delta)$ for  $(X_v,W_v)$ for some $v \in \B^{0}$. Thus \cite[Proposition 6.8]{HRSS}, says the QT--type simplices have $\C(\Delta)$ a quasi-tree and the QL--type simplices have $\C(\Delta)$ a quasi-line.
\end{proof}

The remaining facts to verify that $(X,W)$ is a combinatorial HHS are that  $X^{+W}$ is hyperbolic and that the unbounded $\C(\Delta)$'s quasi-isometrically embedded in the complements of their saturations. The latter will be the content of the next subsection. The hyperbolicity of $X^{+W}$ is established by showing is it quasi-isometric to the tree bundle constructed in Section \ref{subsec_E_for_RGF}.

\begin{proposition}
\label{prop:XW_hyp}
    $X^{+W}$ is hyperbolic. In particular, item \eqref{CHHS:hyp_X} of Definition  \ref{defn:CHHS} holds.
\end{proposition}

\begin{proof}
    We argue that $X^{+W}$ is quasi-isometric to the graph $\mathcal E_1$ described in Subsection \ref{subsec_E_for_RGF}, which is hyperbolic by Theorem \ref{thm:combination_applies} and Theorem \ref{thrm:combination_theorem}. First of all, $X^{+W}$ is quasi-isometric to the graph obtained collapsing each cone in $X$ to a single vertex. This way we obtain a new graph $\E'$, with the same vertex set as $\T$, which in turn is exactly the same as the vertex set of $\mathcal E_1$. Both $\E'$ and $\mathcal E_1$ are connected graphs with finitely many $\Gamma_G$--orbits of edges (see Lemma \ref{lem:E_1-orbits}), and therefore they are quasi-isometric by the version of the Schwartz-Milnor Lemma given by \cite[Theorem 5.1]{CharneyCrisp}.

    Links of non-empty, non-maximal simplices are hyperbolic by Proposition \ref{prop:import_CHHS_facts}, so this was the last remaining fact to check item \eqref{CHHS:hyp_X} of Definition \ref{defn:CHHS}.
\end{proof}

\subsection{Quasi-isometric embedding of unbounded links}
The last remaining axiom of Definition \ref{defn:CHHS} to check is Item \eqref{CHHS:geom_link_condition}, that $\C(\Delta)$ quasi-isometrically embeds in $Y_\Delta$  for non-empty, non-maximal simplices. First, we will need a preliminary lemma on a version of subsurface projections.

For every vertex $v \in \B^0$, let $H_v < G$ be the stabilizer of $v$, and as usual, identify $v$ with the associated multicurve in $\C(S)$ so that $H_v$ is  generated by a multitwists along the curves in $v$.

Given $v\in \B^0$, and a multicurve $\alpha$, define $\Pi_v(\alpha)$ to be the union of all arcs with endpoints in $v$ that are subarcs of $\alpha$ intersecting $v$ in exactly 4 points, where we put $\alpha$ in minimal position with $v$ and we consider the arcs $\alpha_0$ up to homotopy of maps $(\alpha_0,\partial \alpha_0) \to (S,v)$.  For brevity, we refer to this homotopy as {\em homotopy relative to $v$}.

We consider the set $\Pi^2_v$ defined to be
\[ \Pi^2_v = \bigcup_{\stackrel{u,w\in \B^0-\{v\},}{d_{\B}(u,w) = 1}}\Pi_v(u)\times \Pi_v(w). \]

\begin{lemma} \label{lem:finitely many long arcs}
 There exists $C > 0$ so that for all $v \in \B^{0}$ with, the set $\Pi^2_v$ contains at most $C$ pairs of homotopy classes of arcs up to the (diagonal) action of $H_v$.
\end{lemma}

Before proving this lemma, we prove a stronger statement for the usual subsurface projections. Given distinct vertices $u,v \in \B^0$, and a component $Y \subset S -v$, the arc-and-curve projection of $u$ to $Y$, $\pi_Y(u)$, is a union of arcs cutting $Y$ into disks, since $u$ and $v$ fill $S$.  

\begin{lemma} \label{lem:finitely many subsurface projections to link}
    There exist $C_0 > 0$ so that for all $v \in \B^0$ and any component $Y \subset S-v$, the set $\pi_Y(\B^0-\{v\})$ consists of at most $C_0$ isotopy classes.
\end{lemma}
\begin{proof}
    In \cite[Proposition 4.3]{Udall}, Udall showed there is a constant $K$ so that  $\diam( \pi_Z(\B^0-\{v\})) \leq K$ for any subsurface $Z$ that is not an annulus whose core curve is a component of some $u \in \B^0$.  If $\pi_Y(\B^0-\{v\})$ contained infinitely many arcs, then there would be two for which the distance in a subsurface $Z \subset Y$ would have to be larger than $K$ (cf. \cite[Corollary D]{ChoiRafi}), contradicting Udall's result.  There are only finitely many $G$--orbits of components $Y$ of complements of multicurves $v$ in $\B^0$, so we may take $C_0$ to be the maximum of $\pi_Y(\B^0-\{v\})$ over $G$--orbit representatives of such subsurfaces $Y$.
\end{proof}

\begin{proof}[Proof of Lemma~\ref{lem:finitely many long arcs}]
Since there are only finitely many $G$--orbits of vertices in $\B^0$, it suffices to prove the lemma for a fixed $v \in \B^0$.  Let $v_1,\ldots,v_n$ be the components of $v$ and $\pi_{v_i}$ and $d_{v_i}$ denote the projection and distance in the arc complex of the annulus with core $v_i$.

Given any $u \in \B^0-\{v\}$, observe that every arc in $\Pi_v(u)$ is obtained by gluing together three arcs of $\pi_{S-v}(u)$ (which is the union of the arcs of in $\pi_Y(u)$ over the components $Y$ of $S-v$).
Only arcs of $\pi_{S-v}(u)$ that have endpoints on the same component of $v$ can be glued together.  Not all such pairs get glued to form subarcs of arcs in $\Pi_v(u)$.  

Given a triple of arcs that {\em can} be glued to produce an arc of $\Pi_v(u)$ for some $u$, we can consider {\em all} possible ways that they can be glued together.  Any two of these differ by a product of Dehn twists in the components of $v$ on which the glued endpoints lie.  If both endpoint-pairs that are to be glued are on the same component of $v$, then there there is only one Dehn twist ambiguity.
Furthermore, if $u,w \in \B^0$ are two vertices such that arcs $\alpha_u \in \Pi_v(u)$ and $\alpha_w \in \Pi_v(w)$ differ by Dehn twists in components $v_i,v_j$ of $v$ as just described, then the number of Dehn twists needed to obtain $\alpha_w$ from $\alpha_u$ is is bounded by a (linear) function of $d_{v_i}(u,w)$ and $d_{v_j}(u,w)$.  

Next, we prove that $\Pi_v(\B^0-\{v\})$ contains only finitely many $H_v$--orbits.  Let $w_1,\ldots,w_k$ be representatives of the finitely many $H_v$--orbits of vertices in the link of $v$ in $\B$.  Let $W \subset \B^0- \{v\}$ be the set of vertices for which some geodesic to $v$ passes through one of $w_1,\ldots, w_k$.  By Lemma~\ref{lem:bounded projections in B}, for each component $v_i$ of $v$, we see that $\diam(\pi_{v_i}(W))$ is bounded.   Now observe that all arcs of $\Pi_v(W)$ are constructed from a triple of arcs of the set of arcs in $\pi_{S-v}(W)$.  There are at most $C_0$ arcs in this set by Lemma~\ref{lem:finitely many subsurface projections to link}.  The previous paragraph, and the bound on $\diam(\pi_{v_i}(W))$ implies a bound on the number of possible choices we can make in gluing the arcs, and hence we deduce a bound on $\Pi_v(W)$.  Since $\Pi_v(W)$ contains the $H_v$--orbit representative of any element of $\Pi_v(\B^0-\{v\})$, it follows that $\Pi_v(\B^0-\{v\})$ contains only finitely many $H_v$--orbits.

If $u,w \in \B^0-\{v\}$ are adjacent, we can apply an element of $H_v$ and assume that $w$ is in $W$, and that (the image of) $u$ is distance $1$ from $W$ in $\B$.  Repeating the argument in the previous paragraph with $W$ replaced by its $1$--neighborhood proves finiteness of pairs $\Pi_v(u) \times \Pi_v(w)$ where $w \in W$, and hence finiteness of the number of $H_v$--orbits of elements of $\Pi_v^2$.
\end{proof}

\begin{proposition}
\label{prop:links_qie}
Let $z \in \A$ and $s\in \T^{0}$ so that $s$ and $ z$ are joined by an edge of $X$.
    \begin{enumerate}
        \item The QT--type simplex $\Delta=\{z,\nu(z)\}$ has $\C(\Delta)$ uniformly quasi-isometrically embedded in $Y_\Delta$.
        \item The QL--type simplex $\Delta=\{s,z,\nu(z)\}$ has $\C(\Delta)$ uniformly quasi-isometrically embedded in $Y_\Delta$.
    \end{enumerate}
    In particular, Item \eqref{CHHS:geom_link_condition} of Definition \ref{defn:CHHS} holds.
\end{proposition}

\begin{proof} In Subsection \ref{subsec:cc fibrations}, we defined an open convex region $\Phi_{v}^{-1}(t)$ for each $v \in \B^{0}$ and each $t \in T_v^{0}$. Thus for each $t\in\T$ we can define $C_t$ to be the closure of $\Phi_{v}^{-1}(t)$ in $\mathbb{H}^2$ where $v$ is the vertex of $\B$ where $t \in T_{v}^0$.  We extend this to all vertices of $X$ by setting $C_x=C_{\nu(x)}$ for each vertex  $x \in \A$. By slight abuse of notation, we use $\Phi_v(C_x)$ to denote the image under $\Phi_v$ of the intersection of $C_x$ with the domain of $\Phi_v$ (which is $\mathbb{H}^2$ minus the lifts of $v$).

In each case of the proposition,  we prove the quasi-isometric embedding of $\C(\Delta)$ by defining a course retraction $f \colon Y_{\Delta} \to \C(\Delta)$. 

\medskip
\emph{Case 1: The QT--type simplex.}  Let $\Delta = \{z ,\nu(z)\}$ be a QT--type simplex of $X$. Let $q = \nu(z)$ and $w$ be the vertex of $\B$ so that $q,s\in T_w$. Let $\rho\colon T_w-\{q\}\to \lk_{T_w}(q)$ be the closest point projection. For a vertex $x \in Y_{\Delta}$ with  $x \in X_w$, define $f(x)=\rho(\nu(x))$.
    
    For $x \in X^{0} - X_w$, let $\Theta_x$ be the set of vertices $t \in T_w$ so that $C_t \cap C_x$ is non-empty.
    That is, $\Theta_x = \Phi_{w}(C_x)$. We define $f(x) =\rho(\Theta_x-\{q\})$.

    For later purposes, observe  that for each $x \in \lk_{T_w}(q)$, $C_x$ and $C_q$  intersect in a unique boundary component of $C_q$. Thus there is a bijection between the boundary components of $C_q$ and the vertices of $\lk_{T_w}(q)$.

    We need to prove that the diameter of $f(x)$ is uniformly bounded for all $x \in X^{0}$. Consider $C_x\cap C_q$, which is a polygon in $\widetilde S\cong \mathbb H^2$. Since  $C_x\cap C_q$ is a lift to $\mathbb H^2$ of the intersection of two subsurfaces, the number of sides of $C_x\cap C_q$ is bounded in terms of the ambient surface $S$.

Removing $C_x\cap C_q$ from $C_x$, we obtain connected components each of which corresponds to a  connected component of $\Theta_x-\{q\}$. The closest-point projection $\rho$ is constant on these components, and therefore we are left to prove the following. 

\begin{claim}
\label{claim:cross}
    Given a side $\sigma$ of $C_x\cap C_q$ connecting boundary components of $C_q$, the vertices of $T_w$ corresponding to these components lie within uniformly bounded distance in $\C(\Delta)$.
\end{claim}

\begin{proof}
    Let $V$ be the image of $C_q$ under the covering map $\mathbb H^2\to S$, and suppose that $x$ lies in the tree $T_{u}$.  Then the side $\sigma$ is a lift to $\mathbb H^2$ of an arc of $\pi_V(u)$. By Lemma \ref{lem:finitely many subsurface projections to link}, there are only finitely many $\stab_\Gamma(s)$--orbits of lifts of projection arcs in $V$, each orbit giving pairs of points at some finite distance in $\C(\Delta)$. Since there are also finitely many orbits of vertices in $\T$, there is in fact a uniform bound for any $x \in X$.
\end{proof}

We now need to prove that $f$ is coarsely Lipschitz. There are two cases depending if the adjacent vertices are in the same or different $X_v$'s. For adjacent vertices $x,x'\in Y_\Delta$ that lie in the same $X_{v}$, the intersection of the corresponding convex regions $C_x \cap C_{x'}$ contains a bi-infinite line. Since $w$ and $v$ fill $S$, this bi-infinite line cannot be contained in $C_q$. Thus the image of this line under  $\Phi_{w}$ contains a vertex of $\Theta_x \cap \Theta_{x'}$ that is not the vertex $q$. Since $f$ is coarsely well-defined, the fact that the subtrees $\Theta_x$ and $\Theta_{x'}$ share a vertex outside $q$ implies  that $f(x)\cup f(x')$ has uniformly bounded diameter in $\C(\Delta)$. 

Now suppose  that $x,x' \in Y_\Delta^{0}$ are adjacent with $x \in X_v$ and $x' \in X_u$ for $u\neq v$.  If $C_x \cap C_{x'}$ contains a point outside $C_q$, we can again reduce to $f$ being coarsely well-defined. If not, there are two intersection arcs $\alpha_1,\alpha_2$ of boundary components of $C_x,C_{x'}$ with $C_q$ such that $\alpha_1\cap\alpha_2$ is a point of $C_s$. This pair of arcs is the lift of a pair of arcs on $S$ which, by Lemma \ref{lem:finitely many long arcs}, comes from a finite set of possible pairs up to $H_{p(s)}$ ($p$ is the map $p \colon \E \to \B$). Therefore, there are only finitely many $\stab_\Gamma(s)$--orbits of possible pairs $\alpha_1,\alpha_2$, each orbit resulting in some finite distance between $f(x),f(x')$ in $\C(\Delta)$. Thus , there is a uniform bound over all pairs $x,x'$ on the distance between $f(x),f(x')$, as required.

\medskip

\emph{Case 2: The QL--type simplex.} Let $\Delta = \{s,z,\nu(z)\}$ be a QL--type simplex of $X$. Our candidate retraction  is constructed similarly to the previous case, except we replace $q$ with $\st_{T_w}(s)$ and we replace the closest-point projection in $T_w$ with a retraction used in \cite{HRSS}.  The proof of \cite[Lemma 6.15]{HRSS} gives a coarse retraction $\rho \colon X_w -\st_{T_w}(s)\to \mathcal C(\Delta)$. By construction,  if $\nu(x)$ is more than 2 away from $q$ in $T_w$, then $\rho(x)$ is the unique vertex on the geodesic from $\nu(x)$ to $s$ that is distance 2 from $s$; see the paragraph above \cite[Lemma 6.17]{HRSS}. Similarly to the first case, for $x \in Y_{\Delta}^{0}$ with $\nu(x)\in T_w$, define $f(x)=\rho(\nu(x))$, while setting $f(x) =\rho(\Theta_x-\st(s))$ if $\nu(x)\notin T_w$.

We first show that $f$ is coarsely well-defined for $x \in Y_{\Delta}^{0}$. If $x \in X_w$, this is immediate from $\rho$ be coarsely well defined, thus we can assume $x \in X_v$ for some $v \neq w$. Let $D_s$ to be the union of $C_{s}$ with all the $C_{x'}$ for every vertex $x' \in \lk_{T_w}(s)$. Consider the intersection $D_s\cap C_x$. Since $C_x \cap C_s$ has uniformly finitely many sides, $D_s \cap C_x$ is the union of $C_s \cap C_x$ with a uniformly finite number of polygons of the form $C_x \cap C_{x'}$ for some $x ' \in \lk_{T_w}(\Delta)$. Thus $D_s \cap C_x$ is a (possibly non-convex) polygon with a uniformly finite number of sides. As in the first case, it now suffices to consider an arbitrary side $\sigma$ of $D_s\cap C_x$ connecting boundary components of $D_s$, and argue that the endpoints connect convex regions $C_y,C_{y'}$ with $\rho(y)$ uniformly close to $\rho(y')$. 

 As in the previous case, Lemma \ref{lem:finitely many long arcs} ensures there are finitely many $\stab_\Gamma(s)$--orbits of such arcs $\sigma$ up to homotopy keeping endpoints on a boundary component of a convex region. Therefore, if $y,y' \in X_w - \st_{T_w}(s)$ are the vertices so the end points of $\sigma$ are in $C_y$ and $C_{y'}$, then $y$ and $y'$ are uniformly bounded distance in $Y_\Delta$. Since  $X_w - \st_{T_w}(s)$ is where $\rho$ is defined and coarsely Lipschitz, $\rho(y)$ and $\rho(y')$ are uniformly bounded as desired.

The argument for $f$ being coarsely Lipschitz is now nearly identical to the first case, replacing $q$ with $\st_{T_w}(s)$, and $C_q$ with $D_s$.   
\end{proof}

We have now proved the more precise version of Theorem \ref{thm:intro_PGF_hyp}.

\begin{theorem}\label{thm:HHG detailed}
Let $G <\Mod(S)$ be a parabolically geometrically finite group with cyclic peripheral subgroups and $\Gamma_G$ the corresponding surface group extension.
\begin{enumerate}

    \item     The pair $(X,W)$ from Section \ref{sec:cyclic PGF HHG} is a combinatorial HHS.
    \item  The action of $\Gamma_G$  on $X$ has finitely many orbits of links of simplices, and the induced action of $\Gamma_G$ on $W$ is metrically proper and cobounded.
    
\end{enumerate}
 As a consequence, $\Gamma_G$ is a hierarchically hyperbolic group.
    
\end{theorem}

\begin{proof}
    Item (1) is the combination of Proposition \ref{prop:import_CHHS_facts}, Proposition \ref{prop:XW_hyp}, and Proposition \ref{prop:links_qie}.
    Item (2) is Lemma \ref{lem:action_on_(x,W)}.  Theorem 1.18 plus Remark 1.19 of \cite{BHMS} give  that these two items imply $\Gamma_G$ is a hierarchically hyperbolic group.
\end{proof}

\bibliographystyle{alpha}
\bibliography{biblio}

\end{document}